\documentclass[11pt,reqno]{amsart}

\usepackage{setspace}
\usepackage{color}
\usepackage[all]{xy}
\usepackage{enumitem}
\usepackage{amsmath,amssymb}
\usepackage{amsthm}
\usepackage{mathabx}
\usepackage{tikz}
\usepackage{amsfonts}
\usepackage{pst-node}
\usepackage{tikz-cd} 
\usepackage{graphicx}
 \usepackage{pifont}
 \usepackage[T1]{fontenc}

 \calclayout

\newtheorem{theorem}{Theorem}[section]
\newtheorem{proposition}[theorem]{Proposition}
\newtheorem{corollary}[theorem]{Corollary}
\newtheorem{lemma}[theorem]{Lemma}
\theoremstyle{definition}
\newtheorem{definition}[theorem]{Definition}
\newtheorem{example}[theorem]{Example}
\newtheorem{conjecture}[theorem]{Conjecture}
\newtheorem{remark}[theorem]{Remark}

\makeatletter
\@addtoreset{equation}{section}
\makeatother

\newcommand{\bbZ}{{\mathbb Z}} 
\newcommand{\bbR}{{\mathbb R}} 

\newcommand{\bbF}{{\mathbb F}} 

\newcommand{\bA}{{\mathbf A}}

\newcommand{\bC}{{\mathbf C}}  
\newcommand{\bD}{{\mathbf D}} 
\newcommand{\bE}{{\mathbf E}} 

\newcommand{\bG}{{\mathbf G}}
\newcommand{\bH}{{\mathbf H}}

\newcommand{\bO}{{\mathbf O}}

\newcommand{\bX}{{\mathbf X}}

\newcommand{\bZ}{{\mathbf Z}}

\newcommand{\cA}{{\mathcal A}}
\newcommand{\cB}{{\mathcal B}}
\newcommand{\cC}{{\mathcal C}}
\newcommand{\cD}{{\mathcal D}}
\newcommand{\cE}{{\mathcal E}}
\newcommand{\cF}{{\mathcal F}}
\newcommand{\cG}{{\mathcal G}}
\newcommand{\cH}{{\mathcal H}}

\newcommand{\cL}{{\mathcal L}}
\newcommand{\cM}{{\mathcal M}}
\newcommand{\cN}{{\mathcal N}}
\newcommand{\cO}{{\mathcal O}}

\newcommand{\cT}{{\mathcal T}}

\newcommand{\Aut}{\mathrm{Aut}}
\newcommand{\Out}{\mathrm{Out}}
\newcommand{\Inn}{\mathrm{Inn}}
\renewcommand{\hom}{\mathrm{Hom}}
\newcommand{\Ind}{\mathrm{Ind}}
\newcommand{\Mor}{\mathrm{Mor}}
\newcommand{\Tot}{\mathrm{Tot}}
\newcommand{\Ext}{\mathrm{Ext}}

\newcommand{\Res}{\mathrm{Res}}
\newcommand{\Ob}{\mathrm{Ob}}
\newcommand{\id}{\mathrm{id}}
\DeclareMathOperator*{\hocolim}{hocolim}
\DeclareMathOperator*{\colim}{colim}
\def\maprt#1{\smash{\,\mathop{\longrightarrow}\limits^{#1}\,}}
\def\maplf#1{\smash{\,\mathop{\longleftarrow}\limits^{#1}\,}}
\newcommand{\leftexp}[2]{{\vphantom{#2}}^{ #1}{\hskip-1pt#2}}

\begin{document}
\title{Higher limits over the fusion orbit category}
 
\author{Erg{\" u}n Yal{\c c}{\i}n}
\address{Department of Mathematics, Bilkent University, 06800 
Bilkent, Ankara, Turkey}

\email{yalcine@fen.bilkent.edu.tr}

\begin{abstract}  
The fusion orbit category $\overline \cF _{\cC} (G)$ of a discrete group $G$ 
over a collection $\cC$ is the category whose objects are the subgroups $H$ in $\cC$, 
and whose morphisms $H \to K$ are given by the $G$-maps $G/H \to G/K$ modulo 
the action of the centralizer group $C_G(H)$. We show that the higher limits over 
$\overline \cF_{\cC} (G)$ can be computed using the hypercohomology 
spectral sequences coming from the  Dwyer $G$-spaces for centralizer and normalizer 
decompositions for $G$.  

If $G$ is the discrete group realizing a saturated fusion system $\cF$, then these 
hypercohomology spectral sequences give two spectral sequences that converge 
to the cohomology  of the centric orbit category $\cO ^c (\cF)$. This allows us to  
apply our results to the sharpness problem for the subgroup decomposition of 
a $p$-local finite group. We prove that the subgroup decomposition for every 
$p$-local finite group is sharp (over $\cF$-centric subgroups) if it is sharp 
for every $p$-local finite group with nontrivial center.  We also show that  
for every $p$-local finite group $(S, \cF, \cL)$, the subgroup decomposition 
is sharp if and only if the normalizer decomposition is sharp. 
\end{abstract}

\thanks{2020 {\it Mathematics Subject Classification.} Primary: 20J06; Secondary: 18G10, 55R40, 20D20}

\keywords{Fusion systems, higher limits, orbit category, $p$-local finite group, cohomology of small categories}

\maketitle


\section{Introduction}\label{sect:Intro}    

A \emph{saturated fusion system $\cF$} over a finite $p$-group $S$ is a category whose objects 
are subgroups of $S$ and whose morphisms are injective group homomorphisms satisfying 
certain axioms (see Definitions \ref{def:FusionSystem} and  \ref{def:Saturated}). The main 
example of a saturated fusion system is the fusion system of a finite group defined 
over one of its Sylow $p$-subgroups whose morphisms are the group homomorphisms induced 
by conjugations in $G$.  However, fusion systems that do not come from finite groups also exist.  

A subgroup $P \leq S$ is called \emph{$\cF$-centric} if $C_S(Q) \leq Q$ for every subgroup $Q$ 
isomorphic to $P$ in $\cF$. We denote by $\cF^c$ the full subcategory  of $\cF$ whose objects 
are the $\cF$-centric subgroups of $S$. The \emph{orbit category} $\cO (\cF )$ of a fusion system 
$\cF$ is the category whose objects are the subgroups in $S$ and whose morphisms are given by 
$$\Mor _{\cO (\cF) } (P, Q) :=\Inn (Q) \backslash \Mor _{\cF} (P, Q ).$$
The orbit category over the collection of $\cF$-centric subgroups is denoted by $\cO ^c (\cF)$ 
and is called the \emph{centric orbit category of $\cF$}.

A \emph{$p$-local finite group} is a triple $(S, \cF, \cL)$ where $S$ is a finite $p$-group, $\cF$ 
is a saturated fusion system over $S$, and $\cL$ is a centric linking system associated to $\cF$. 
A centric linking system $\cL$ associated to a fusion system $\cF$ is a category whose objects are the 
$\cF$-centric subgroups of $S$, and the morphisms are defined such that there is a quotient functor 
$\pi: \cL \to \cF^c$. As part of the structure of $\cL$, there exist also distinguished monomorphisms
$\delta _P : P \to \Aut _{\cL } (P)$ defined for every $P \in \cF^c$ satisfying certain properties 
(see \cite[Def 1.7]{BLO2}). In Section \ref{sect:NormalizerDec} we give a more recent definition of 
a linking system using the transporter category $\cT _S ^{\cF ^c}$ (see Definition \ref{def:LinkingSystem}).
It has been proved by Chermak \cite{Chermak} that for every saturated fusion system $\cF$, there exists 
a unique centric linking system $\cL$ associated to $\cF$ (see \cite[\S III.4]{AshOliver} for details). 

The classifying space of a $p$-local finite group $(S, \cF, \cL)$ is defined to be the  Bousfield-Kan
$p$-completion of the geometric realization $|\cL|_p ^{\wedge}$ of the category $\cL $. 
There is a subgroup homology decomposition for a $p$-local finite group $(S, \cF, \cL)$, 
similar to the subgroup decomposition for finite groups, introduced 
by Broto, Levi, and Oliver \cite[Prop 2.2]{BLO2}. They showed that for every $p$-local finite group 
$(S, \cF, \cL)$, there is a homotopy equivalence 
\begin{equation}\label{eqn:IntroSubgroupDec} |\cL| \simeq  \hocolim _{\cO ^c (\cF) } \widetilde {B} 
\end{equation}
where $\widetilde B : \cO ^c (\cF) \to \mathrm{Top}$ is a functor such that ${\widetilde B} (P)$ is 
homotopy equivalent to the classifying space $BP$ for every $P \in \cF ^c$. The Bousfield-Kan
 spectral sequence associated to the above homology decomposition gives a spectral sequence
$$E_2 ^{s, t} =  \underset{ \cO ^c (\cF )}{\lim {}^s}\  H^t (- ; \bbF _p ) \Rightarrow H^{s+t} ( |\cL| ; \bbF _p )$$ 
where $H^t (-; \bbF_p)$ denotes the contravariant functor $\cO^c (\cF)\to \bbF_p$-Mod that sends 
an $\cF$-centric subgroup $P \leq S$ to its group cohomology $H^t (P; \bbF_p)$ in $\bbF_p$-coefficients.

\begin{definition}\label{def:SharpSub} 
The subgroup decomposition for $(S, \cF, \cL)$  is said to be  \emph{sharp} if the associated Bousfield-Kan 
spectral sequence collapses at the $E_2$-page to the vertical axis, i.e., if $E_2 ^{s, t} =0$ for all $s>0$ and 
for all $t\geq 0$. 
\end{definition}

Note that if the subgroup decomposition is sharp, then the edge homomorphism
$$H^* ( |\cL|; \bbF_p) \to E_2 ^{0, *}= \lim _{P \in \cO ^c (\cF)} H^* (P ; \bbF_p)$$ is an isomorphism. 
The limit term on the right-hand side is called the \emph{cohomology of the fusion system $\cF$} 
and it is denoted by $H^* (\cF ; \bbF_p)$. When this isomorphism holds, we say that the 
Cartan-Eilenberg theorem holds for $(S, \cF, \cL)$.  

Diaz and Park \cite[Thm B]{DiazPark}  proved that the subgroup decomposition is sharp if $\cF$ 
is a fusion system realized by a finite group.  Broto, Levi, and Oliver \cite[Thm 5.8]{BLO2} proved 
that the Cartan-Eilenberg theorem holds for every saturated fusion system. These two results 
suggest that the subgroups decomposition is sharp for every saturated  fusion system.
This is stated as a conjecture  by Diaz and Park \cite{DiazPark} and as a question 
by Ashbacher and Oliver \cite[Ques 7.12]{AshOliver}.   
 
\begin{conjecture}\label{conj:Sharpness} Let $(S, \cF, \cL)$ be a $p$-local finite group. Then for every 
$n\geq 0$ and every $i \geq 1$, $$ \underset{ \cO ^c (\cF )} {\lim {}^i} \ H^n (- ; \bbF _p )=0.$$
\end{conjecture}

This conjecture is the main motivation for us to study the higher limits over the centric orbit category 
$\cO^c (\cF)$. To compute the higher limits over $\cO ^c (\cF)$, we propose to use an ambient discrete 
group $G$ that realizes the fusion system $\cF$. By theorems of Leary and Stancu \cite{LearyStancu} 
and Robinson \cite{Robinson}, for every (saturated) fusion system $\cF$, there is a discrete group $G$ 
(possibly infinite) with a finite Sylow $p$-subgroup $S$ such that $\cF \cong \cF_S(G)$. For a discrete 
group $G$ and a collection $\cC$ of subgroups of $G$ (always assumed to be closed under conjugation), 
the following categories are defined:

\begin{enumerate} 
\item The \emph{orbit category} $\cO_{\cC} (G)$ of $G$ is the category whose objects are subgroups 
$H \in \cC$, and whose morphisms $\Mor _{\cO _{\cC} (G) } (H, K)$ are given by $G$-maps $G/H \to G/K$.  
\item The  \emph{fusion category} $\cF_{\cC} (G)$ of $G$ is the category whose objects are the subgroups 
$H \in \cC$, and whose morphisms $H \to K$ are given by conjugation maps $c_g : H \to K$ for an element 
in $g \in G$. 
\end{enumerate}
For every $H, K \in \cC$, let $N_G(H, K)  :=\{ g \in G \, |\, gHg^{-1} \leq K \}.$  The category whose objects 
are subgroups $H\in \cC$ and whose morphisms from $H$ to $K$ are given by $N_G(H, K)$ is called the 
\emph{transporter category} of $G$ and  is denoted by $\cT_{\cC} (G)$. Both the orbit category and 
the fusion category can be viewed as the quotient category of the transporter category 
(see Section \ref{sect:Orbit} for details).
 
\begin{definition}\label{def:FusionOrbit} The \emph{fusion orbit category} $\overline \cF _{\cC} (G)$ of  
a discrete group $G$ over a collection $\cC$ is the category whose objects are subgroups $H \in \cC$, 
and whose morphisms are given by $$\Mor _{\overline \cF _{\cC} (G) } (H, K) :=K \backslash 
N_G(H, K) /C_G (H)$$  for every $H, K\in \cC$.
\end{definition}

The fusion orbit category $\overline \cF_{\cC} (G)$ is a quotient category of both the orbit category 
$\cO_{\cC} (G)$ and the fusion category $\cF_{\cC} (G)$.  If $G$ is a discrete group realizing 
a saturated fusion system $\cF$ and if we take $\cC$ to be the collection of all $p$-subgroups 
in $G$ that are conjugate to a $\cF$-centric subgroup in $S$, then $\overline \cF _{\cC} (G)$ 
is equivalent to $\cO ^c (\cF)$ as categories (see Lemma \ref{lem:CatEquivalence}). This allows 
us to calculate the higher limits over $\cO^c (\cF)$ as the higher limits over  $\overline \cF_{\cC} (G)$ 
for a discrete group $G$ realizing $\cF$. The idea of using an ambient discrete group for proving 
theorems for abstract fusion systems was also used by Libman in \cite{Libman-WebbConj}.

Let $G$ be a discrete group and $\cC$ be a collection of subgroups of $G$. To compute 
the higher limits of fusion orbit category $\overline \cF_{\cC} (G)$, we consider the hypercohomology 
spectral sequences coming from certain $G$-spaces. Let $X$ be a $G$-CW-complex and 
$R$ be a commutative ring with unity. Associated with $X$, there is a chain complex of 
$R\cO_{\cC} (G)$-modules $C_* (X^? ; R)$ defined by $$H \to C_* (X^H ; R) \quad \quad 
\text{and} \quad \quad  (f: G/ H \to G/K) \to (f^* : C_* (X^K ; R) \to C_* (X^H ; R))$$
for every $H \in \cC$. When the collection $\cC$ is large enough to include all the isotropy 
subgroups of $X$, the complex $C_* (X^?; R)$ is a chain complex of projective $R\cO _{\cC} (G)$-modules. 
The Bredon cohomology of the space $X$ is defined using this chain complex.

\begin{definition}\label{def:BredonCoh} Let $\cO (G)$ denote the orbit category of $G$ 
over all subgroups of $G$, and let $M$ be an $R\cO(G)$-module. The \emph{(ordinary) 
Bredon cohomology of $X$} with coefficients in $M$ is defined by
$$H^* _{\cO (G)} (X^? ; M):= H^* (\hom _{R\cO(G) } (C_* (X^? ; R) , M)).$$
\end{definition}

If the isotropy subgroups of $X$ do not lie in $\cC$, then it is still possible to define the 
Bredon cohomology using hypercohomology. In this case the Bredon cohomology is defined 
by $$H^* _{\cO _{\cC} (G)} (X^?; M) := H^* ( \hom _{R\cO_{\cC} (G) } ( \Tot ^{\oplus} (P_{*,*} ) ; M))$$
where $P_{*,*}$ is a Cartan-Eilenberg resolution of the complex $C_* ( X^? ; R)$ as a chain 
complex of $R \cO_{\cC} (G)$-modules. This definition of the Bredon cohomology is due
to Symonds \cite{Symonds}.

The definitions above can be modified to obtain chain complexes over the fusion orbit category 
and to define a fusion orbit category version of the Bredon cohomology. For a $G$-CW-complex $X$, 
let $C_* (C_G (?)\backslash X^? ; R)$ denote the chain complex of  $\overline \cF_{\cC} (G)$-modules 
defined by $$H \to C_* (C_G(H)\backslash X^H ; R) $$
$$ (f: G/ H \to G/K) \to (f^* : C_* (C_G(K)\backslash X^K ; R) \to C_* (C_G(H) \backslash X^H ; R))$$
for every $H, K \in \cC$. These chain complexes are introduced  by L\" uck  in \cite{Lueck-HomChern} and 
\cite{Lueck-CohChern} to give a more efficient way to compute equivariant Chern 
characters. 

As in the orbit category case, if the isotropy subgroups of $X$ lie in $\cC$, 
then $C_* (C_G (?)\backslash X^? ; R)$ is a chain complex of projective 
$R\overline \cF _{\cC} (G)$-modules. 
When $\cC$ does not include all the isotropy subgroups of $X$, in general
the complex $C_* (C_G (?)\backslash X^? ; R)$ is not a chain complex of projective $R\overline \cF _{\cC} (G)$-modules. 
In this case we define the fusion orbit category version of the Bredon cohomology using hypercohomology.

\begin{definition}\label{def:FusionBredon} Let $P_{*,*}$ be a Cartan-Eilenberg resolution 
of the chain complex  $C_* (C_G(?) \backslash X^? ; R)$ as a chain complex of 
$R\overline \cF_{\cC} (G)$-modules.  The \emph{fusion Bredon cohomology} of 
a $G$-CW-complex $X$ is defined by  $$H^* _{\overline \cF _{\cC} (G)} (X^?; M) 
:= H^* ( \hom _{R\overline \cF _{\cC} (G) } ( \Tot ^{\oplus} (P_{*,*} ) ; M)).$$ 
\end{definition}

When the family $\cC$ includes all isotropy subgroups of $X$, for every $R\overline \cF _{\cC} (G)$-module $M$,
there is an isomorphism
$$H^* _{\overline \cF _{\cC} (G) } (X^?; M) \cong H^* _{\cO _{\cC} (G)} (X^?; \Res _{pr} M)$$
where $pr: \cO _{\cC} (G) \to \overline \cF _{\cC} (G)$ is the projection functor (see Proposition \ref{pro:TwoBredons}).
However in general the fusion Bredon cohomology is not isomorphic to the Bredon cohomology  
(see Example \ref{ex:BredonDifferent}).

There are two spectral sequences converging to the fusion Bredon 
cohomology of $X$ coming from the two different ways of filtering the double complex 
$\hom _{R\overline \cF _{\cC} (G)} (P_{*, *} , M)$
(see Proposition \ref{pro:HyperCohSS}). One of these spectral sequences can be used 
to prove that if $X$ is a $G$-CW-complex 
such that for every $H \in \cC$, the orbit space $C_G(H)\backslash X^H$ is $R$-acyclic, 
then for any $R\overline \cF _{\cC} (G)$-module $M$,  
$$H^* _{\overline \cF _{\cC} (G)} (X^?; M) \cong H^* (\overline \cF _{\cC} (G); M).$$
In this case, the second spectral sequence gives a spectral sequence that converges to the 
$H^* (\overline \cF _{\cC} (G); M)$ (see Theorem \ref{thm:HyperCohSS}).
We apply this spectral sequence to the Dwyer spaces for the homology decompositions of $G$ 
(see Section \ref{sect:DwyerSpaces} for the definitions of Dwyer spaces). We show that 
for the following choices of $X$ and the collection $\cC$, the condition  on $X$ given 
in Theorem \ref{thm:HyperCohSS} holds (see Propositions \ref{pro:CentFixedPoints} 
and  \ref{pro:NormFixedPoints}):
\begin{enumerate}
\item $X=X_{\cE} ^{\alpha}=E\bA_{\cE}$ is the Dwyer space for the centralizer decomposition over 
the collection $\cE$ of all nontrivial elementary abelian $p$-subgroups in $G$, and $\cC$ is 
any collection of nontrivial $p$-subgroups in $G$.
\item $X=X_{\cC} ^{\delta}= |\cC|$ is the Dwyer space for the normalizer decomposition over $\cC$, and
$\cC$ is any collection of nontrivial $p$-subgroups of $G$ closed under taking products (see Definition 
\ref{def:TakingProducts}).
\end{enumerate}
As a consequence, we obtain two spectral sequences that converge to the cohomology 
of fusion orbit category $\overline \cF_{\cC} (G)$ (see Propositions \ref{pro:CentDecOrb}  
and \ref{pro:NormalizerSS}). These spectral sequences can be considered as the centralizer 
and normalizer decompositions for the cohomology of fusion orbit category.   

If we take $G$ to be a discrete group realizing a saturated fusion system $\cF$, then the
hypercohomology spectral sequences that we constructed give two spectral sequences that converge 
to the cohomology  of the centric orbit category $\cO ^c (\cF)$. We now explain these spectral sequences.

Let $\cF$ be a saturated fusion system over $S$, and let $\cF^e$ denote the full subcategory of $\cF$ 
whose objects are the collection of all nontrivial elementary abelian $p$-subgroups of $S$ which are 
fully $\cF$-centralized. For every  $E\in \cF^e$, let $C_{\cF} (E)$ denote the centralizer fusion 
system over $C_S (E)$ as defined in Definition \ref{def:GenNormalizer}.

\begin{theorem}\label{thm:IntroCentDecSS}  Let $\cF$ be a saturated fusion system over $S$, 
$R$ be a commutative ring with unity, and $M$ be an $R \cO ^c (\cF)$-module. 
For every $j\geq 0$, let $\cH ^j _{M, C_\cF} : \cF^e \to R\text{-Mod}$ denote the functor defined 
in Lemma \ref{lem:HCF} such that for every $E\in \cF^e$,
$$\cH ^j _{M, C_\cF }  (E) =H^j (\cO ^c ( C_{\cF} (E) ) ;  \Res ^{\cO ^c (\cF )} _{\cO ^c (C_{\cF} (E) )} M ).$$
Then, there is a spectral sequence 
$$E_2 ^{s,t} = \underset{\cF^e\ }{\lim{}^s} \ \cH ^t _{M, C_\cF} \Rightarrow 
H^{s+t} ( \cO ^c (\cF); M).$$
\end{theorem}

We call this spectral sequence the centralizer decomposition for the cohomology of the centric 
orbit category. We have a similar spectral sequence involving normalizer fusion systems 
that can be considered as  the normalizer decomposition which we now describe. 

Let $\cF$ be a saturated fusion system over $S$ and let ${\overline s}d (\cF ^c)$ denote 
the poset category of $\cF$-conjugacy classes of chains 
$\sigma =(P_0 < P_1 <\cdots < P_n)$ of $\cF$-centric subgroups $P_i $ of $S$. For every 
fully $\cF$-normalized chain $\sigma$, let $N_{\cF} (\sigma)$ denote the normalizer fusion system 
of $\sigma$ as defined in Definition \ref{def:NormalizerFusSys}.

\begin{theorem}\label{thm:IntroNormDecSS}  Let $\cF$ be a saturated fusion system over $S$, 
and $M$ be an $\bbZ_{(p)} \cO (\cF)$-module. For every $j\geq 0$, 
let $\cH ^j _{M, N_\cF} : \overline{s}d( \cF ^c) ^{op} \to \bbZ_{(p)}\text{-Mod}$ denote the functor defined 
in Lemma \ref{lem:HNF} such that for every $[\sigma] \in \overline{s}d(\cF^c) $,  
$$\cH ^j _{M, N_\cF }  ([\sigma]) =H^j (\cO ^c ( N_{\cF} (\sigma) ) ;  \Res ^{\cO (\cF )} _{\cO ^c (N_{\cF} (\sigma))  } M ).$$
Then, there is a spectral sequence 
$$E_2 ^{s,t} = \underset{\overline{s}d (\cF ^c) }{\lim{}^s} \ \cH ^t _{M, N_\cF} \Rightarrow 
H^{s+t} ( \cO ^c (\cF); \Res ^{\cO (\cF)} _{\cO ^c (\cF) } M).$$
\end{theorem}

We apply these spectral sequences to the sharpness problem stated at the beginning. 
A subgroup  $Q\leq S$ is \emph{central} in $\cF$ if $C_{\cF} (Q)=\cF$. 
 The product of all central subgroups in $\cF$ is called the 
center of $\cF$ and denoted by $Z(\cF)$. 
Using the centralizer spectral sequence constructed in Theorem \ref{thm:IntroCentDecSS}, 
we prove the following:

\begin{theorem}\label{thm:IntroCentReduction}
If the subgroup decomposition is sharp for every $p$-local finite group $(S, \cF, \cL)$ 
with $Z(\cF)\neq 1$, then it is sharp for every $p$-local finite group.
\end{theorem}
  
This reduces the sharpness problem for $p$-local finite groups to the ones with nontrivial center. 
The main ingredient for the proof of Theorem \ref{thm:IntroCentReduction} is the sharpness of the centralizer decomposition 
for $p$-local finite groups. The centralizer decomposition for a $p$-local finite group is introduced by  
Broto, Levi, and Oliver in \cite[Thm 2.6]{BLO2} and the proof of the sharpness 
of the centralizer decomposition can be found in the proof of  \cite[Thm 5.8]{BLO2}. 
 
Next we consider the normalizer decomposition of $p$-local finite groups introduced 
by Libman \cite{Libman-NormDec}. Applying the spectral sequence in Theorem 
\ref{thm:IntroNormDecSS},  we show that for every $p$-local finite group $(S, \cF, \cL)$, 
there is an isomorphism $$ \underset{\cO ^c (\cF )}{\lim {}^i}  \ H^n (- ; \bbF _p ) \cong 
\underset{\overline{s}d(\cF^c) } {\lim {}^i} \ H^n ( N_{\cF} (-)  ; \bbF_p)$$
(see Theorem \ref{thm:NormSharp}).
The higher limits on the right are  isomorphic to the $E_2 ^{i,n}$-term in the Bousfield-Kan 
spectral sequence for the normalizer decomposition for $(S, \cF, \cL)$.
As a consequence we obtain the following theorem.

\begin{theorem}\label{thm:IntroNormSharp} For every $p$-local finite group $(S, \cF, \cL)$, 
the subgroup decomposition is sharp  if and only if the normalizer decomposition is sharp 
(over the collection of $\cF$-centric subgroups).
\end{theorem}  
 
In the last section of the paper, we consider the Dwyer space $X:=X_{\cC} ^{\beta}$ for the subgroup decomposition.
When $G$ is an infinite group, it is not true in general that for every $H \in \cC$, the orbit space $C_G(H)\backslash X^H$ 
is $R$-acyclic (see Example \ref{ex:Sharpness}). However we show that this holds when $G$ is a finite group, $R=\bbZ_{(p)}$, 
and $\cC$ is a collection of $p$-subgroups of $G$ closed under taking $p$-overgroups (see Proposition \ref{pro:DwyerSubgroup}). 
As a consequence, we obtain the following. 

\begin{theorem}\label{thm:IntroTheSame}
If $G$ is a finite group and $\cC$ is a collection of $p$-subgroups of $G$ closed under taking $p$-overgroups, 
then for every $\bbZ_{(p)} \overline \cF _{\cC}(G)$-module $M$,   $$H^* ( \overline \cF _{\cC} (G) ; M ) \cong 
H ^* ( \cO _{\cC} (G) ;  \Res_{pr} M) $$ where the isomorphism is induced by the projection functor  
$pr: \cO _{\cC} (G) \to \overline \cF _{\cC} (G)$.
\end{theorem}

Using Theorem \ref{thm:IntroTheSame}, we also prove the following vanishing result. 

\begin{theorem}\label{thm:Vanishing} Let $\cF$ be a saturated fusion system over $S$ realized by a finite 
group and $\cC$ be a collection of subgroups of $S$ closed under taking $p$-overgroups. 
If $\cC$ includes all $\cF$-centric and $\cF$-radical subgroups in $S$, then for all $n\geq 0$
and $i>0$, $$ \underset{ \cO(\cF_{\cC} )}{\lim {}^i}  H^n (- ; \bbF _p )=0.$$
\end{theorem}

This generalizes the vanishing result proved by Diaz and Park  in \cite[Thm B]{DiazPark} 
for fusion systems realized by finite groups and for the collection of $\cF$-centric subgroups.  
We end the paper with an explicit example of an infinite group $G$
where Theorem \ref{thm:IntroTheSame} no longer holds for the collection of all $p$-subgroups 
(see Example \ref{ex:Sharpness}). For the infinite group $G$ in this example, the subgroup decomposition 
for $BG$ is not sharp, but Conjecture \ref{conj:Sharpness} still holds for the fusion system $\cF=\cF_S(G)$.
 
{\it  Notation:} Throughout the paper, $G$ is a discrete group and $p$ is a prime number. 
When there exists a Sylow $p$-subgroup of $G$, it is denoted by $S$. 
In homological algebra sections, we work over an arbitrary commutative ring $R$ with unity. 
When we say $\cC$ is a collection of subgroups in $G$, we always assume that $\cC$ is closed 
under conjugation.  The subcategory of the fusion system $\cF$  with object set $\cC$ is denoted 
by $\cF_{\cC}$.  The full subcategory of $\cF$ generated by the collection of $\cF$-centric subgroups in $S$ is denoted 
by $\cF^c$, and the orbit category of $\cF$ defined over $\cF$-centric subgroups is denoted by $\cO^c (\cF)$.

{\bf Acknowledgements:}  
We would like to thank the referee for a careful reading of the paper and for 
many corrections and valuable suggestions.  In particular current versions of Lemma 4.6 and Lemma 8.13
were suggested by the referee along with many other comments that improved the paper
substantially.

\setcounter{tocdepth}{1}
\tableofcontents


\section{Fusion systems and the orbit category}\label{sect:FusionOrbitCat}

In this section we introduce some basic definitions related to fusion systems and orbit categories. 
For further details on this material, we refer the reader to  \cite{AKO}, \cite{AshOliver},
and \cite{Craven-Book}.

\subsection{Fusion systems} 
\begin{definition}\label{def:FusionSystem}\label{sect:Fusion} 
A \emph{fusion system} $\cF$ over a finite $p$-group $S$ is a category whose objects are subgroups of $S$
and whose set of morphisms $\Mor _{\cF} (P, Q)$ between two subgroups $P, Q \leq S$ satisfies the 
following properties:
\begin{enumerate}
\item $\hom_S(P, Q) \subseteq  \Mor _{\cF} (P, Q) \subseteq \mathrm{inj} (P,Q)$ where $\hom _{S} (P, Q)$ denotes the set of 
all conjugation maps $c_s : P \to Q$ induced by an element in $S$.
\item If $\varphi :  P \to Q$ is a morphism in $\cF$, then $\varphi ^{-1}: \varphi P \to P$ is a morphisms in $\cF$.
\end{enumerate}
\end{definition}

Two subgroups $P$ and $Q$ are \emph{$\cF$-conjugate} if there is an isomorphism $\varphi : P \to Q$ in $\cF$. 
In this case we write $P \sim _{\cF } Q$.  A subgroup $P\leq S$ is \emph{fully $\cF$-normalized}  if $|N_S (P )| 
\geq |N_S (P') |$ for every $P'\sim _{\cF } P$.  A subgroup $P\leq S$ is \emph{fully $\cF$-centralized} 
if $|C_S (P )| \geq |C_S (P') |$ for every $P'\sim _{\cF } P$. We say $P$ is \emph{fully $\cF$-automized} if $\Aut _S (P)$ 
is a Sylow $p$-subgroup of $\Aut _{\cF} (P)$. A subgroup $Q \leq S$ is called \emph{$\cF$-receptive} if every morphism 
$\varphi : P \to Q$ in $\cF$ extends to a morphism $\widetilde \varphi : N_{\varphi } \to S$ where 
$$ N_{\varphi } := \{ x \in N_S (P) \, | \, \exists y \in N_S (\varphi (P) ) \text{ such that} \ c_x 
= \varphi^{-1} \circ c_y \circ \varphi \}.$$ 
 
\begin{definition}{\cite[Part I, Prop 2.5]{AKO}}\label{def:Saturated} A fusion system $\cF$ over $S$ is \emph{saturated} 
if it satisfies the following properties:
\begin{enumerate} 
\item If $P \leq S$ is fully $\cF$-normalized, then $P$ is fully $\cF$-centralized and fully $\cF$-automized.
\item If $P \leq S$ is fully $\cF$-centralized, then $P$ is $\cF$-receptive.
\end{enumerate}
\end{definition}

A subgroup $P \leq S$ is called \emph{$\cF$-centric} if $C_S(Q) \leq Q$ for every  $Q \leq S$ such that 
$Q \sim _{\cF} P$. We denote by $\cF^c$ the full subcategory of $\cF$ whose object set is the set of 
$\cF$-centric subgroups in $S$.

\begin{definition}\label{def:O(F)} The \emph{orbit category $\cO (\cF)$} of the  fusion system $\cF$ is 
the quotient category of $\cF$ whose morphisms are defined by 
$$\Mor _{\cO (\cF)} ( P, Q )= \Inn(Q) \backslash \Mor _{\cF} (P, Q).$$
\end{definition}

One can similarly define the orbit category $\cO (\cF^c)$ of the subcategory $\cF^c$. We define the centric 
orbit category $\cO ^c (\cF)$ to be the full subcategory of $\cO(\cF)$ generated by $\cF$-centric subgroups 
in $S$. Note that  $\cO ^c (\cF) \cong \cO (\cF^c)$.

\subsection{Realization of fusion systems}\label{sect:Realization}

For a discrete group $G$, a finite $p$-subgroup $S$ of $G$ is called a \emph{Sylow $p$-subgroup} of $G$, 
if for every finite $p$-subgroup $P$, there is an element $g\in G$ such that $gPg^{-1} \leq S$.   In general, 
discrete groups do not have Sylow $p$-subgroups even when the orders of finite subgroups are bounded 
from above.  A simple example of such a group is $G=C_2\ast C_2$. If $G$ has a Sylow $p$-subgroup $S$, 
then the fusion system $\mathcal{F}_S(G)$ is defined to be the category whose objects are subgroups of $S$, 
and whose morphisms $P \to Q$ are the group homomorphisms induced by conjugation by elements of $G$.  

If $G$ is a finite group and $S$ is a Sylow $p$-subgroup of $G$, then $\cF_S(G)$ is saturated 
(see \cite[Thm 4.12]{Craven-Book}). In general this is not true for fusion systems induced by infinite groups. 
A fusion system $\cF$ over $S$ is said to be \emph{realized by $G$}, if  $G$ has a Sylow $p$-subgroup 
isomorphic to $S$ and there is an isomorphism of categories $\cF \cong \cF_S(G)$.  We have the following 
realization theorem for fusion systems.
 
\begin{theorem}[Leary-Stancu \cite{LearyStancu}, Robinson \cite{Robinson}]\label{thm:realization} Every fusion 
system $\cF$ is realized by a (possibly infinite) discrete group $G$ constructed as the fundamental group of 
a graph of groups $(\cG, Y)$ with finite vertex and edge groups $G_v$ and $G_e$ over a finite graph $Y$.
\end{theorem}

The graph $Y$ and the vertex groups $G_v$ are described explicitly in each construction. 
Leary and Stancu \cite{LearyStancu} constructs $G$ as an iterated HNN-extension (with only one vertex) 
and their construction works for any fusion system (even if it is not necessarily saturated). Robinson model 
\cite{Robinson} is a generalized amalgamation with vertex groups given by a family of finite groups $G_i$ 
which generate the fusion system. By this we mean that for each $i$, there is a monomorphism 
$\varphi _i: S_i \to S$ from a Sylow $p$-subgroup $S_i$ of $G_i$ to $S$  such that $\cF$ is generated 
by the images of $\cF_{S_i} (G_i)$ under the map induced by $\varphi_i$. By  Alperin's fusion theorem 
\cite[Thm 4.52]{Craven-Book}  and by the model theorem for constrained fusion systems \cite{5A1}, 
a collection of finite groups $\{G_i\}$ that generates $\cF$ can always be found when $\cF$ is a saturated 
fusion system.

\subsection{Fusion orbit category}\label{sect:Orbit} 
Let $G$ be a discrete group and $\cC$ be a collection of subgroups of $G$.  The transporter category
$\cT _{\cC}(G)$ is the category whose objects are subgroups $H \in \cC$, and whose 
morphisms are given by $$\Mor  _{\cT _{\cC}(G ) } (H, K ):=N_G (H, K)=\{ g \in G \, | \, gHg^{-1} \leq K \}.$$ 
 
The orbit category $\cO _{\cC} (G)$ is the category whose objects are subgroups $H \in \cC$, 
and whose morphisms are given by $G$-maps $G/H \to G/K$. A $G$-map $f: G /H \to G/K$ such that 
$f(H)=g^{-1} K$ can be identified with the coset $Kg$ satisfying $gHg^{-1} \leq K$.  This gives an identification 
of morphism sets $$ \Mor _{\cO _{\cC} (G) } (H, K )= K \backslash N_G (H, K)$$ where the $K$-action 
of $N_G(H, K)$ is defined by the left multiplication. 
 
The fusion category $\cF _{\cC}  (G)$ is the category whose objects are the subgroups $H \in \cC$, and whose 
morphisms are given by the group homomorphisms  $c_g : H \to K$ defined by $c_g (h) =ghg^{-1}$ for some 
$g \in G$. Each conjugation map $c_g: H \to K $ can be identified with the coset $g C_G (H) $ where 
$g \in N_G(H, K)$, hence there is a bijection $$ \Mor _{\cF _{\cC}(G)} ( H, K ) = N_G (H, K)/ C_G (H).$$ 

From these identifications, we see that both $\cO _{\cC} (G)$ and $\cF_{\cC} (G)$ are isomorphic to 
quotient categories of the transporter category $\cT _{\cC}(G)$. The fusion orbit category $\overline \cF_{\cC} (G)$ 
is defined in Definition \ref{def:FusionOrbit} with morphism set 
$$\Mor  _{\overline \cF _{\cC} (G ) } (H, K )= K  \backslash N_G (H, K)/  C_G(H).$$ 
Hence there are bijections 
$$\Mor _{\overline \cF _{\cC} (G)} ( H, K ) = \Mor _{\cO _{\cC} (G) } (H, K)/ C_G (H)$$ 
$$\Mor _{\overline \cF _{\cC} (G)} ( H, K ) = \Inn (K) \backslash \Mor _{\cF _{\cC} (G)} (H, K)$$
where $\mathrm{Inn} (K)$ denotes the set of conjugations $c_k : K \to K$ induced by elements $k\in K$.  
These bijections show that the fusion orbit category is a quotient category of both the orbit category 
and the fusion category. We can summarize these in the following diagram:
$$
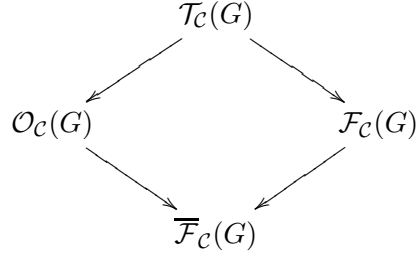
\xymatrix{  & \cT _{\cC} (G) \ar[dl] \ar[dr] &  \\ 
\cO _{\cC} (G)  \ar[dr] & & \cF _{\cC} (G)  \ar[dl] \\
& \overline \cF_{\cC} (G) & \\} $$
where each arrow is the projection functor to the corresponding quotient category. 

Let $\cF$ be  a fusion system over $S$, and $G$ be a discrete group with a Sylow $p$-subgroup $S$ such that 
$\cF \cong \cF_S(G)$. If we take $\cC$ to be the collection of all $p$-subgroups of $G$, then the fusion 
category $\cF _{\cC} (G)$  is equivalent (as a category) to the fusion system $\cF$. If $\cC$ is the collection 
of all $p$-subgroups of $G$ that are conjugate to an $\cF$-centric subgroup in $S$, then $\cF_{\cC}(G)$ 
is equivalent to $\cF^c$ as categories. We show in Lemma \ref{lem:CentricG} that the collection 
of all $p$-subgroups of $G$ that are conjugate to an $\cF$-centric subgroup in $S$ is equal 
to the collection of all $p$-centric subgroups in $G$. 

\begin{lemma}\label{lem:CatEquivalence} Let $\cF$ be a saturated fusion system over $S$, and $G$ be a discrete 
group with a Sylow $p$-subgroup $S$ realizing $\cF$. Let $\cC$ be the collection of all $p$-subgroups 
of $G$ that are conjugate to an $\cF$-centric subgroup in $S$. Then the orbit category $\cO ^c (\cF)$ 
is equivalent, as categories, to the fusion orbit category $\overline \cF_{\cC} (G)$.
\end{lemma}

\begin{proof} By definition of $\cC$, every object in $\overline \cF_{\cC} (G)$ is isomorphic to an $\cF$-centric 
subgroup in $S$. For $\cF$-centric subgroups $P, Q \leq S$, $$\Mor _{\cO (\cF^c)} (P, Q)= \Inn(Q) \backslash 
\Mor _{\cF ^c} (P, Q) =  \Inn (Q) \backslash \Mor _{\cF _{\cC} (G)}  (P, Q) =\Mor _{\overline \cF _{\cC} (G)} (P, Q).$$
Hence these two categories are equivalent.
\end{proof}

By Proposition \ref{pro:EquivCat}, equivalent categories have isomorphic cohomology groups, 
hence the higher limits over $\cO ^c (\cF)$ are isomorphic to the higher limits over $\overline \cF _{\cC} (G)$.


\section{Modules over the orbit category and induction}\label{sect:Induction}

In this section, we introduce the preliminaries on modules over small categories 
and discuss the restriction and  induction functors induced by a functor between 
two categories. We refer the reader to \cite[Chp 9]{Lueck-Book}  and \cite{HPY} 
for more details on homological algebra over orbit categories. 

\subsection{Cohomology of a small category} Let $\bC$ be a nonempty small category 
and $R$ be a commutative ring with unity. A (right) $R\bC$-module $M$ is a contravariant 
functor $M : \bC \to R$-mod, and the $R\bC$-module homomorphism $\varphi: M_1\to M_2$  
is defined as a natural transformation of functors. The category of $R\bC$-modules is an abelian 
category so the usual notions of kernel, cokernel, and exact sequence exist and they are defined 
objectwise. For example, a short exact sequence of $R\bC$-modules $M_1\to M_2 \to M_3$ 
is exact if for every $x\in \Ob(\bC)$ the sequence of $R$-modules $$M_1 (x)  \to M_2 (x) \to M_3 (x)$$ 
is exact. For $x\in \Ob (\bC)$, we define the $R\bC$-module $R\Mor _{\bC} (?, x)$ 
as the module which takes $y\in \Ob (\bC)$ to the free $R$-module $R\Mor _{\bC} (y, x)$. 
For any $x\in \bC$ and for any $R\bC$-module $M$, there is an isomorphism 
$$\hom _{R \bC } (R\Mor _{\bC} (? , x) , M)  \cong M(x).$$
This proves that for every $x\in \Ob (\bC)$, the $R\bC$-module $R\Mor_{\bC} (?, x)$ 
is projective. Using these projective modules one can show that for every $R\bC$-module 
$M$, there is a projective module $P$ and a surjective $R\bC$-module homomorphism 
$P \to M$. Hence there are enough projectives in the category of $R\bC$-modules. 
There are also enough injectives in this category (see \cite[p. 43]{Weibel-Book}). 

For $R\bC$-modules $M$ and $N$, we define the $n$-th ext-group to be
$$\Ext _{R\bC} ^n (M, N) := H^n ( \hom _{R\bC} (P_*, N))$$ where $P_* \to M$ 
is a projective resolution of $M$ as an $R\bC$-module. By the balancing theorem 
for ext-groups, the ext-group $\Ext _{R\bC} ^n (M, N)$ can also be calculated 
as the $n$-th cohomology of the cochain complex  $\hom _{R\bC} (M, I^*)$ 
where $I^*$ is an injective co-resolution of $N$ as an $R\bC$-module 
(see \cite[Thm 2.7.6]{Weibel-Book}). 

\begin{definition}
The \emph{constant functor} $\underline R$ over the category $\bC$ is the $R\bC$-module 
that sends every object $x\in \Ob (\bC)$ to $R$ and every morphism in $\bC$ to the  identity map 
$\id_R : R \to R$. If $M$ is an $R\bC$-module, then for every $n \geq 0$, the \emph{$n$-th 
cohomology group of $\bC$} with coefficients in $M$ is defined to be 
$$ H^n (\bC ; M) :=\Ext  ^n _{R \bC} (\underline R, M).$$ 
\end{definition}

For any $R\bC$-module $M$, the limit  of $M$ over $\bC$ is defined by
$$\lim _{x \in \bC } M(x) := \{ (m_x) \in \prod _{x \in \bC} M(x) \, | \,  
M(\alpha) (m_y)= m_x \text{ for every }  \alpha \in \Mor _{\bC} (x, y) \}.$$ 
The functor $M \to \lim _{x\in \bC} M$ is left exact. The $n$-th right derived functor 
of $\lim _{\bC} (-)$ is called the \emph{$n$-th higher limit of $M$}, and it is denoted 
by $\lim _{\bC} ^n M$. Since $$\hom_{R\bC}  (\underline R; M)\cong \lim _{\bC} M,$$ 
there is an isomorphism $$\underset{\bC\ }{\lim {}^n} M \cong H^n (\bC, M)$$ of right 
derived functors for every $n\geq 0$.  Throughout the paper we will replace the higher 
limits $\lim _{\bC} ^n M$ of an $R\bC$-module $M$ with the cohomology groups 
$H^n (\bC; M)$ without further explanation.

When $M: \bC \to R$ is a covariant functor, then we say $M$ is a left $R\bC$-module.
For left $R\bC$-modules, we can define projective resolutions and ext-groups in the same 
way that we defined them for the right $R\bC$-modules. In fact, a left $R\bC$-module 
is a right $R\bC^{op}$-module, so all the definitions above can be repeated 
easily. The $n$-th cohomology $H^n (\bC; M)$ of $\bC$ with coefficients in a left $R\bC$-module 
$M$ is defined to be the ext-group $\Ext ^n _{R\bC} (\underline R, M)$ over the category 
of left $R\bC$-modules.The $n$-th higher limits is defined in a similar way and we have 
$\underset{\bC}{\lim{}^n}  M \cong H^n (\bC; M)$.

\begin{remark} 
In the literature, the cohomology groups of a category $\bC$ are sometimes defined only 
for a left $R\bC$-module and the cohomology groups of $\bC$ with coefficients in a right 
$R\bC$-module $M$ are denoted by $H^* (\bC ^{op} ; M)$. In this paper, most of our cohomology 
groups are with coefficients in a right $R\bC$-module so it becomes inconvenient for us 
to use this convention. We denote both of the cohomology groups of $\bC$ with coefficients 
in a left $R\bC$-module and a right $R\bC$-module by $H^* (\bC; M)$. 
Similarly we will always denote the higher limits of both covariant and contravariant
functors $M: \bC \to R\text{-Mod}$ over $\bC$ by $\underset{\bC}{\lim {}^i } M$.
\end{remark}

\begin{remark}\label{rem:Emptycat} The empty category $\emptyset$ is a category 
with no objects and no morphisms. For every category $\bC$, there is a unique functor 
$F : \emptyset \to \bC$, i.e. the empty category $\emptyset$ is the initial object in the 
category of small categories. If $\bC$ is the empty category then the only $R\bC$-module 
is the zero module which is both injective and projective. Applying the definitions, we see 
that if $\bC=\emptyset$, then $\lim _{\bC} M=0$ and $H^i (\bC, M) = \lim _{\bC} ^i M=0$ 
for all $i$.
\end{remark}

\subsection{Restriction and induction via a functor}\label{sect:ResInd}

Let $F: \bC \to \bD$ be a functor between two nonempty small categories. The restriction functor 
$$\Res_F : R\bD\text{-Mod} \to R\bC\text{-Mod}$$ is defined by composition with $F$, 
i.e., $\Res_F M =M \circ F$. The induction functor $$\Ind _F : R\bC\text{-Mod} \to R\bD\text{-Mod}$$
is defined to be the functor which is left adjoint of the restriction functor. The left adjoint 
of the restriction functor exists because the category of $R$-modules is cocomplete. 
We explain this in detail below using Kan extensions.
 
For an $R\bC$-module $M$,  let $LK_{F^{op}} (M)$ denote  the left Kan extension of 
$M: \bC ^{op} \to R$-Mod along the functor $F ^{op} : \bC ^{op}  \to \bD ^{op} $. 
There is a formula for left Kan extensions using colimits over a comma category. 
We first recall the definition of a comma category.

\begin{definition}\label{def:Comma} Let $F :\bC \to \bD$ be a functor and $d\in \Ob (\bD)$. 
The \emph{comma category}  $d \backslash F$ is the category whose 
objects are the pairs $(c, f)$ where $c\in \Ob (\bC)$ and $f \in \Mor _{\bD} ( d , F(c) )$, and whose 
morphisms $(c, f) \to (c', f')$ are given by the morphisms $\varphi : c\to c'$ in $\bC$ such that 
$f'= F(\varphi) \circ f$. 
\end{definition}

The comma category $F/d$ is defined in a similar way as the category whose 
objects are the pairs $(c, f)$ where $c\in \Ob (\bC)$ and $f \in \Mor _{\bD} ( F(c), d)$.
Note that $F^{op}/d=(d \backslash F)^{op}$, and there is a functor  $\pi_d: F^{op} / d \to \bC ^{op} $ 
defined by $(c, f)\to c$. For every $d\in \Ob(\bD)$, we have
$$LK_{F^{op}} (M) (d) =\colim _{F^{op} /d} (M \circ \pi_d) =\colim _{(d\backslash F) ^{op}} (M \circ \pi_d)$$
(see \cite[Thm 6.2.1]{Reihl-Book}). Since the colimits exists in the category of $R$-modules, the left Kan-extension 
exists (see \cite[Cor 6.2.6]{Reihl-Book}). 
When it exist, the Kan extension functor $LK_{F^{op}} (-)$ is left adjoint to the functor 
defined by precomposing with $F^{op}$, i.e., to the restriction functor $\Res_F (-)$  
(see \cite[Prop 6.1.5]{Reihl-Book}). Hence we can take
$\Ind_F(-)$ to be the functor $LK_{F^{op}} (-)$. 

In \cite[9.15]{Lueck-Book}, the induction functor $\Ind _F (-)$  is defined by using a tensor product with the
$R\bC$-$R\bD$-bimodule $R\Mor _{\bD} (??, F(?)) : \bC \times \bD ^{op} \to R$-Mod 
defined by $$ (c, d) \to R\Mor _{\bD} ( d, F(c) )$$ on objects. 
This gives an explicit formula for $\Ind _F M$ that can be described 
as follows:  for every $x \in \Ob (\bD)$, $$(\Ind _F M) (x) = \Bigl ( \bigoplus _{y \in \Ob (\bC) } 
M(y) \otimes_R R\Mor _{\bD} (x, F(y)) \Bigr ) /J$$ where $J$ is the ideal generated by the elements 
of the form $m \otimes f - m' \otimes f'$ where  $m \in M(y)$, $f \in \Mor _{ \bD} (x , F(y))$, $m'\in M(y')$, 
and $f' \in \Mor _{\bD} (x', F(y') )$ such that there is a morphism $\varphi \in \Mor _{\bC} (y, y')$ satisfying 
$$M(\varphi ) (m')=m \ \ \text{ and }  \ \  f'= F( \varphi ) \circ f.$$ Since the tensor product is adjoint 
to the hom-functor, we have the following:

\begin{lemma}\label{lem:Adjoint} Let $\Ind _F: R\bC\text{-Mod} \to R\bD\text{-Mod}$ be the functor 
defined by the formula given above. Then for every $R\bC$-module $M$ and $R\bD$-module $N$, 
there is a natural isomorphism $$\hom _{R\bC} (\Ind _F M, N ) \cong \hom _{R \bD} (M, \Res_F N).$$
\end{lemma}

\begin{proof} See \cite[9.22]{Lueck-Book}.
\end{proof}

Note that these two different descriptions of $\Ind _F M$ coincide by the uniqueness of the left adjoints.
We can also see this by proving that the two formulas for $(\Ind _F M) (x) $ given by colimits and 
by using the definition of tensor products give isomorphic modules for every $x \in \Ob (\bD)$.

\begin{lemma}\label{lem:TakesProjPorj}
The restriction functor $\Res_F(-)$ preserves exact sequences, hence its left adjoint the induction 
functor $\Ind _F (-)$ takes projectives to projectives. 
\end{lemma}

\begin{proof} This is clear from the definition of the restriction functor. The second statement follows 
from the adjointness of the restriction and induction functor.
\end{proof}

 In general the induction functor $\Ind _F(-)$ is not an exact functor. When $\Ind _F(-)$  is an exact functor, 
there is a version of Shapiro's isomorphism for ext-groups over the corresponding module categories.

\begin{proposition}\label{pro:Shapiro} Let $F:\bC \to \bD$ be a functor such that the associated 
induction functor $\Ind _F : R\bC\text{-Mod} \to R\bD\text{-Mod}$ is exact. Then 
for every $R\bC$-module $M$ and for every $R\bD$-module $N$, there is an isomorphism 
$$\Ext^* _{R\bD} (\Ind _F M, N) \cong \Ext ^*_{R\bC} (M, \Res_F N),$$  called 
\emph{Shapiro's isomorphism} for the functor $F:\bC \to \bD$.
\end{proposition}

\begin{proof} Let $P_* \to M$ be a projective resolution of $M$ as an $R\bC$-module. Since the induction 
functor takes projectives to projectives, the induced module $\Ind_F P_i$ is a projective $R\bD$-module 
for every $i$. By assumption the induction functor $\Ind_F(-)$ is an exact functor, therefore we can conclude that 
$\Ind _F P_* \to \Ind_F M$ is a projective resolution of $\Ind_FM$ as an $R\bD$-module. This gives
\begin{align*}
\Ext^* _{R\bD} (\Ind _F M, N) & \cong H^* (\hom _{R\bD} (\Ind_F P_*, N)) \\ & \cong 
H^*(  \hom _{R\bC} (P_*, \Res_F N) ) \cong \Ext ^*_{R\bC} (M, \Res_F N ). 
\end{align*}
\end{proof}

\begin{remark}\label{rem:EmptyFunctor}  If $\emptyset$ is the empty category and  
$F : \emptyset \to \bC$ is the unique functor between $\emptyset$ and a category $\bC$, 
then $\Res _F : R\bC\text{-Mod} \to R\emptyset\text{-Mod}$ is the functor which takes 
every $R\bC$-module to the zero module. $\Ind _F : R\emptyset\text{-Mod}  \to R\bC\text{-Mod}$ 
is the functor which sends the zero module to the zero  module. It is easy to see that  both 
Lemma \ref{lem:Adjoint} and Proposition \ref{pro:Shapiro} hold for this functor $F$.
\end{remark}

We will use the empty category in the paper when we are discussing the restriction and 
induction functors for the (fusion) orbit categories. We also will refer to the following 
well-known result for the equivalent categories.

\begin{proposition}\label{pro:EquivCat} If $F: \bC \to \bD$ is an equivalence of 
two small categories, then for every $R\bD$-module $M$, there is an isomorphism 
$$H^n (\bD ; M) \cong H^n (\bC ; \Res _F M)$$ induced by $F$.
\end{proposition}

\begin{proof} See \cite[Prop 5.6]{AKO}.
\end{proof}


\section{Induction from a subgroup}\label{sect:IndSub}
 
Let $\bG$ denote the category with one object whose set of 
endomorphisms is a group given by the group $G$. Then an $R\bG$-module is the same as 
an $RG$-module. If $H \leq G$ is a subgroup of $G$, then the inclusion map $H \to G$  
induces a functor $F : \bH \to \bG$, and the corresponding induction functor 
$\Ind _F : RH\text{-Mod} \to RG\text{-Mod}$ coincides with the usual induction of a module 
defined by $\Ind _H ^G M=M \otimes _H RG$.  Since $RG$ is free as an $RH$-module, 
$\Ind _H ^G (-)$ is an exact functor. The exactness of the induction gives an isomorphism 
$$\Ext^* _{RG} (\Ind _H ^G N, M) \cong \Ext ^*_{RH} (N, \Res^G _H  M)$$ called \emph{Shapiro's 
isomorphism} for ext-groups. We will show below that Shapiro's isomorphism also holds 
for ext-groups over the (fusion) orbit category.

\subsection{Induction for the orbit category} 

Let  $\cC$ be a collection of subgroups in $G$, and let $H$ be a subgroup of $G$. We will 
consider the orbit category of $H$ over the collection $$\cC |_H := \{K \in \cC \, | \, K \leq H\}.$$  
We denote the orbit category $\cO _{\cC |_H} (H)$ by $\cO _{\cC} (H)$ to simplify the notation. 
Let $$i_H^G  : \cO  _{\cC} (H) \to \cO _{\cC} (G)$$ be the functor which takes a subgroup 
$K \in \cC |_H$ to itself in $\cC$, and takes an $H$-map $f: H/K\to H/L$ to the $G$-map 
$i_H ^G (f) : G/K \to G/L$ defined by $i_H ^G (f) (gK )=gf(K)$.  We denote the induction 
functor associated to $i_H^G$ by $\Ind _{\cO _{\cC} (H)}  ^{\cO _{\cC} (G)}$.  If $\cC|_H$ is the empty set, then 
$\cO_{\cC} (H)$ is the empty category $\emptyset$, and $i_H ^G$ is the unique functor 
$\emptyset \to \cO _{\cC} (G)$. In this case the restriction and induction functors are defined 
as described in Remarks \ref{rem:Emptycat} and \ref{rem:EmptyFunctor}. 

There is an explicit  formula for the induced module 
$\Ind _{\cO _{\cC} (H)}  ^{\cO _{\cC} (G)} M$  which is due to Symonds \cite[p. 266]{Symonds} (see also
Lemma 3.1 in \cite{HPY}).

\begin{proposition}[Symonds \cite{Symonds}]\label{pro:OrbitInd} 
Let $M$ be an $R\cO_\cC (H)$-module. Then for every $K \in \cC$,
$$ (\Ind _{\cO _{\cC} (H)}  ^{\cO _{\cC} (G)} M ) (K) \cong \bigoplus _{x^{-1} H \in (G/H)^K } M({}^x K).$$
\end{proposition}
 
Note that the indexing set $(G/H)^K$ for the above direct sum can be identified with 
$${\rm Map}_G (G/K, G/H)=\Mor _{\cO (G) } (K, H)$$ where $\cO(G)$ denotes the orbit 
category of $G$ over all subgroups of $G$. For every $G$-map $f: G/L \to G/K$, 
the induced $R$-module homomorphism $$f^*: (\Ind _{\cO _{\cC} (H)}  ^{\cO _{\cC} (G)}  M) (K) 
\to (  \Ind _{\cO _{\cC} (H)}  ^{\cO _{\cC} (G)}  M     ) (L)$$  
can be described using the map $\Mor _{\cO (G)} (K, H) \to \Mor _{\cO(G)} (L, H) $ induced 
by $f$. On the summands one uses inclusion and conjugation maps between summands
(see \cite[p. 266]{Symonds} for details). As a consequence, we conclude the following.
 
\begin{proposition}[{\cite[Lemma 2.9]{Symonds}}]\label{pro:OrbitShapiro}
The induction functor $$ \Ind _{\cO _{\cC} (H)}  ^{\cO _{\cC} (G)}  : R\cO _{\cC} (H)\text{-Mod} \to R\cO _{\cC} (G)\text{-Mod}$$ 
is exact. Hence, for every $R\cO _{\cC} (H)$-module $N$ and $R\cO_{\cC} (G)$-module $M$, 
there is an isomorphism $$\Ext^* _{R\cO_{\cC} (G) } (\Ind _{\cO _{\cC} (H)}  ^{\cO _{\cC} (G)} N, M) \cong 
\Ext ^*_{R\cO_{\cC} (H) } (N,   \Res _{\cO _{\cC} (H)}  ^{\cO _{\cC} (G)}     M).$$  
\end{proposition}

\begin{proof} The first sentence follows from Proposition \ref{pro:OrbitInd}. The second part 
is a consequence of Proposition \ref{pro:Shapiro}.
\end{proof}

Let $R[G/H^?]$ denote the $R\cO_{\cC} (G)$-module defined by $K \to R[(G/H)^K]$ for every 
$K\in \cC$. If  $K \in \cC$ such that $(G/H)^K =\emptyset$, then $R[(G/H)^K]=0$. If there is no 
$K \in \cC$ such that $K \leq H$, i.e. if $\cC|_H$ is the empty collection, then $R[G/H^?]=0$ 
as an $R\cO _{\cC} (G)$-module.
As a consequence of Proposition \ref{pro:OrbitInd}, we obtain the following:

\begin{proposition}{\cite[Lemma 2.7]{Symonds}}\label{pro:OrbitIso}
Let $\underline R$ denote the constant functor for $\cO_{\cC} (H)$.  Then, 
there is an isomorphism of $R \cO _\cC (G)$-modules
$$\Ind _{\cO _{\cC} (H)}  ^{\cO _{\cC} (G)} \underline R \cong R[G/ H^?].$$ 
\end{proposition}

\begin{proof} By Proposition \ref{pro:FusionInd}, for every $K \in \cC$, we have  
$$(\Ind _{\cO _{\cC} (H)}  ^{\cO _{\cC} (G)} \underline R) (K) \cong  \bigoplus _{x^{-1} H \in (G/H)^K } R \cong R[(G/ H)^K].$$
For every $G$-map $f: G/L \to G/K$, the induced maps $$f^* : (\Ind _{\cO _{\cC} (H)}  ^{\cO _{\cC} (G)} \underline R) (K) \to
(\Ind _{\cO _{\cC} (H)}  ^{\cO _{\cC} (G)} \underline R) (L)$$ and  $f^*: R[(G/H)^K ]\to R[(G/H)^L]$ commutes with the
isomorphisms given above.  
\end{proof}

As a consequence we have the following.
 
\begin{corollary}\label{cor:OrbitIso}
For every $R\cO_{\cC} (G)$-module $M$, there is an isomorphism 
$$\Ext ^n _{R\cO_{\cC} (G) } (R [G/H^?], M ) \cong H^n ( \cO _{\cC} (H);  \Res _{\cO _{\cC} (H)}  ^{\cO _{\cC} (G)} M).$$ 
\end{corollary}
 
 \begin{proof}  This follows from Propositions \ref{pro:OrbitShapiro} and \ref{pro:OrbitIso}, and from the definition
 of the cohomology of a category $\bC$ as the ext-group over $R\bC$.
 \end{proof}

 
\subsection{Induction for the fusion orbit category}

Let $\cC$ be a collection of subgroups in $G$, and  $H$ be a subgroup of $G$.
We denote by $\overline \cF _{\cC} (H)$ the fusion orbit category of $H$ over the collection 
$\cC |_{H}$. Let $$j_H^G:  \overline \cF _{\cC} (H)  \to \overline \cF _{\cC} (G)$$ denote 
the functor which takes a subgroup $K \in \cC |_H$ to itself in $\cC$, and takes a conjugation 
map $c_h: L \to K$ modulo $\Inn(K)$ to itself in $\overline \cF_{\cC} (G)$.  The induction functor 
associated to $j_H ^G$ will be denoted by $\Ind _{\overline \cF _{\cC} (H)}  ^{\overline \cF _{\cC} (G)}$.  
The case where $\cC|_H$ is the empty collection is handled with the empty category 
$\emptyset$ as in the case of the orbit category.
  
\begin{proposition}\label{pro:FusionInd}  For every $R \overline \cF _{\cC} (H)$-module $M$ 
and for every subgroup $K \in \cC$, $$ ( \Ind _{\overline \cF _{\cC} (H)} ^{\overline \cF _{\cC} (G)} M) (K) 
\cong \bigoplus _{[f] \in \Mor _{\overline \cF (G) } (K, H ) } M( f(K))$$ where $\overline \cF(G)$ 
denotes the fusion orbit category defined on all subgroups of $G$. In particular, the induction 
functor $$\Ind _{\overline \cF _{\cC} (H)}  ^{\overline \cF _{\cC} (G)} :  R \overline \cF  _{\cC} (H)\text{-Mod} \to 
R \overline \cF _{\cC} (G)\text{-Mod}$$ is an exact functor.
\end{proposition}

To prove Proposition \ref{pro:FusionInd}, we first prove some lemmas. If $F: \bC\to \bD$ is a
functor and $d\in \Ob (\bD)$, then the comma category $d \backslash F$ is the category whose 
objects are the pairs $(c, f)$ where $c\in \Ob (\bC)$ and $f :d \to F(c) $ is a morphism in $\bD$ 
(see Definition \ref{def:Comma}). If $\bC$ is a subcategory of $\bD$ and $j:\bC \to \bD$ is 
the inclusion functor, then we can assume that the objects of $d \backslash j$ are the morphisms  
$f: d \to c$ in $\bD$ instead of the pairs $(c, f)$. A morphism $f_1 \to f_2$ in $d \backslash j$ 
is a morphism $\varphi : c_1 \to c_2$ in $\bC$ such that $\varphi \circ f_1=f_2$.

\begin{lemma}\label{lem:AbsNon}
Let $\bC$ be a subcategory of a small category $\bD$ and let $j: \bC \to \bD$ be the inclusion functor. 
Assume that
\begin{enumerate}
\item Every morphism $d \to c$ in $\bD$ with $d\in \Ob(\bD)$ and $c\in \Ob(\bC)$ factors as 
$d \maprt{\cong} c' \maprt{C} c$, an isomorphism in $\bD$ followed by a morphism in $\bC$.
\item Given any $c, c', c'' \in \Ob(\bC)$, if $c \maprt{D} c' \maprt{C} c''$ are morphisms 
such that $C \in \bC$, $D \in \bD$, and $C\circ D$ is a morphism in $\bC$ then $D \in \bC$.
\end{enumerate}
Then for any $d \in \Ob(\bD)$, $d \backslash j = \coprod _{i \in I} \bE _i$ where $I$ is a set 
of representatives for the $\bC$-isomorphism classes in the set of all $\bD$-isomorphisms 
$d \to c$, $\bE_i$ is the component of $d \backslash j$ containing $i \in I$ and furthermore 
$i$ is an initial object in $\bE_i$. 
\end{lemma}

\begin{proof} We can write $\bE:=d \backslash j$ as  a disjoint union 
$\coprod _{\alpha \in A} \bE_{\alpha}$ of its connected components 
over some indexing set $A$. By (1), every morphism $f:d \to c$ in $\bD$ factors 
as $d \xrightarrow{i}  c' \xrightarrow{f'} c$ where $i$ is an isomorphism in $\bD$ and 
$f'$ is  a morphism in $\bC$. The morphism $f'$ defines a morphism  $i \to  f$ in $\bE$, 
hence $i$ and $f$ lie in the same component of $\bE$. This shows that in every component 
$\bE_{\alpha}$ there is an object $i: d \to c$ which is a $\bD$-isomorphism. 
Choose a $\bD$-isomorphism $i _{\alpha}: d \to c$ in each component $\bE_{\alpha}$.

If $\varphi : (f_1: d \to c_1) \to (f_2 : d \to c_2) $ is a morphism in $\bE$, then 
there is a commuting diagram 
$$\xymatrix{ & d \ar[dl]_{i_1} \ar[dr]^{i_2} & \\ c_1'  \ar[rr]^{\psi}  
\ar[d]_{f_1'}& & c_2' \ar[d]^{f_2'}  \\  c_1 \ar[rr]^{\varphi} & & c_2}$$
where $\psi:=i_2 \circ i_1^{-1}$ is an isomorphism in $\bD$. Since $\varphi$ is in $\bC$, 
the composition $\varphi \circ f_1'$ is in $\bC$. By condition (2) applied to 
$c_1' \xrightarrow{\psi} c_2' \xrightarrow{f_2'} c_2$, we obtain that 
$\psi$ is in $\bC$. By applying the condition (2) to $c_2' \xrightarrow{\psi ^{-1}} c_1' \xrightarrow{\psi} c_2'$,
we conclude that $\psi ^{-1}$ is also in $\bC$, hence $\psi$ is an isomorphism in $\bC$. 

If there is a zigzag of morphisms between $f_1: d \to c_1$ and $f_n : d \to c_n$  in $\bE$, then applying 
the argument above to each morphism in the zigzag, we obtain an $\bE$-isomorphism between $i_1: d \to c_1'$ 
and $i_n : d \to c_n'$.  Combining this isomorphism with $f_n' $ gives a morphism from 
$i_1$ to $f_n$ in $\bE$. This shows that if $f: d \to c$ lies in the component $\bE_{\alpha}$, then
there is a morphism $\varphi : i_{\alpha} \to f$ in $\bE$. If $\varphi_1, \varphi_2 : i_{\alpha} \to f$
are two such morphisms, then $\varphi _1 \circ i_{\alpha} = f= \varphi _2 \circ i_{\alpha}$
gives $\varphi _1=\varphi_2$. So the morphism $\varphi$ is unique. This proves that $i_{\alpha}$ 
is an initial object in $\bE_{\alpha}$. 

Two $\bD$-isomorphisms $i_1: d \to c_1$ and $i_2: d \to c_2$ lie in the same component of 
$\bE$ if and only if there is a $\bC$-isomorphism $\psi : c_1 \to c_2$ such that $i_2=\psi \circ i_1$. 
So, the index set $A$ can be taken as the $\bC$-isomorphism classes of $\bD$-isomorphisms.
\end{proof}

Using Lemma \ref{lem:AbsNon}, we prove the following.

\begin{lemma}\label{lem:Intermediate} Let $\bC$ be a subcategory of  a small category 
$\bD$ satisfying the conditions in Lemma \ref{lem:AbsNon}. Let 
$$\Ind _{\bC} ^{\bD} : R \bC\text{-Mod} \to R\bD\text{-Mod}$$ denote the induction functor 
induced by the inclusion $j: \bC\to \bD$. Then for every $R\bC$-module $M$ and for every 
$d\in \Ob(\bD)$, we have $$(\Ind _{\bC} ^{\bD} M ) (d) \cong \bigoplus _{(i: d \to c_i)  \in I} M( c_i ) $$
where $I$ is a set of representatives for the $\bC$-isomorphism classes
in the set of all $\bD$-isomorphisms $d \to c$.
\end{lemma}

\begin{proof} By the definition of the induction functor via left Kan extensions given in Section 
\ref{sect:ResInd}, we have
$$(\Ind _{\bC} ^{\bD} M ) (d) = \colim _{(c, f) \in (d \backslash j ) ^{op}} (M \circ \pi_d).$$
By Lemma \ref{lem:AbsNon}, $d \backslash j = \coprod _{i \in I} \bE_i$ where
$\bE_i$ is the component of $d \backslash j$ containing the isomorphism $i:d \to c_i$.  
Since $i$ is an initial object for the $\bE_i$, it is a terminal object for $\bE _i ^{op}$, 
hence  for every $d \in \Ob (\bD)$ we have, 
$$(\Ind _{\bC} ^{\bD} M ) (d) = \bigoplus _{i \in I} \colim _{(c, f) \in \bE_i ^{op} } ( M\circ \pi_d \circ inc_i ) 
\cong \bigoplus _{i \in I} M( c_i )$$ where $inc_i : \bE_i \to d \backslash j$ is the inclusion map. 
This completes the proof.
\end{proof} 

\begin{proof}[Proof of Proposition \ref{pro:FusionInd}]
We claim that if we take $\bC=\overline \cF _{\cC} (H)$ and $\bD=\overline \cF _{\cC} (G)$, 
then the conditions given in Lemma \ref{lem:AbsNon} are satisfied. 

Let $[f]: K \to L $ be a morphism in $\overline \cF _{\cC} (G)$ represented by a morphism 
$f: K \to L$ in $\cF _{\cC} (G)$ such that $L \leq H$. Then $f$ factors as $K \xrightarrow{i} f(K) \xrightarrow{f'} L$ 
where the first map is an isomorphism in $\cF _{\cC} (G)$ and the second map is the inclusion map 
$f': f(K) \hookrightarrow L$ which is a morphism in $\cF _{\cC} (H)$. Then 
$[f]=[f']\circ [i]$ where $[i]$ is an isomorphism in $\overline \cF (G)$ and $[f']$ is in $\overline \cF_{\cC} (H)$.
Hence condition $(1)$ in Lemma \ref{lem:AbsNon} holds for these categories.

Let $U \xrightarrow{[\varphi]} U' \xrightarrow{[\psi]} U''$ be a sequence of morphisms such that 
$U, U', U'' \leq H$, $[\varphi]\in \overline \cF _{\cC} (G)$, and $[\psi] \in \overline \cF _{\cC} (H)$. 
Suppose that $\psi \circ \varphi$ is in $\cF _{\cC} (H)$. Let $g\in G$ and $h_1, h_2\in H$ are such that  
$\varphi =c_g$, $\psi=c_{h_1}$, and $\psi \circ \varphi =c_{h_2}$, then $h_1g=h_2 z$ for some 
$z\in C_G (U)$. This gives $g=h_1^{-1} h_2 z$, hence $\varphi =c_g$ is in $\cF_{\cC} (H)$. 
We conclude that condition (2)  in Lemma \ref{lem:AbsNon} holds for these categories.

Applying Lemma \ref{lem:Intermediate} to $\bC$ and $\bD$, we obtain that for every $K \in \cC$,
$$(\Ind _{\overline \cF _{\cC} (H)} ^{\overline \cF _{\cC} (G)} M ) (K) \cong \bigoplus _{[f] \in I} M( f(K) ) $$
where $I$ is a set of representatives for the $\overline \cF _{\cC} (H)$-isomorphism classes
in the set of all $\overline \cF _{\cC} (G)$-isomorphisms $[f]: K \to f(K)$ with $f(K) \leq H$.
The equivalence relation 
is given by $[f_1]\sim [f_2]$ if there is an isomorphism $[c_h]: f_1 (K) \to f_2(K)$ in $\overline \cF_{\cC} (H)$ such that 
$[f_2]=[c_h] \circ [f_1]$. This shows that the index set $I$ can be taken as the set of morphisms 
$\Mor _{\overline \cF  (G) } (K, H)$ which is the orbit set of $\Mor _{\cF (G)} (K, H) $ under 
the action of $\Inn (H)$.
\end{proof}

The following is an easy consequence of Propositions  \ref{pro:Shapiro}  and  \ref{pro:FusionInd}. 
We call this isomorphism the Shapiro's isomorphism for the fusion orbit category.

\begin{proposition}\label{pro:FusionShapiro}  For every $R\overline \cF_{\cC} (H)$-module 
$N$ and every $R\overline \cF_{\cC} (G)$-module $M$, there is an isomorphism 
$$\Ext^* _{R \overline \cF _{\cC} (G) } (  \Ind _{\overline \cF _{\cC} (H)}  ^{\overline \cF _{\cC} (G)}   N, M) \cong 
\Ext ^*_{R\overline \cF _{\cC} (H) } (N, \Res _{\overline \cF _{\cC} (H)}  ^{\overline \cF _{\cC} (G)} M).$$
\end{proposition}
 
As a consequence of Proposition \ref{pro:FusionInd}, we also have the following.

\begin{proposition}\label{pro:FusionIso}
There is an isomorphism of $R\overline \cF _\cC (G)$-modules
$$\Ind _{\overline \cF _{\cC} (H)}  ^{\overline \cF _{\cC} (G)} \underline R \cong R[ C_G(?) \backslash (G/ H)^?].$$ 
\end{proposition}

\begin{proof} By Proposition \ref{pro:FusionInd}, we have  
$$\Ind _{\overline \cF _{\cC} (H)}  ^{\overline \cF _{\cC} (G)} 
\underline R \cong R[\Mor_{\overline \cF (G) } (?, H)]\cong R[ C_G (? ) \backslash 
\Mor _{\cO(G)} (?, H) ] \cong R[ C_G(?) \backslash (G/ H)^?] $$ 
as $R\overline \cF _\cC (G)$-modules. The induced maps also coincide
because of the description of the isomorphism.
\end{proof}

We conclude the following.

\begin{corollary}
For every $R \overline \cF_{\cC} (G)$-module $M$,
$$   \Ext^* _{R \overline \cF _{\cC} (G) } ( R[ C_G(?) \backslash (G/ H)^?] , M) 
\cong H ^* ( \overline \cF _{\cC} (H)  ;  \Res _{\overline \cF _{\cC} (H)}  ^{\overline \cF _{\cC} (G)} M).$$
\end{corollary}

\begin{proof} This follows from Propositions \ref{pro:FusionShapiro} and \ref{pro:FusionIso}. 
\end{proof}

We use this isomorphism throughout the paper to replace the ext-groups on the left with the cohomology of
the fusion orbit category $\overline \cF _{\cC} (H)$.
 

\subsection{Induction via the projection map}\label{sect:IndProj} 
 
Let $pr : \cO _{\cC} (G) \to \overline \cF _{\cC} (G)$ denote the \emph{projection functor} 
that takes every subgroup $H \in \cC$ to itself and takes a morphism  $f: L \to K$ defined 
by $f(L)=g^{-1} K$ to the morphism  $KgC_G(L)$ in $\overline \cF _{\cC} (G)$. The restriction 
functor $\Res_{pr} $ takes an $R\overline \cF _G$-module $M$ to an $R\cO _{\cC} (G)$-module 
via composition with the functor $pr$. In particular $(\Res _{pr} M) (K) =M(K)$ for every $K \in \cC$. 
If $M'=\Res _{pr} M$ for some   $R\overline \cF _{\cC} (G)$-Module $M$, then for every $K\in \cC$, 
the centralizer $C_G (K)$ acts trivially on $M'(K)$. Such an $R\cO _{\cC} (G)$-module is called 
a \emph{conjugation invariant} $R\cO _{\cC} (G)$-module in \cite{HPY}, and called a \emph{geometric 
coefficient system} in \cite{Symonds}. For the induction functor we have the following observation.

\begin{lemma}\label{lem:IndProj} Let $M$ be an $R\cO _{\cC} (G)$-module. Then for every $K \in \cC$, 
$$(\Ind _{pr} M) (K) \cong M(K)_{C_G(K)} := M / \langle m-c_x m \, | \, m\in M(K),  
x\in C_G(K)\rangle$$ as an $RN_G(K)$-module. 
\end{lemma}

\begin{proof}
By the formula for the induction functor given in Section \ref{sect:ResInd}, for every $K \in \cC$, we have
$$ (\Ind _{pr} M ) (K)=\colim _{(K \backslash pr )^{op}} (M \circ \pi_K).$$
The comma category $\bD:=K \backslash pr$ is the category whose objects are the pairs 
$(L, f)$ where $L \in \cC$ and $f: K \to L$ is a morphism in $\overline \cF _{\cC} (G)$. 
Let $\bC$ be the full subcategory of $\bD$ with one object $(K, \id_K)$.
The automorphisms of $(K, \id_K)$ in $\bD$ are given by the morphisms $Kx: G/K \to G/K$ 
in $\cO _{\cC} (G)$ that are sent to the identity map under the projection functor $pr$. 
So  the automorphism group of $(K, \id_K)$ is the group $KC_G(K)/K=\{ Kx\, |\, x \in C_G(K)\} \leq N_G(K)/K$. 
We claim that the inclusion functor $j: \bC \to \bD$ defines an equivalence  of categories.
To see this, let us fix a morphism $\hat f: K \to L$ in $\cO_{\cC} (G)$ for every $f: K \to L$ in $\overline \cF _{\cC} (G)$.
For every $\varphi : (L_1, f_1) \to (L_2, f_2) $ in $\bD$, there is a unique $\hat \varphi  \in \Aut _{\bD} (K, \id_K)$ such that
$\varphi \circ \hat f_1= \hat f_2 \circ \hat \varphi$.  Define $\pi : \bD \to \bC$ to be the functor that sends each object 
$(L, f)$ in $\bD$ to $(K, \id_K)$ in $\bC$, and each morphism $\varphi : (L_1, f_1) \to (L_2, f_2) $ to $\hat \varphi$ in 
$\Aut _{ \bD} (K, \id _K)$. It is easy to see that $\pi \circ j =\id _{\bC}$ and $j \circ \pi \simeq \id_{\bD}$, so we can conclude that  
$\bC$ and $\bD$ are equivalent categories. This gives
$$ (\Ind _{pr} M ) (K)=\colim _{(K \backslash pr )^{op}} (M \circ \pi_d) \cong \colim _{\bC^{op} } (M \circ \pi_d \circ j) \cong M(K) _{C_G (K)}.$$ 
This completes the proof.
\end{proof}

By Lemma \ref{lem:IndProj}, it is easy to see that $$ \Ind _{pr} \Res_{pr} =\id _{\overline \cF _{\cC}  (G)}.$$ 
For $H \leq G$, there is a commuting diagram of functors 
$$\xymatrix{  \cO _{\cC} (H)  \ar[r]^{pr}  \ar[d]^{i_H ^G}  &  \overline \cF _{\cC} (H) 
\ar[d]^{j_H ^G }\\ \cO _{\cC} (G)  \ar[r]^{pr} & \overline \cF _{\cC} (G)}$$
where $i_H ^G$ and $j_H ^G$ are the functors defined in the previous section. 
For every pair of composable functors $F_1 : \bC \to \bD $ and $F_2: \bD \to \bE$, 
we have $\Ind _{F_2 \circ F_1}= \Ind _{F_2 } \Ind _{F_1} $. In particular, we have
$$\Ind _{j_H^G} \Ind _{pr} =\Ind_{pr} \Ind _{i_H^G}.$$ Thus for every $R\overline 
\cF _{\cC} (H)$-module $M$, $$ \Ind _{j_H^G} M= \Ind _{j_H^G} \Ind _{pr} \Res _{pr} M
=\Ind _{pr} \Ind _{i_H^G}  \Res _{pr} M.$$ This equality can be used to give a second 
proof of Proposition \ref{pro:FusionInd} as a consequence of Proposition \ref{pro:OrbitInd}. 
This explains the similarities between the formulas in Propositions \ref{pro:OrbitInd} and 
\ref{pro:FusionInd}. 

We now state a lemma which also appears in \cite[\S 3]{Lueck-CohChern} 
(for a covariant version, see \cite[\S 3]{Lueck-HomChern}).

\begin{lemma}[{L\" uck \cite[\S 3]{Lueck-CohChern}}]\label{lem:TwoWays}  Let $pr:\cO_{\cC} (G)  
\to \overline \cF_{\cC} (G) $ be the projection functor defined above. Then, for every subgroup 
$H$ of $G$, there is an isomorphism $$\Ind_{pr} (R[G/H^?])\cong R[ C_G(?) \backslash (G/H)^?]$$ 
of $R\overline \cF _{\cC} (G)$-modules.  In particular, for every $G$-set $X$ and for every  
$R\overline \cF _G$-module $N$, $$\hom _{R\overline \cF _{\cC} (G) } (R[C_G(?) 
\backslash X^?], N) \cong \hom _{R \cO _{\cC} (G) } (R[X ^? ] , \Res_{pr } N ).$$
\end{lemma}

\begin{proof} For every $K \in \cC$, we have $$\Ind_{pr} (R[G/H^?]) (K) =R[(G/H)^K]_{C_G(K)} 
\cong R[ C_G(K) \backslash (G/H)^K].$$ The second formula follows from the adjointness of 
induction and restriction functors.
\end{proof}


\section{Fusion Bredon cohomology}\label{sect:BredonCohomology}

In \cite[p. 281]{Symonds}, Symonds defines Bredon cohomology over an arbitrary collection 
$\cC$. The main aim of this section is to extend Symonds' definition to chain complexes over 
the fusion orbit category. We also show that there is a hypercohomology spectral sequence 
converging to the fusion Bredon cohomology of a $G$-CW-complex.

Throughout, $\cC$ is an arbitrary collection of subgroups in the group $G$. 
For all the statements that hold for both orbit category $\cO _{\cC} (G)$ and the fusion orbit category
$\overline \cF _{\cC} (G)$, we use $\bO_\cC (G)$ to denote $\cO _{\cC}  (G)$ or $\overline \cF _{\cC} (G)$.
Similarly $\bO (G)$ is used to denote $\cO (G)$ or $\overline \cF (G)$.

\subsection{Yoneda functors}

Given a category $\bD$ and a full subcategory $\bC \subseteq \bD$, define a functor 
$$ \tilde Y_{\bC} ^{\bD} : \bC ^{op} \times \bD \to R\text{-Mod}, \  \ \tilde Y _{\bC} ^{\bD} (c, d)=R[\Mor _{\bD} (c, d) ].$$
This gives a functor 
$$Y _{\bC } ^{\bD} : \bD \to R\bC\text{-Mod}, \ \ Y_{\bC} ^{\bD} (d)= R [\Mor  _{\bD} (- , d)].$$
If $M$ is an $R\bC$-module, then $\Ext ^j _{R\bC} (-, M)$ is a contravariant functor from $R\bC$-mod to $R$-Mod, so
$$\Ext ^j _{R\bC} (Y_{\bC} ^{\bD} , M)$$ is an $R\bD$-module. 
Applying this construction to the categories $\bO _{\cC} (G) \subseteq \bO (G)$, we obtain that 
for every $R\bO _{\cC} (G)$-module 
$M$ and for $j \geq 0$, 
$$
\tilde M ^j:=\Ext ^j _{R \bO _{\cC} (G) } (Y ^{\bO (G) } _{\bO _{\cC} (G) } , M )
$$
is an $R\bO (G)$-module. 
For any collection $\cD$, we can regard $\tilde M^j$ as an $R\bO _{\cD}(G)$-module by
precomposing with the inclusion  $\bO _{\cD} (G) \subseteq \bO (G)$. 

For every $H \leq G$, we have $$Y^{\cO (G) } _{\cO _{\cC} (G)} (H) \cong R[G/H^?]\ \ 
\text{ and } \ \ Y^{\overline \cF (G) } _{\overline \cF _{\cC} (G) } (H) \cong R[C_G(?)\backslash (G/H)^?].$$
By Shapiro's lemma proved in Propositions
\ref{pro:OrbitIso} and \ref{pro:FusionIso}, for every $H \leq G$, we have an isomorphism of $R$-modules
\begin{equation}\label{eqn:Shapiro}
\tilde M^j (H) \cong H^j (\bO _{\cC} (H); \Res ^{\bO _{\cC} (G) } _{\bO _{\cC} (H) } M ).
\end{equation}
We can use these isomorphisms to define an $R\bO (G)$-module as follows:
 
\begin{lemma}\label{lem:HMj}
For every $R\bO _{\cC} (G)$-module $M$, and for every integer $j \geq 0$, there is an $R\bO (G)$-module
$\overline \cH_M^j $ such that for every $H \leq G$, 
$$\overline \cH ^j _M (H)= H^j (\bO _{\cC} (H) ; \Res ^{\bO _{\cC} (G) } _{\bO _{\cC} (H) } M).$$ For every 
$f: G/K \to G/H$, the induced map $\overline \cH _M ^j (f): \overline \cH_M ^j (H) \to \overline \cH_M ^j (K)$ 
is defined in such a way that the Shapiro's isomorphism given in (\ref{eqn:Shapiro}) induces an isomorphism 
of $R\bO(G)$-modules $$\overline \cH _M ^j \cong \tilde M^j.$$
\end{lemma}

\begin{proof}  For each subgroup $H\leq G$, 
let $\varphi_H: \tilde M^j (H) \xrightarrow{\cong} H^j (\bO _{\cC} (H); \Res ^{\bO _{\cC} (G) } _{\bO _{\cC} (H) } M )$
denote the canonical homomorphism that gives the Shapiro's isomorphism.   
 For every $f: G/K \to G/H$, let $\overline \cH _M ^j(f): \overline \cH_M ^j (H) \to \overline \cH_M ^j (K)$ denote the 
unique homomorphism which makes the following diagram commute
$$\xymatrix{  \tilde M ^j (H)  \ar[r]^{\varphi_H \cong} \ar[d]_{\tilde M ^j (f)} & \overline  \cH _M ^j  (H) \ar[d]^{\cH ^j _M (f)} \\ 
\tilde M ^j (K)  \ar[r]^{\varphi_K \cong }  & \overline  \cH _M ^j (K).}$$ It is clear from the commutativity of the above 
diagram that $\overline \cH _M ^j (-)$ defines  a functor $\bO (G)^{op}\to R\text{-Mod}$,
and the $R$-module isomorphisms given in (\ref{eqn:Shapiro}) define an 
isomorphism of $R\bO(G)$-modules  $\tilde M^j \cong \overline \cH ^j _M$.
 \end{proof}

In our applications we use the $\bO(G)$-module $\overline \cH _M ^j$ as coefficients for the ordinary 
Bredon cohomology of a $G$-CW-complex. This is done by considering $\overline \cH _M ^j$ 
as a $R\cO (G)$-module via the restriction functor induced by the projection functor $pr: \cO (G) \to \bO (G)$. 

\begin{definition}\label{def:HMj} For every $R\bO _{\cC} (G)$-module $M$, and for every integer $j \geq 0$,
we denote by $\cH_M ^j$ the $R\cO (G)$-module $\Res_{pr} \overline \cH ^j _M$.
Note that for each $H\leq G$, we have $$\cH _M ^j (H)= H^j (\bO _{\cC} (H) ; \Res ^{\bO _{\cC} (G) } _{\bO _{\cC} (H) } M).$$ 
\end{definition}


\subsection{Fusion Bredon cohomology} 

Let $X$ be a $G$-CW-complex (with a left $G$-action). For each subgroup $K \in \cC$, the fixed point set 
$X^K$ can be identified with the set of $G$-maps from the transitive $G$-set $G/K$ to $X$. 
This gives a contravariant functor
$$\bO_{\cC} (G) \to \text{CW-complexes}$$ defined by  $K \to X^K$ when $\bO _{\cC} (G)= \cO _{\cC} (G)$ and 
by  $K \to C_{G} (K)  \backslash X^K$ when $\bO _{\cC} (G)= \overline \cF _{\cC} (G)$. 
Composing these functors with the functor which takes a CW-complex to its cellular chain complex 
with coefficients in  $R$, we obtain a chain complex $C_* ( X^? ; R)$  of $R\bO _{\cC}  (G)$-modules, 
 
For each $n\geq 0$, we have 
$$C_n (X^? ; R)\cong \bigoplus _{i \in I_n} R[G/ H_i ^?] \cong \bigoplus _{i\in I_n} 
R\Mor _{\cO _{\cC} (G)} ( ?, H_i )$$ where $I_n$ is the set of $G$-orbits of $n$-dimensional 
cells in $X$, and $H_i$ denotes the stabilizer of the $i$-th cell in $X$.  Similarly, 
$$C_n (C_{G} (?) \backslash X^? ; R) \cong \bigoplus _{i\in I_n} R[ C_{G}(?) 
\backslash (G/H_i)^?]\cong \bigoplus _{i\in I_n} R\Mor _{\overline \cF _{\cC} (G)} ( ?, H_i )$$ 
as $R\overline \cF _{\cC} (G)$-modules. Using the Yoneda functors introduced above, we can write this as 
$$C_n (X^? ; R) \cong \bigoplus _{i\in I_n}  Y_{\bO _{\cC} (G) } ^{\bO (G) } (H_i)$$
as an $R\bO _{\cC} (G)$-module. 

\begin{definition} Let $X$ be a CW-complex. The \emph{ordinary (fusion) Bredon cohomology} 
of $X$ with coefficients in an $R\bO(G)$-module $M$ is defined by 
$$H^* _{\bO (G)} (X^?; M) := H^* (\hom _{R\bO (G)} ( C_* (X^?; R) ; M).$$
For an arbitrary collection $\cC$, the \emph{(fusion) Bredon cohomology} 
of $X$  with coefficients in an $R\bO_\cC (G)$-module $M$ is defined by
$$H_{\bO_\cC (G)} ^* (X ^? ; M) := H^* (\hom _{R\bO _\cC (G) } ( \Tot ^{\oplus} (P_{*,*} ) , M) ),$$ 
where $P_{*, *}$ is a Cartan-Eilenberg resolution of $C_*(X^? ; R)$.  
\end{definition}

 When $\bO (G)=\cO (G)$, this gives the ordinary Bredon cohomology  $H^* _{\cO(G)} (X^?; M)$ 
 as defined in Definition \ref{def:BredonCoh}. When $\bO _{\cC} (G)=\overline \cF _{\cC} (G)$, 
 then $H^* _{\bO_{\cC} (G)} (X^?; M)$ is the \emph{fusion
Bredon cohomology} of $X$ as defined in Definition \ref{def:FusionBredon}.

\begin{lemma} If  the isotropy subgroups of $X$ lie in the collection $\cC$, then $C_*(X^? ; R)$ 
is a chain complex of projective $R\bO_{\cC} (G)$-modules. In this case, for every 
$R\bO _{\cC} (G)$-module $M$ we have $$H^* _{\bO (G) } (X^?; M) \cong H^* _{\bO _{\cC} (G) } (X^?; M).$$ 
\end{lemma}

\begin{proof} 
The first sentence follows from the descriptions of $C_n (X^?; R)$ given above. 
The second part follows from the fact that the Cartan-Eilenberg 
resolution of a chain complex of projective modules is itself.
\end{proof}

\begin{proposition}\label{pro:TwoBredons} Let $X$ be a $G$-CW-complex and $\cC$ be a collection of subgroups of $G$
such that the isotropy subgroups of $X$ lie in $\cC$. Then for every 
$R\overline \cF_{\cC} (G)$-module $M$, there is an isomorphism 
$$H^* _{\overline \cF _{\cC} (G)} (X^?; M ) \cong H^* _{\cO _{\cC} (G) } (X^? ; \Res _{pr} M ).$$  
\end{proposition}

\begin{proof} In this case $C_* (X^?; R)$ is a chain complex of projective $R\bO _{\cC} (G)$-modules, hence we have
$$H^* _{\bO _\cC  (G) } (X^?; M) \cong H^* (\hom _{R \bO _{\cC} (G) } (C_* (X^?; R), M).$$
By Lemma \ref{lem:TwoWays}, for every 
$R\overline \cF_{\cC} (G)$-module $M$, there is an isomorphism
$$\hom _{R\cO _{\cC} (G)} (C_* (X^? ; R), \Res_{pr} M) \cong \hom _{R\overline \cF _{\cC} (G)} ( C_* (C_G(?) \backslash X^? ; R), M).$$
Hence we have
\begin{align*}
H^* _{\overline \cF _{\cC} (G)} (X^?; M ) & \cong H^* (\hom _{R \overline \cF _{\cC} (G) } (C_* (X^?; R), M) \\
& \cong H^* (\hom _{R  \cO _{\cC} (G) } (C_* (X^?; R), \Res_{pr} M) \cong H^* _{\cO _{\cC} (G) } (X^? ; \Res _{pr} M ).
\end{align*}  
\end{proof}

In general when $\cC$ does not include all the isotropy subgroups of $X$, then the isomorphism 
in Proposition \ref{pro:TwoBredons} does not hold.
 
\begin{example}\label{ex:BredonDifferent} Let $G$ be a discrete group. If  $\cC=\{ 1 \}$, $X=pt$, 
and $A$ is an abelian group with trivial $G$-action, then  $H_{\cO_{\cC} (G)} ^* (X; A)$ is isomorphic 
to  the group cohomology $H^* (G; A)$ whereas $H_{ \overline \cF_{\cC} (G)} ^i (X ^?; A)=0$ for $i\geq 1$ 
because $\Aut _{\overline \cF _{\cC} (G)} (1)=1$. So if $G$ and $A$ are taken such that 
$H^i (G, A) \neq 0$ for $i\geq 1$, then in this case the fusion Bredon cohomology 
and Bredon cohomology are not isomorphic.
\end{example}


\subsection{Hypercohomology spectral sequence}  

There are two hypercohomology spectral sequences converging to the Bredon cohomology 
of a $G$-space introduced by Symonds \cite[p. 279]{Symonds}. We show in this section that these 
spectral sequences exist also for the fusion Bredon cohomology.  
Throughout this section, $\bO_{\cC} (G)$ denotes either $\cO _{\cC} (G)$ or 
$\overline \cF _{\cC} (G)$. For a $G$-CW-complex $X$, 
 $X ^?$ denotes either $X ^?$ or $C_G(?)\backslash X^?$ depending on the category $\bO_\cC (G)$.

\begin{proposition}\label{pro:HyperCohSS} 
Let $X$ be a $G$-CW-complex and $M$ be an $R\bO_\cC (G)$-module.
For every integer $j \geq 0$, let $\cH ^j _M$ denote the $R\cO (G)$-module defined in Definition \ref{def:HMj}.
Then  there are two spectral sequences 
$${}^I E_2 ^{s,t} \cong H^s _{\cO (G) } (X^? ; \cH _M ^t )  \ \ \ \text{   and     }   \ \ \ 
{}^{II}  E_2 ^{s,t} \cong \Ext^s _{R\bO _{\cC} (G) } ( H_t (X^? ; R),  M) $$
converging to the (fusion) Bredon cohomology $H^{*}_{\bO_\cC (G)} (X ^?; M)$.   
\end{proposition}
 
\begin{proof} Let $C_*(X^?; R)$ denote the chain complex of $R\bO_{\cC} (G)$-modules for $X$, and
$P_{*,*}$ denote the Cartan-Eilenberg projective resolution of $C_* (X^?; R)$ (see \cite[Def 5.7.1]{Weibel-Book}). 
 Applying the hom-functor $\hom _{R\bO _{\cC} } (-; M)$ to $P_{*,*}$, we obtain a double (cochain) complex where
 $$C^{s, t} =\hom _{R\bO _{\cC} (G)} (P_{s,t} , M)$$
for every $s,t \geq 0$. The differentials are in two directions
$$d_v ^{s,t} : C^{s,t} \to C^{s, t+1}  \ \ \text{  and } \ \ d_h^{s,t} : C^{s,t} \to C^{s+1, t}$$ satisfying $d_h d_v+d_vd_h=0$.
The total complex of $C^{*,*}$ is defined by
$X^n=Tot^n (C^{*, *} )=\oplus _{s+t=n} C^{s,t}$ with differential $d_v+ d_h$. 
 Note that this is a first quadrant double complex so we have
 $\Tot ^{\oplus}=\Tot ^{\pi}$ in this case. 
 
 We can filter the total complex $X^*$ by taking 
 $$F^s (X^{s+t})= \bigoplus _{\substack{i+j =s+t, \\  i\geq s}} C^{i, j}.$$
So the the layers of the filtration are the columns of $C^{*,*}$. This gives a spectral sequence 
$\{{}^I E_r ^{s,t} \}_{r\geq 0}$ with $${}^I E_2 ^{s,t}= H^s_h (H^t _v ( C^{*,*} ))$$
(see \cite[Thm 3.4.2]{Benson-Book2} and \cite[Def 5.6.1]{Weibel-Book}).
For each $s\geq 0$, let $I_s$ denote the 
set of $G$-orbits of $s$-dimensional cells in $X$, and let $H_i$ denote the stabilizer of the $i$-th cell 
in $G$. Then  $$C_s (X^? ; R)\cong \bigoplus _{i \in I_s} Y _{\bO _{\cC} (G)} ^{ \bO (G)} (H_i) $$
as an $R\bO _{\cC} (G)$-module. By Lemma \ref{lem:HMj}, we obtain
$$H_v ^t (C^{s, *} ) \cong \prod _{i \in I_s}   \Ext ^t _{R\bO_{\cC} (G) } (Y_{\bO _{\cC} (G)} ^{\bO (G)} (H_i ),  
M)\cong \prod _{i\in I_s}  \overline \cH_M ^t (H_i)\cong  \hom _{R\bO (G)} ( C_s (X^? ; R) , \overline \cH ^t _M)
$$ for every $s,t \geq 0$. The horizontal differential $d_h$ is induced by the differentials
of the chain complex $C_* (X^? ; R)$. Hence we have $${}^I E_2 ^{s, t} = H^s_h( H_v ^t (C^{*, *} ) )\cong 
H^s ( \hom _{R\bO (G)} ( C_* (X^? ; R) , \overline \cH ^t _M  ))\cong H^s _{\bO (G) } (X^? ; \overline \cH _M ^t ).$$
When $\bO (G)=\overline \cF (G)$, by Proposition \ref{pro:TwoBredons}, we have
$$H^* _{\overline \cF (G) } (X^? ; \overline \cH _M ^t )\cong H^*_{\cO (G) } (X^? ; \Res _{pr} \overline \cH _M ^t )=
 H^*_{\cO (G) } (X^? ; \cH _M ^t ).$$ 
So for both categories $\cO (G)$ and $\overline \cF (G)$, we have  
$${}^I E_2 ^{s, t} \cong H^s _{\cO (G)} (X^?; \cH ^t _M).$$
 
The second spectral sequence comes from filtering the total complex $X^*$ horizontally by taking 
 $$F^s (X^{s+t})= \bigoplus _{\substack{i+j =s+t, \\ j \geq s}} C^{i, j}.$$
In this case the layers of the filtration are the rows of $C^{*,*}$. We take ${}^{II} E_0 ^{s,t}=C^{t,s}$ and 
obtain a spectral sequence 
$\{{}^{II} E_r ^{s,t} \}$ with $${}^{II} E_2 ^{s,t}= H^s _v (H^t _h ( C^{*,*} ))$$
(see \cite[Def 5.6.2]{Weibel-Book}). By the properties of the Cartan-Eilenberg resolutions, the horizontal 
cohomology group $H^t _h (C^{*, s} ) \cong \hom _{R\bO _{\cC} (G)} (Q_s , M )$ where $Q_*$
is a projective resolution of $H_t (X^? ; R)$ (see \cite[Def 5.7.1]{Weibel-Book}). 
Hence $${}^{II} E_2 ^{s,t} \cong  H^s ( \hom _{R\bO _{\cC} (G)} (Q_* , M ) )   
\cong  \Ext ^s _{R\bO _{\cC} (G) } ( H_t ( X^?; R ), M)$$ for all $s,t\geq 0$.
 \end{proof}

One special case of the above hypercohomology spectral sequence  is the case where $C_* ( X ^H)$ 
(resp. $C_* ( C_G(H)\backslash X^H )$) has a homology of a point for every $H \in \cC$. In this case 
we say $X$ is $R\cO _{\cC} (G)$-acyclic (resp. $R\overline \cF_{\cC} (G)$-acyclic).
 
\begin{theorem}\label{thm:HyperCohSS} 
Let $X$ be an $R\bO _{\cC} (G)$-acyclic $G$-CW-complex and $M$ be an $R\bO_{\cC} (G)$-module.
For every integer $j \geq 0$, let $\cH ^j _M$ denote the $R\cO (G)$-module defined in Definition \ref{def:HMj}.
Then there is a spectral sequence
$$E_2 ^{s,t} \cong H^s _{\cO(G)} (X^?; \cH _M ^t ) \Rightarrow H^{s+t} ( \bO_\cC (G) ; M).$$
\end{theorem}
 
\begin{proof}
By Proposition \ref{pro:HyperCohSS}, there are two spectral sequences which converge 
to the Bredon cohomology $H^* _{\bO_\cC (G) } (X^?; M)$. Since $X$ is $R\bO_\cC (G)$-acyclic, we have
$$H_t (X^? ; R ) \cong  \begin{cases} \underline R  & \text{if }\  t=0 \\ 0 &  \text{if } \ t \neq 0 
\end{cases}$$ as an $R\bO_\cC (G)$-module. This gives that the spectral sequence $^{II} E_2 ^{*,*}$  
collapses at the $E_2$-page to the horizontal line at $t=0$, hence  for each $n \geq 0$, 
there is an isomorphism $$H^n _{\bO_\cC (G) } (X^?; M ) \cong H^n  (\bO_\cC (G) ; M).$$ 
Now the first spectral sequence ${}^I E_2 ^{s,t}$ in Proposition \ref{pro:HyperCohSS} becomes 
a spectral sequence that converges to $H^ *( \bO _{\cC} (G); M)$.
\end{proof}


\section{Dwyer spaces for homology decompositions}\label{sect:DwyerSpaces}

The Dwyer spaces $X_{\cC} ^{\alpha}, X_{\cC} ^{\beta}, X_{\cC} ^{\delta}$ for centralizer, subgroup, 
and normalizer decompositions were introduced by Dwyer to give a unified treatment for the homology 
decompositions for classifying spaces of discrete groups. In this section we give the definitions of Dwyer 
spaces and show that the Dwyer spaces for the centralizer and normalizer decompositions satisfy 
the conditions of Theorem \ref{thm:HyperCohSS}. For more details on the homology decompositions 
of classifying spaces, we refer the reader to \cite[\S 7]{Dwyer-Book}, \cite[Chp 5]{BensonSmith-Book}, 
and  \cite{Grodal-Higher}.
 
\subsection{$G$-Categories}  We first recall some definitions on categories with a group action. 
We follow the terminology and notation 
introduced by Grodal in \cite{Grodal-Higher} and \cite{Grodal-Endo}.  Let $G$ be a discrete group. 
A \emph{$G$-category} is a category $\bC$ together with a set of functors $F_g: \bC\to \bC$ for all 
$g\in G$ which satisfy $F_g \circ F_h=F_{gh}$ for all $g, h \in G$, and $F_1=\id_{\bC}$. We denote 
the image of the object $x\in \Ob (\bC)$ under $F_g$ by $gx$. Similarly for a morphism 
$\alpha : x\to y$, we write $F_g(\alpha)=g\alpha$. The nerve $\cN (\bC)$ of a small category $\bC$ 
is a simplicial set whose $n$-simplices are given by a chain of composable maps 
$x_0 \maprt{\alpha_1} x_1 \maprt{\alpha_2}  \cdots \maprt{\alpha_n} x_n$. When $\bC$ is a $G$-category, 
the nerve $\cN (\bC)$ is a $G$-simplicial set where the $G$-action on the simplices 
is given by  $$g( x_0 \maprt{\alpha_1} x_1 \maprt{\alpha_2}  \cdots \maprt{\alpha_n} x_n)
=(gx_0 \maprt{g \alpha_1} gx_1 \maprt{g\alpha_2} \cdots \maprt{g\alpha_n} gx_n)$$
for every $g\in G$.

The geometric realization $|\cN (\bC) |$ of the simplicial set $\cN (\bC)$ is called the geometric 
realization of the category $\bC$, and it is denoted by $|\bC|$. When $\bC$ is a $G$-category, then
the geometric realization $|\bC|$ has a $G$-CW-complex structure. 
For every subgroup $H \leq G$, let $\bC^H$ denote the subcategory of $\bC$ whose objects 
and morphisms are the objects and morphisms of $\bC$ that are fixed under the $H$-action.
We have $$|\bC| ^H =|\cN (\bC) |^H= |\cN (\bC) ^H |=|\cN (\bC^H)|=|\bC^H|.$$ The orbit space 
$|\bC|/G$ is homeomorphic to the topological realization of the simplicial set $\cN (\bC)/G$.

A functor  $F: \bC \to \bD$ between two $G$-categories is called a \emph{$G$-functor} 
if for every $g\in G$, $(i)$ $F(gx)=gF(x)$ for every $x\in \Ob ( \bC)$, and 
$(ii)$ $F(g\alpha)=gF(\alpha)$ for every morphism $\alpha $ in $\bC$. If $F: \bC \to \bD$ 
is a $G$-functor,  then the induced map on the realizations $|F|: |\bC| \to |\bD|$ is a $G$-map. 
A natural transformation $\mu: F \to F'$ between two $G$-functors $F, F': \bC\to \bD$ is called 
a \emph{natural $G$-transformation} if $\mu _{gx} : F(gx) \to F'(gx) $ is equal to 
$g\mu_x : gF(x) \to g F'(x)$ for every $g\in G$ and $x\in \Ob (\bC)$. If $\mu: F \to F'$ is a natural 
$G$-transformation, then  the realizations $|F|, |F'| : |\bC|\to |\bD|$ are $G$-homotopic.  
These give the following:  

\begin{proposition}[{\cite[Prop 5.6]{Dwyer-Book}}]\label{pro:Contractible} 
Let $\bC$ be a $G$-category and $x_0 \in \Ob (\bC^G)$. If the identity functor of $\bC$ 
is connected to the constant functor $c_{x_0} : \bC \to \bC$ by a zigzag of natural 
$G$-transformations, then $|\bC|$ is $G$-equivariantly contractible. In particular, 
the orbit space $|\bC|/G$ is contractible.
\end{proposition} 
 
Our main example of a $G$-category is a category obtained by applying Grothendieck 
construction to a functor $F : \bD \to G$-sets. In this case the Grothendick construction 
is a category $\bD \wr F$ whose objects are pairs $(d, x)$ where $d\in \Ob (\bD)$ and 
$x\in F(d)$. A morphism $(d, x) \to (d', x')$ is a pair $(\alpha, x)$ such that $\alpha : d \to d'$ 
is a morphism in $\bD$ such that $F(\alpha ) (x)=x'$. The composition of morphisms is defined 
by $(\alpha',x') \circ (\alpha, x)= (\alpha' \circ \alpha, x)$. The $G$-action on $\bD \wr F$
is given by $g(d, x)=(d, gx)$ and $g(\alpha, x)=(\alpha, gx)$ for every $g\in G$. There is  
a natural isomorphism of simplicial sets
$$ \hocolim _{d \in \bD} F \maprt{\cong} \cN ( \bD \wr F )$$
(see \cite[Thm 1.2]{Thomason} and \cite[p.~3]{Ramras} for details).
It is easy to see that this isomorphism an isomorphism of $G$-simplicial sets, so the fixed 
point subspaces and the orbit space of $X=|\hocolim _{d \in \bD} F|$ can be computed using 
the nerve  $\cN (\bD \wr F)$.

\begin{remark} Throughout the paper whenever it makes sense we will work 
in the category of simplicial sets and call a simplicial set a \emph{space}  
and a $G$-simplicial set a \emph{$G$-space}.  
Note that the geometric realization of a $G$-space is a $G$-CW-complex, so we can associate 
to a $G$-space $X$ a chain complex $C_* (X^?; R)$ over (fusion) orbit category using its geometric 
realization. This allows us to apply the theorems proved for $G$-CW-complexes to $G$-spaces.
\end{remark}

In the following subsections, $G$ is a discrete group and $\cC$ denotes an arbitrary collection 
of subgroups in $G$.


\subsection{The Dwyer space for the subgroup decomposition} \label{sect:SubDecomp} 

Let $\cO _{\cC} (G)$ denote the orbit category of $G$ over the collection $\cC$. 
The collection $\cC$ can be considered as a poset with the order relation given by the inclusion 
of subgroups. There is a $G$-action on  this poset by conjugation. Let $K_{\cC}$ denote the order 
complex of the poset $\cC$. The simplices of $K_{\cC}$ are the chains of subgroups 
$\sigma=(H_0 < H_1 < \cdots < H_n)$ in $\cC$, and an element  $g\in G$ acts on a chain 
by the action defined by 
$$g \sigma = (gH_0 g^{-1} < gH_1 g^{-1} < \cdots < gH_n g^{-1}).$$
There is a simplicial set associated to the simplicial complex  $K_{\cC} $. We denote this simplicial set 
by $\cN \cC$ and its geometric realization by $|\cC|$.
 
\begin{definition} Consider the Borel construction $EG\times _G \cN \cC :=(EG \times \cN \cC)/G$. 
A collection $\cC$ is called \emph{$p$-ample}  if the map induced by projection to the first
coordinate $EG \times _G \cN \cC \to BG$ is a mod-$p$ homology isomorphism.
\end{definition} 

The Dwyer space for the subgroup decomposition is defined as follows.

\begin{definition}\label{def:DwyerSubgroup}
Let $\widetilde \beta : \cO _{\cC} (G) \to G\text{-Sets}$ denote the functor 
that sends $H \in \cC$ to the transitive $G$-set $G/H$.  The $G$-space 
$$X_{\cC} ^{\beta} := \hocolim _{\cO _{\cC} (G)} \widetilde \beta$$ is called 
the \emph{Dwyer space for the subgroup decomposition}. 
\end{definition}

The space $X_{\cC} ^{\beta}$ is a nerve 
of a category $\bX _{\cC} ^{\beta}$ whose objects are pairs $(H, gH)$ with $H \in \cC$ and $gH\in G/H$.
A morphism $(H, xH) \to (K, yK)$ is a $G$-map $f: G/H\to G/K$ such that $f(xH)=yK$.  
There is a $G$-map $X_{\cC} ^{\beta} \to \cN \cC$ which is a weak equivalence (non-equivariantly). 
This gives a homotopy equivalence $EG\times _G X_{\cC} ^{\beta} \to EG\times _G \cN \cC$.
If $\beta : \cO _{\cC } (G) \to Spaces$ is the functor defined by $\beta= EG \times _G \widetilde \beta$, then
there is a natural isomorphism $$EG\times_G X_{\cC} ^{\beta} \cong  \hocolim _{\cO _{\cC} (G) } \beta$$
(see \cite[Prop 4.7.6]{BensonSmith-Book}). Hence if $\cC$ is a $p$-ample collection, then the composition 
$$\hocolim _{\cO _{\cC} (G) } \beta \cong EG\times _G X_{\cC} ^{\beta} \to EG \times _G \cN \cC  \to  BG$$ is a mod-$p$ 
homology isomorphism (see \cite[Prop 7.14]{Dwyer-Book} or \cite[Thm 5.5.4]{BensonSmith-Book}  for details).
This mod-$p$ homology isomorphism is called the \emph{mod-$p$ subgroup decomposition for $BG$}  

The Dwyer space $X=X_{\cC} ^{\beta}$ is also denoted by $E_{\cC} G$ and called the universal space 
for the collection $\cC$. For every $H \in \cC$, the fixed point subspace $X^H$ is homotopy equivalent 
to the realization of the subposet $\cC_{\geq H}=\{ K\in \cC \, |\, K\geq H\}$ of $\cC$, 
hence it is contractible (see \cite{GrodalSmith}). 
However, in general the orbit space $C_G(H)\backslash X^H$ is not $\bbZ _{(p)}$-acyclic (see Example
\ref{ex:Sharpness}). So the Dwyer space $X=X_{\cC}^{\delta}$ does not satisfy the conditions of 
Theorem \ref{thm:HyperCohSS} for the fusion orbit category $\overline \cF _{\cC} (G)$.


\subsection{The Dwyer space for the centralizer decomposition} \label{sect:CentDecomp} 
The \emph{$\cC$-conjugacy category}  $\cA_{\cC} (G)$ is the category whose objects are 
pairs $(H, [i])$ where  $H$ is a group and $[i]$ is a $G$-conjugacy class of  monomorphisms 
$i : H \to G $ such that $i(H)\in \cC$. The $G$-action on a monomorphism $i: H \to G$ is defined 
by $g \cdot i := c_g \circ i$ for $g \in G$ where $c_g:G\to G$ is conjugation map $x\to gxg^{-1}$. 
A morphism $(H, [i])\to (H', [i'])$ is given by a group homomorphism $f: H \to H'$ such that $[i]=[i'\circ f]$.  

\begin{remark} The category $\cA _{\cC} (G)$ is not small. To take homotopy colimits over $\cA _{\cC} (G)$
we replace it with an 
equivalent category which is small. For example we can use the subcategory of $\cA_{\cC} (G)$ where the 
objects are pairs $(H, [i])$ where $H\in \cC$ and $[i]$ is a conjugacy class of monomorphisms 
$i: H \to G$ with $i(H)\in \cC$ as above.
\end{remark}
 
\begin{definition}\label{def:DwyerCent} Let $\widetilde \alpha _{\cC}: \cA _{\cC} (G) ^{op} \to \text{$G$-Sets}$  
denote the functor which takes 
$(H, [i])$ to the $G$-conjugacy class $[i]$.  The \emph{Dwyer space 
for the centralizer decomposition} over the collection $\cC$ is the $G$-space defined by
$$ X^{\alpha} _{\cC} :=\hocolim  _{\cA_{\cC} (G) ^{op} } \widetilde \alpha_{\cC}.$$  
\end{definition} 

The space $X_{\cC} ^{\alpha}$ is also denoted by $E \bA _{\cC}$. For every $(H, [i] ) $  in $\cA _{\cC} (G)$, we have 
$$\widetilde \alpha_{\cC} (H, [i]) \cong  G/C_G(i(H))$$ as $G$-sets. The Dwyer space $X_{\cC} ^{\alpha}$ is isomorphic  
to the nerve of the category $\bZ_{\cC} ^{\alpha} $ whose objects are the pairs $(H, i) $ 
where $H$ is a group and $i : H \to G$ is a monomorphism such that $i(H)\in \cC$. 
There is a unique morphisms from $(H', i')$ to $(H, i)$ if there is a group homomorphisms 
$f: H \to H'$ such that $i= i' \circ f$.  As it often done in the literature,
we will work with the 
opposite category $\bX_{\cC} ^{\alpha} :=(\bZ _{\cC} ^{\alpha} ) ^{op}$
where there is a unique morphisms from $(H, i)$ to $(H', i')$ if there is a group homomorphisms 
$f: H \to H'$ such that $i= i' \circ f$ (see \cite[Prop 7.12]{Dwyer-Book} and \cite[Thm 5.44]{BensonSmith-Book}). 
Note that the realization of the category
$\bX_{\cC} ^{\alpha}$ is $G$-homeomorphic to the realization of the Dwyer space $X_{\cC} ^{\alpha}$, 
so it does not affect the proofs to work with either category.

If $\alpha_{\cC} : \cA_{\cC} (G) ^{op} \to Spaces$ 
denote the functor $EG\times _G \widetilde \alpha _{\cC} $, then there is a natural isomorphism
$$\hocolim _{\cA _{\cC} (G) ^{op} } \alpha _{\cC}  \cong EG\times _G X_{\cC} ^{\alpha}.$$   
There is also a $G$-map $X_{\cC} ^{\alpha} \to \cN \cC$ which is a weak equivalence 
(see proof of \cite[Prop 7.12]{Dwyer-Book}). If $\cC$ is a $p$-ample collection, then the composition
$$\hocolim _{\cA _{\cC} (G) ^{op}} \alpha_{\cC}  \cong EG\times _G X_{\cC} ^{\alpha}  
\to EG\times _G \cN \cC \to BG$$ is a mod-$p$ homology isomorphism. This isomorphism 
is called the \emph{centralizer decomposition} for $G$ (see \cite[Prop 7.12]{Dwyer-Book}).
 For the Dwyer space $X_{\cC} ^{\alpha}$ we prove the following:

\begin{proposition}\label{pro:CentFixedPoints} Let $G$ be  a discrete group and 
$\cE$ denote the collection of all nontrivial elementary abelian $p$-subgroups in $G$. 
If $X=X^{\alpha}_{\cE}$ is the Dwyer space for the centralizer decomposition 
for $G$, then for every nontrivial finite $p$-subgroup $P \leq G$, the fixed point 
subspace $X^P$ and the orbit space $C_G(P)\backslash X^P$ are contractible.
\end{proposition}
 
\begin{proof} 
 A proof for the contractibility of $X^P$ can be found in \cite[\S 13.3]{Dwyer-Book}.  The argument given 
there for the contractibility of $X^P$ does not give a $C_G(P)$-equivariant contraction. We modify Dwyer's 
argument to obtain a $C_G(P)$-equivariant contraction.

Let $X=X^{\alpha} _{\cE}$ and $P$ be a fixed nontrivial finite $p$-subgroup of $G$. 
The fixed point subspace $|X|^P=|X^P|$ is $G$-homeomorphic to the realization of the category $\bC:=(\bX _{\cE} ^{\alpha} )^P$ 
whose objects are the pairs $(E, i)$ that satisfy $i(E) \leq C_G(P)$. Let $Z$ be a central subgroup of order $p$ in $P$. 
Then $Z$ is a subgroup of $C_G(P)$, and furthermore it is a central subgroup of $C_G(P)$ because $Z\leq P$.
Let $j: Z\to G$ denote the inclusion map and let $c_{(Z, j)} : \bC \to \bC$ denote the constant functor that takes 
every object $(E, i)$ in $\bC$ to $(Z, j)$. We will show that there is a zigzag of natural $C_G(P)$-transformations 
between $\id _{\bC}$ and the constant functor $c_{(Z, j)}$. By Proposition \ref{pro:Contractible}, this will imply 
that $C_G(P) \backslash X^P$ is contractible.

For every $(E, i)\in \bC$, the elementary abelian $p$-subgroups $i(E)$ and $Z$ commute with each other. 
In particular the product $i(E) \cdot Z$ is an elementary abelian $p$-subgroup of $C_G(P)$. Let $(E', i')$ be the pair such that
$$E':= \begin{cases} E&\text{if} \ Z\leq i(E)  \\ E \times Z&\text{if}\ Z \not \leq i(E) \end{cases} \qquad \qquad
i' :=\begin{cases} i&\text{if}\ E'=E, \\  (e, z)\to i(e)z&\text{if}\ E' \neq E. \\
\end{cases}
$$
Let $F: \bC \to \bC$ be the functor that sends the pair $(E, i)$ to $(E', i')$, and sends a morphism 
$f: (E_1, i_1) \to (E_2, i_2)$ to the morphism $f' : (E_1' , i_1')  \to (E_2' , i_2')$ where 
\begin{enumerate}
\item $f'=f$ if $E_1'=E_1$ and $E_2'=E_2$, 
\item $f': E_1 \times Z \to E_2$ is defined by $(e, z) \to f(e)i_2^{-1} (z)$ if $E_1' \neq E_1$ and $E_2'=E_2$,
\item $f'=f \times \id_Z$ if $E_1' \neq E_1$ and $E_2' \neq E_2$.
\end{enumerate}
Note that the case $E_1 '=E_1$ and $E_2'\neq E_2$ can not happen because the morphism 
$f: (E_1 , i_1) \xrightarrow{f} (E_2, i_2)$ is a group homomorphism $f: E_1 \to E_2$ such that 
$i_2 \circ f =i_1$, hence $i_1 (E_1) = i_2 (f(E_1))\leq i_2(E_2)$, so $Z \leq i_1(E_1)$ implies $Z\leq i_2(E_2)$.
 
There is a zigzag of natural transformations $$\id _{\bC}  \maprt{\mu}  F \maplf{\eta} c_{(Z, j)}$$
between the identity functor and the constant functor with value $(Z, j)$. The morphism 
$\mu_{(E, i)} : (E, i) \to (E', i')$ is defined by the group homomorphism
$$( \mu _{(E,i)} : E \to E' )= \begin{cases} \id_E : E \to E \text{  if  } 
E'=E, \\ E\to E\times Z \text{     defined by }  e\to (e,1) \text{  if  } E' \neq E. \\
\end{cases}
$$
It is easy to check that $\mu$ is a natural transformation. The only nontrivial case to check 
is when $E_1' \neq E_1$ and $E_2' =E_2$.
In this case, for every morphism $f: (E_1, i_1) \to (E_2, i_2)$, and for every $e\in E_1$, 
we have $$f' ( \mu _{(E_1 , i_1)} (e))=f' (e, 1) = f(e) i_2^{-1} (1)= f(e)=\mu_{(E_2, i_2)} ( f (e) ).$$
Hence, the equality $ f' \circ \mu _{(E_1, i_1)} = \mu _{(E_2, i_2)} \circ f$ holds in this case.

For each pair $(E, i)$, the morphism $\eta_{(E, i)}$ is defined by
the group homomorphism
$$( \eta _{(E,i)} : Z \to E' )= \begin{cases} i^{-1} |_Z  \text{  if  } E'=E, \\ Z \to E\times Z 
\text{     defined by } z \to (1, z) \text{  if  } E' \neq E. \\
\end{cases}$$
In the case $E_1'=E_1$ and $E_2'=E_2$, for each morphism $f: (E_1, i_1) \to (E_2, i_2)$, and for each $z\in Z$,
we have $f (i_1 ^{-1} (z) ) =i_2 ^{-1} (z)$, hence $f' \circ \eta_{(E_1, i_1)} = \eta _{(E_2, i_2)} \circ \id_{(Z, j)}$ holds.
It is easy to check that this equality holds in the remaining cases and conclude that $\eta$ is a natural transformation.
Hence $X^P$ is contractible.
 
It remains to show that $F$ is a $C_G(P)$-functor, and $\mu$ and $\eta$ are 
$C_G(P)$-equivariant transformations. Since $Z$ is central in $C_G(P)$, for every 
$g\in C_G (P)$, we have $Z\leq i(E) $ if and only if $Z\leq (c_g \circ i) (E)$. So $E'$ for $(E,i)$ 
and $E'$ for $(E, c_g \circ i)$ are equal. We also have $(c_g \circ i)'=c_g \circ i'$. This gives 
$$F(g \cdot (E, i) )=F( E, c_g \circ i )= (E' , (c_g \circ i)' )=(E', c_g \circ i')= g \cdot F(E, i)$$ for every 
$g\in C_G(P)$. This shows that the functor $F$ is $C_G(P)$-invariant on the objects. 
If the morphism $f: (E_1, i_1) \to (E_2, i_2)$ is defined by the group homomorphism $f: E_1 \to E_2$,
then $g\cdot f: (E_1, c_g \circ i_1) \to (E_2, c_g \circ i_2)$ is also defined by  $f: E_1 \to E_2$. 
Then the morphism $(g\cdot f)': (E_1', c_g \circ i')\to (E_2', c_g \circ f_2')$ is given by $f'$. 
This gives that $(g\cdot f)'= g \cdot f'$ for every $g \in C_G(P)$. Hence $F$ is a $C_G(P)$-functor.

For every $g \in C_G(P)$, the morphism $$\mu_{g \cdot (E, i)}= \mu _{(E, c_g \circ i) } : (E, c_g\circ i) \to (E' , (c_g\circ i)')=(E', c_g \circ i')$$ 
is equal to the morphism $$ g \cdot \mu _{(E, i)} : g \cdot (E, i) \to g \cdot (E', i').$$
Hence $\mu$ is a $C_G(P)$-transformation. For the transformation $\eta$, note that for every $g \in C_G(P)$, 
the morphism $$\eta _{g \cdot (E, i) }=\eta _{(E, c_g \circ i)} : (Z, j )\to (E', (c_g \circ i)')= (E', c_g \circ i')$$  is defined by the group  
homomorphism $(c_g \circ i) ^{-1} |_{Z} : Z \to E'$ if $E'=E$. Since $Z$ is central in $C_G(P)$, we have 
$$ (c_g \circ i)^{-1} |_{Z}=i^{-1} \circ c_{g^{-1}} |_{Z} =i^{-1} |_{Z},$$ hence $\eta _{g \cdot (E, i)}$ is equal to 
$$ g \cdot \eta _{(E, i)} : (Z, j) \to g\cdot (E', i')= (E', c_g \circ i').$$
In the case $E' \neq E$, both $\eta _{g \cdot (E, i)}$ and $g\cdot  \eta _{(E, i)}$ are given by 
the group homomorphism $Z\to E\times Z$ defined by $z\to (1, z)$. Hence $\eta$ is $C_G(P)$-equivariant. 

By Proposition \ref{pro:Contractible}, we conclude that the orbit space $C_G(P)\backslash X^P$ 
is contractible.
\end{proof}


\subsection{Dwyer space for the normalizer decomposition}

Let $K_{\cC}$ denote the order complex of the poset $\cC$. 
There is a $G$-action on the poset $\cC$ by conjugation, which makes $K_{\cC}$ a 
$G$-simplicial complex. The simplices of $K_{\cC}$ forms a poset with order relation given by
$\tau \leq \sigma$ if $\tau$ is a face of $\sigma$. We denote this poset by $sd (K_{\cC})$, 
and its opposite poset by $Sd_{\cC} (G)$. There is a $G$-action on 
$Sd_{\cC} (G)$, and the stabilizer of the simplex $\sigma := (H_0 < H_1 < \cdots < H_n)$ is the subgroup
$N_G(\sigma)= \bigcap _{i=0} ^n N_G(H_i).$

\begin{definition}\label{def:OrbitSimplices} The category of \emph{orbit simplices}, denoted by $Sd_{\cC} (G)/G$, 
is the category whose objects are the $G$-orbits $[\sigma]$ of the simplices of $K_{\cC}$. 
Between two objects there is a unique morphism $[\sigma] \to [\tau]$ if there is an element 
$g\in G$ such that $\tau $ is a face of $g \sigma$. 
\end{definition}
 
Note that $Sd_{\cC} (G)/G$ is a poset category. There is a functor 
$\widetilde \delta _{\cC} : Sd_{\cC} (G)/G \to G\text{-Sets}$ which takes a $G$-orbit $[\sigma]$ 
to itself as a $G$-set. We can describe the effect of $\widetilde \delta _{\cC} $ on morphisms as follows:
Choose a representative in each $G$-orbit of simplices and write $[\sigma]$ for the $G$-orbit 
whose representative is $\sigma$. The stabilizer of $\sigma =(H_0 < H_1 < \cdots < H_n )$ 
under the $G$-action is $N_G(\sigma)= \bigcap _{i=0} ^n N_G(H_i).$ Given a morphism 
$[\sigma] \to [\tau]$, let $g \in G$ be such that $\tau$ is a face of $g \sigma$. This gives 
that $gN_G(\sigma) g^{-1} = N_G (g\sigma ) \leq N_G (\tau)$. Hence there is a $G$-map 
$f: [\sigma] \to [\tau]$ defined by $f( g' \sigma) =g'g^{-1} \tau$ for all $g'\in G$. 
Note that the $G$-map $f$ does not depend on the group element $g\in G$. 
If $g_1, g_2 \in G$ are such that $\tau$ is a face of $g_1 \sigma$ and $g_2 \sigma$, then
we must have $g_1 ^{-1} \tau =g_2^{-1} \tau$ since they are both subchains of $\sigma$ 
with subgroups in both chains having the same order. Hence the $G$-map
$f : [\sigma] \to [\tau]$ is well-defined.  It is clear that $\widetilde \delta _{\cC}:Sd_{\cC} (G)/G \to G\text{-Sets}$
with these assignments defines a functor.

\begin{definition}\label{def:DwyerNorm} 
The \emph{Dwyer space 
for the normalizer decomposition} is the $G$-space
$$X_{\cC} ^{\delta}:=\hocolim _{ Sd_{\cC}(G)/G} \widetilde \delta.$$ 
\end{definition}

The space $X_{\cC} ^{\delta}$ is  isomorphic to the nerve of the poset category 
$\bX _{\cC} ^{\delta}$ whose objects are simplices of $K_{\cC}$, and there is one morphism 
$\sigma \to \tau$ if $\tau$ is a subchain of $\sigma$. The category $\bX_{\cC} ^\delta$ is the opposite category 
of the simplex category of $K_{\cC}$. This gives that $|X_{\cC} ^{\delta}|$ 
is $G$-homeomorphic to $|sd (K_{\cC} )|$. By a standard result on barycentric subdivisions, $|K_{\cC}|$ is 
$G$-homeomorphic to $|sd (K_{\cC})|$ 
(see \cite[Lemma 1.6.4]{Smith-Book}). Hence $|X_{\cC} ^{\delta}|$ is $G$-homeomorphic to  $|K_{\cC} |$.
 
\begin{lemma}\label{lem:Realization} Let $|\cC|$ denote the realization of the simplicial set $\cN \cC$ of the poset $\cC$.
Then $|X_{\cC} ^{\delta}|$ is $G$-homeomorphic to $|\cC|$.
\end{lemma} 

\begin{proof} $\cN \cC$ is the simplicial set for the simplicial complex $K_{\cC}$, so we have a $G$-homeomorphism
$|\cC|\cong |K_{\cC}|$. Hence by the remarks above $|X_{\cC} ^{\delta}|$ and $|\cC|$ are $G$-homeomorphic.
\end{proof}

If $\delta_{\cC} : Sd_{\cC} (G)/G \to Spaces$ 
denote the functor $EG\times _G \widetilde \delta _{\cC} $, then there is a natural isomorphism
$$\hocolim _{Sd_{\cC} (G) /G } \delta _{\cC}  \cong EG\times _G X_{\cC} ^{\delta}.$$   
If $\cC$ is a $p$-ample collection, then the composition
$$|\hocolim _{Sd_{\cC} (G) /G } \delta _{\cC}|  \cong EG\times _G |X_{\cC} ^{\delta}| \to EG\times _G |\cC| \to BG$$
is a mod-$p$ homology isomorphism  (see \cite[Prop 7.17]{Dwyer-Book}). This mod-$p$ homology isomorphism is 
called the \emph{normalizer decomposition} for $G$.
 
 We are now going to prove an important property of the Dwyer space $X_{\cC} ^{\delta}$. We first give a definition.

\begin{definition}\label{def:TakingProducts} We say a collection $\cC$ of subgroups in $G$ is \emph{closed under taking products} 
if it satisfies the following condition: if $P, Q\in \cC$ such that $PQ$ is a subgroup in $G$, then $PQ\in \cC$.
\end{definition}

We prove the following:
 
\begin{proposition}\label{pro:NormFixedPoints} Let  $\cC$ be a collection of subgroups of $G$ 
closed under taking products, and let $X=X^{\delta}_{\cC}$ be the Dwyer space for the 
normalizer decomposition for $G$. Then for every $P \in \cC$, the fixed point set $X^P$ and 
the orbit space $C_G(P)\backslash X^P$ are contractible.
\end{proposition}

\begin{proof} By Lemma \ref{lem:Realization}, the realization $|X|$ is $G$-homeomorphic 
to $|\cC|$, and hence it is enough to prove the statements for the $G$-space $|\cC|$.
For each $P\in \cC$, we have $|\cC|^P=|\cC^P |$, so we need to show that $\cC^P$ is a contractible poset.  
If $Q\in \cC ^P$, then $P\leq N_G(Q)$, hence $PQ$ is a subgroup in $G$.  Since $\cC$ is closed under taking products, 
$PQ \in \cC$. Since $P$ normalizes  $Q$, it normalizes $PQ$. Hence $PQ \in \cC ^P$. There is a zigzag 
of poset contractions $P \leq PQ \geq Q$ in $\cC ^P$, hence the poset $\cC^P$ is canonically contractible. 
The poset maps in the above contractions are $C_G(P)$-equivariant maps, hence $|\cC|^P$ is 
$C_G(P)$-equivariantly contractible. We conclude that the orbit space $C_G(P) \backslash X^P$ is contractible.
\end{proof}


\section{Hypercohomology spectral sequences for Dwyer spaces}\label{sect:SpecDwyer}

As an immediate consequence of  Theorem \ref{thm:HyperCohSS} 
and  Propositions 
\ref{pro:CentFixedPoints}  and \ref{pro:NormFixedPoints}, we obtain spectral sequences 
which converge to the cohomology of the orbit category and the fusion orbit category. 
We can state this conclusion as a corollary as follows:

\begin{corollary}\label{cor:HyperCohSSDwyer} 
Let $\bO_\cC (G)=\cO _{\cC}  (G)$ or $\overline \cF _{\cC} (G)$, and let $M$ be an $R\bO_\cC (G)$-module.
\begin{enumerate}
\item Let $X=X_{\cE} ^{\alpha}$ be the Dwyer space for the centralizer decomposition over the collection $\cE$
of nontrivial elementary abelian $p$-subgroups, and $\cC$ be any collection of nontrivial $p$-subgroups of $G$, or
\item let $\cC$ be a collection of $p$-subgroups of $G$ that is closed under taking products, and 
$X=X_{\cC} ^{\delta}$ be the Dwyer space for normalizer decomposition over $\cC$.
\end{enumerate}
For $j \geq 0$, let $\cH _M ^j$ denote the $R\cO (G)$-module defined in Definition \ref{def:HMj}.
Then there is a spectral sequence
$$E_2 ^{s,t} = H^s _{\cO(G)} (X^?; \cH^t _M) \Rightarrow H^{s+t} ( \bO_\cC (G) ; M).$$
\end{corollary}

Our aim in this section is to show that the Bredon cohomology groups appearing in the $E_2$-term 
of the above spectral sequences can be expressed as higher limits over the fusion category  
$\cF _{\cE} (G)$ in the first case, and over the category of orbit simplices $Sd _{\cC} (G)/G$ 
in the second case. To show these we need to introduce more definitions on equivariant local 
coefficient systems. We follow the terminology introduced by Grodal in \cite{Grodal-Higher} 
and \cite[\S 2.4]{Grodal-Endo}. 


\subsection{Local coefficient systems}

Given a simplicial set $X$, let $\Delta X$ denote the simplex category whose objects are simplices 
of $X$ and morphisms are compositions of face maps $d_i: X_n \to X_{n-1}$ and degeneracy maps 
$s_i: X_n \to X_{n+1}$. If $X$ is a $G$-simplicial set, 
the simplex category $\Delta X$ is a $G$-category. Let $(\Delta X)_G$ denote the transporter category 
of the $G$-category $\Delta X$. This is the category whose objects are simplices of $X$ and morphisms 
from $\sigma$ to $\tau$ are given by pairs  $(g, \varphi: g\sigma \to \tau)$ where $g\in G$ and $\varphi$ 
is a morphism in $\Delta X$. For a commutative ring $R$, a general (contravariant) $G$-local coefficient 
system on $X$ is a functor $\cM :  ( (\Delta X)_G)^{op}  \to R$-mod.  The cochain complex $(C^* (X, \cM), \delta)$ 
is defined by $$C^n (X; \cM )=\prod _{\sigma  \in X_n } \cM (\sigma)   \text{   and  }  (\delta ^n f) (\sigma ) 
= \sum _i (-1) ^i \cM ( 1, d_i ) ( f(d_i \sigma ) )$$ for every $f \in C^n (X ; \cM)$, where $d_i$ denotes 
the $i$-th face map $d_i: X_n \to X_{n-1}$  and $(1, d_i)$ denotes the morphism 
$(1, d_i: \sigma \to d_i \sigma ) $ in $(\Delta X)_G $. Note that in the above formula, 
elements of $C^n(X; \cM)$ are considered as a function $f : X_n \to \prod _{\sigma \in X_n} \cM (\sigma)$ 
such that $f(\sigma )\in \cM(\sigma)$. 
The (right) $G$-action on $f\in C^n (X; \cM)$ is defined by $$(fg) (\sigma)=\cM ( g, \id_{g\sigma}) (f(g\sigma))$$ 
for every $g\in G$, where   $\cM (g, \id _{g\sigma} ) : \cM (g\sigma) \to \cM(\sigma)$ is the $R$-module 
homomorphism induced by the morphism $(g, \id_{g\sigma}): \sigma \to g\sigma$ in $(\Delta X)_G$.  
We can convert this action to a left action by taking $gf= fg^{-1}$. With this definition of $G$-action, 
the chain complex  $C^* (X; \cM)$ becomes a chain complex of left $RG$-modules. The $G$-equivariant 
cohomology $H^* _G (X; \cM)$ of a $G$-space $X$ with coefficients in $\cM$ is defined to be the cohomology 
of the cochain complex $(C^* (X; \cM)^G, \delta ^*)$.

Given an $R \cO(G)$-module $M$, we can define a coefficient system $\cM:  ((\Delta X)_G ) ^{op} \to R$-Mod 
associated to $M$ as follows: Let $F: (\Delta X)_G \to \cO (G)$ be the functor that sends the simplex $\sigma$ 
to its stabilizer $G_{\sigma}$, and sends a morphism $(g, \varphi :g\sigma \to \tau) : \sigma \to  \tau$ to the $G$-map 
$$G/G_{\sigma} \maprt{g^{-1}} G/ G_{g\sigma} \maprt{\varphi} G/G_{\tau}$$  (see \cite[\S 2.4]{Grodal-Endo}).  
We define $\cM$ to be the composition $((\Delta X)_G ) ^{op} \maprt{F^{op} } \cO (G) ^{op}  \maprt{M} R\text{-Mod}$.
The coefficient system obtained this way is called the \emph{$G$-isotropy coefficient system} associated to $M$. 
The following is proved by Bredon  \cite[\S I.9]{Bredon-Equivariant}.

\begin{lemma}\label{lem:EqualBredon} Let $M$ be an $R\cO (G)$-module and $\cM$ be the coefficient system 
associated to $M$. Then for every $G$-simplicial set $X$, the $G$-equivariant cohomology $H^* _G (X; \cM )$ 
is isomorphic to the Bredon cohomology $H^*_{\cO(G)} (X^? ; M)$ defined in Definition \ref{def:BredonCoh}.
\end{lemma}

Let $\bC$ be a $G$-category. The transporter category of $\bC$ is a category $\bC_G$ whose objects are 
the same as the objects of $\bC$ are morphisms from $x\to y$ are given by pairs $(g, \varphi : gx\to y)$ 
where $g\in G$, and $\varphi$ is a morphism in $\bC$. If $X=\cN ( \bC)$ is the nerve of the $G$-category $\bC$, 
and $M : ( \bC_G )^{op} \to R$-mod is a functor, then there is a $G$-local coefficient system $\cM$ on $X$ 
defined via the functor $(\Delta X)_G \to \bC_G$ which takes $x_0 \to \cdots \to x_n$ to $x_0$. Similarly 
for a functor $M: \bC_G\to R$-mod, we can define a coefficient system via the functor 
$((\Delta X)_G ) ^{op} \to \bC_G$ which takes $x_0 \to \cdots \to x_n$ to $x_n$.


\subsection{Spectral sequence for the centralizer decomposition}

Let  $\cA_{\cC}  (G)$ denote the conjugacy category of $G$ over the collection $\cC$ as defined 
in Section \ref{sect:CentDecomp}. Recall that the fusion category $\cF _{\cC} (G)$ of $G$ over $\cC$ 
is the category whose objects are subgroups $H \in \cC$ where morphisms $H \to K$ in $\cF _{\cC} (G)$ 
are given by conjugation maps $c_g : H \to K$ defined by $c_g (h) =ghg^{-1}$ for all $h \in H$.  
The following is well-known.

\begin{lemma} The conjugacy category $\cA _{\cC} (G)$ is equivalent to  $\cF _{\cC} (G)$ as categories. 
\end{lemma}

\begin{proof} For each $G$-conjugacy class $[i]$ with $i(H) \in \cC$, choose a representative monomorphism 
$i: H \to G$. Let $T: \cA _{\cC} (G) \to \cF_{\cC} (G)$ be the functor defined by $(H, [i]) \to i(H)$ on objects. 
For every morphism $f: (H_1, [i_1]) \to (H_2, [i_2])$, we define $T(f) : i_1 (H_1) \to i_2 (H_2)$ to be the 
composition $$i_1 (H_1) \maprt{c_g} i_2 ( f(H_1) ) \to i_2 (H_2)$$ where the first conjugation map $c_g$ 
exists because $[i_2 \circ f]=[i_1]$, and the second map is the inclusion map. There is a functor 
$S: \cF _{\cC} (G) \to \cA _{\cC} (G) $ in the other direction defined by $S(H) = (H, [inc_H])$ for every 
$H \in \cC$. It is clear that $T \circ S = \id _{\cF_{\cC} (G)}$. There is a natural transformation 
$\eta: \id _{\cA_{\cC} (G)} \to S \circ T $ such that the morphism  $$\eta _{(H, [i])}: (H, [i]) \to (i(H), [inc_H])$$ 
is given by the group homomorphism $i: H \to i(H)$. It is straightforward to check that $\eta$ is a natural 
isomorphism. Hence these two categories are equivalent.
\end{proof}

The Dwyer space $X_{\cC} ^{\alpha}$  is the realization of the $G$-category $\bX_{\cC} ^{\alpha }$ 
whose objects are the pairs $(H, i)$ where $i$ is a monomorphism $i : H \to G$ with $i(H) \in \cC$. 
Let $X$ denote the nerve of the category $\bX _{\cC} ^{\alpha}$. The simplex category $\Delta X$ 
of $X$ is the category whose objects are chains of morphisms $(H_0, i_0)  \maprt{\alpha_1} (H_1, i_1)
\to \cdots \maprt{\alpha_n} (H_n, i_n ) $ in $\bX_{\cC } ^\alpha$ and whose morphisms are given 
by face and degeneracy maps of the nerve construction. There is a functor $F: ( (\Delta X) _G )^{op} 
\to \cF_{\cC} (G)$ defined by the composition of functors 
$$( (\Delta X) _G )^{op}  \to ( ( \Delta \cN \cC) _G )^{op} \to \cC_G =\cT _{\cC} (G) \to \cF _{\cC} (G)$$
where the first functor takes a simplex
$(H_0, i_0)  \maprt{\alpha_1} (H_1, i_1)\to \cdots \maprt{\alpha_n} (H_n, i_n )$ in $X$ to the simplex
$i_0(H_0) \leq \cdots \leq i_n(H_n)$ in $\cN \cC $, and the second functor takes this chain of subgroups 
to $i_n (H_n)$ in $\cC$ (see \cite[p. 416]{Grodal-Higher} for more details). The equality 
$\cC_G=\cT _{\cC} (G)$ follows from definitions, and the last functor $\cT_{\cC} (G) \to \cF_{\cC} (G)$ 
is the quotient functor defined in Section \ref{sect:Orbit}. 
The  following is proved by Grodal in \cite{Grodal-Higher}.

\begin{proposition}[{\cite[Prop 2.10]{Grodal-Higher}}]\label{pro:FirstStep}
Let $N: \cF _{\cC } (G) \to R\text{-Mod}$ be a (covariant) functor, and let $\cM$ denote the 
local coefficient system defined on $X$ by the composition
$$((\Delta X) _G )^{op}  \maprt{F} \cF  _{\cC} (G) \maprt{N} R\text{-Mod}.$$
Then there is an isomorphism 
$$ H^* _G ( X^{\alpha}_{\cC} ; \cM ) \cong \underset{\cF_{\cC} (G)}{\lim {}^*} \ N.$$
\end{proposition}

There is a functor $\xi _{\cC}: \cF_{\cC} (G) \to \cO(G) ^{op}$ which sends $H \in \cC$ to its centralizer 
$C_G(H) \leq G$. Given an $R\cO (G)$-module $M$, precomposing $M$ with $\xi _{\cC}$ gives a 
functor $M \circ \xi _{\cC}: \cF _{\cC} (G)\to R\text{-Mod}$.  Combining Proposition \ref{pro:FirstStep} with our earlier 
observations, we obtain the following:

\begin{proposition}\label{pro:Step2} 
Let $M$ be an $R\cO (G)$-module and  let $X=X_{\cC} ^{\alpha}$. Then there 
is an isomorphism $$H^* _{\cO(G)}  ( X^?; M)\cong \underset{\cF_{\cC} (G)}{\lim {}^*} (M \circ \xi _{\cC} ).$$
\end{proposition}
 
\begin{proof} Let $\cM$ be the coefficient system on $X$ by defined by the composition
$$M \circ \xi _{\cC} \circ F:  ( (\Delta X) _G )^{op}  \to ( ( \Delta \cN \cC) _G )^{op} \to \cC_G \to \cF _{\cC} (G) 
\maprt{\xi _{\cC}} \cO (G) ^{op} \maprt{M} R\text{-Mod}.$$
By Proposition \ref{pro:FirstStep}, there is an isomorphism 
\begin{equation}\label{eqn:Isom1}
H^* _G( X ; \cM ) \cong \underset{\cF_{\cC} (G)}{\lim {}^*} (M \circ \xi _{\cC} ).
\end{equation}
The composition $\xi _{\cC} \circ F$ takes the
simplex 
$$\sigma :=(H_0, i_0)  \maprt{\alpha_1} (H_1, i_1)\to \cdots \maprt{\alpha_n} (H_n, i_n )$$ in $X$
to the subgroup $C_G (i_n (H_n ) )$ in $\cO (G)$, which is the stabilizer of $\sigma$ under the 
$G$-action on $\sigma$. This implies that $\cM=M \circ \xi _{\cC} \circ F$ is the $G$-isotropy 
coefficient system for $X$ associated to $M$. Hence by Lemma \ref{lem:EqualBredon}, 
there is an isomorphism $$H^* _{\cO (G)} (X^?; M) \cong H^* _G( X ; \cM ).$$ 
Combining this isomorphism with the isomorphism in $(\ref{eqn:Isom1})$ gives the desired isomorphism.
\end{proof}

Let $\cC$ and $\cD$ be two arbitrary collections in $G$. 
Let $\bO_\cC (G)$ denote either $\cO_{\cC} (G)$ or $\overline \cF_{\cC} (G)$ and $M$ denote 
an $R\bO_\cC (G)$-module. For every integer $j \geq 0$, let
$\cH _M ^j $ denote the $R\cO(G)$-module defined in Definition \ref{def:HMj}.
Let $\xi _{\cD}: \cF_{\cD} (G) \to \cO (G) ^{op}$ be the functor defined above which sends $D\in \cD$ to
the centralizer $C_G(D) \leq G$. Precomposing $\cH ^j _M$ with $\xi_{\cD}$ gives a 
functor $\cF_{\cD} (G) \to R\text{-Mod}$. 

\begin{definition}\label{def:HCG} For an $R\bO_{\cC} (G)$-module $M$ and for $j \geq 0$,
the functor $$\cH ^j _{M, C_G} : \cF_{\cD} (G) \to R\text{-Mod}$$ is defined to be the composition 
$\cH ^j _M \circ \xi_{\cD}$. Note that for every $D \in \cD$, 
$$\cH ^j _{M, C_G} (D)= H^* (\bO _{\cC} (C_G (D)) ; \Res ^{\bO _{\cC} (G)} _{\bO _{\cC} (C_G (D) )} M ).$$
\end{definition} 

We have the following result.
 
\begin{proposition}\label{pro:CentDecOrb}  Let $\cC$ be any collection of all nontrivial $p$-subgroups in $G$, 
and $\cE$ be the collection of all nontrivial elementary abelian $p$-subgroups in $G$.
Let $\bO_{\cC} (G) =\cO _{\cC}  (G)$ or $\overline \cF _{\cC} (G)$, 
and $M$ be an $R\bO_{\cC} (G)$-module. For each integer $j\geq 0$,  let $\cH ^j _{M, C_G} $ denote the
functor defined in Definition \ref{def:HCG}. Then, there is a spectral sequence
$$E_2 ^{s,t} = \underset{\cF_{\cE} (G)}{\lim {}^s} \ \cH _{M, C_G} ^t  \Rightarrow H^{s+t} (\bO_{\cC} (G)  ; M).$$  
\end{proposition} 

\begin{proof}  Since $\cH _{M, C_G} ^t =\cH _M ^t \circ \xi _{\cE}$,  this follows from 
Corollary \ref{cor:HyperCohSSDwyer} and Proposition \ref{pro:Step2}.
\end{proof}


\subsection{Spectral sequence for the normalizer decomposition}

Let $G$ be a discrete group and $\cC$ be a collection of subgroups in $G$ such that $\cC$ is closed 
taking products. Let $\bO_{\cC} (G)=\cO _{\cC} (G)$ or $\cF_{\cC} (G)$, 
and  $X:=X_{\cC} ^{\delta}$ denote the Dwyer space for the normalizer decomposition over $\cC$.
By Corollary \ref{cor:HyperCohSSDwyer} there is a spectral sequence
$$E_2 ^{s,t} = H^s _{\cO(G)} (X^?; \cH^t _M) \Rightarrow H^{s+t} ( \bO_\cC (G) ; M)$$
where $\cH _M ^t$ is the $R\cO (G)$-module defined in Definition \ref{def:HMj}. 
By Lemma \ref{lem:Realization},  $X$ is $G$-homeomorphic to the geometric realization $|\cC| $ of $\cC$, 
hence we can replace $X^?$ with $|\cC|^?$ in the above spectral sequence. 
Note that $|\cC|$ is the realization of the simplicial set $\cN \cC$ and there is a functor
$$\eta _{\cC}: (\Delta \cN \cC)_G  \to Sd_{\cC} (G)/G$$ which takes each simplex $\sigma$ in $\cN\cC$ to its $G$-orbit 
$[\sigma]$, and each morphism $$(g, \varphi: g\sigma \to \tau) : \sigma \to \tau$$ in $(\Delta \cN \cC)_G$
to the unique morphism $[\sigma] \to [\tau]$ in $Sd _{\cC} (G)/G$.
For higher limits over the subdivision category we have the following result which is attributed 
to Slomi{\' n}ska \cite{Slominska-SpectralSeq} 
by Grodal (see {\cite[Prop 7.1]{Grodal-Higher}} for a proof).
 
\begin{proposition}\label{pro:Subdiv} Let 
$N: (Sd _{\cC} (G)/G)^{op} \to R\text{-Mod}$ be an arbitrary functor. Then
$$ \underset{Sd _{\cC} (G)/G}{\lim {}^*} N \cong H_G ^* (|\cC| ; \cM )$$
where $\cM$ is the $G$-local coefficient system given via $\eta_{\cC} ^{op}: ( (\Delta \cN \cC)_G )^{op} \to (Sd_{\cC} (G)/G)^{op}$.
\end{proposition}

The definition of the functor $\widetilde \delta _{\cC} : Sd_{\cC} (G) /G \to G\text{-Sets}$ 
can be adjusted to define a functor 
$$\zeta _{\cC} : Sd _{\cC} (G)/G \to \cO (G)$$ which takes a $G$-orbit $[\sigma]$ 
to the normalizer $N_G (\sigma)$. Given a morphism 
$[\sigma] \to [\tau]$, there is a $g \in G$ be such that $\tau$ is a face of $g \sigma$. 
This gives that $gN_G(\sigma) g^{-1} = N_G (g\sigma ) \leq N_G (\tau)$.
Hence there is a $G$-map $f: G/N_G(\sigma)  \to G/N_G (\tau) $ defined by 
$f( g' N_G (\sigma) ) =g'g^{-1} N_G (\tau)$ for all $g'\in G$. The $G$-map $f :[\sigma]\to [\tau]$ 
does not depend on the group element $g \in G$, hence $\zeta _{\cC}$ with these assignments 
defines a functor.  

\begin{proposition}\label{pro:Subdiv2} 
Let $M$ be an $R\cO (G)$-module. 
Then there is an isomorphism
$$H_{\cO (G)}  ^* (|\cC|^? ; M) \cong \underset{Sd _{\cC} (G)/G}{\lim {}^*} (M \circ \zeta_{\cC} ^{op})  .$$ 
\end{proposition}

\begin{proof} By Proposition \ref{pro:Subdiv}, we have 
$$\underset{Sd _{\cC} (G)/G}{\lim {}^*} (M \circ \zeta _{\cC} ^{op})  \cong H^* _G (|\cC|; \cM),$$ where 
the coefficient system $\cM$ on $|\cC|$ is defined by the composition
$$((\Delta \cN \cC )_G ) ^{op} \xrightarrow{\eta _{\cC} ^{op} } (Sd _{\cC} (G) /G) ^{op}  
\xrightarrow{\zeta_{\cC} ^{op}} \cO (G) ^{op} \xrightarrow{M} R\text{-Mod}.$$
Note that $\zeta _{\cC}  \circ \eta _{\cC}$ sends a simplex $\sigma$ to $N_G(\sigma)$ and
a morphism $(g, \varphi :g\sigma \to \tau) : \sigma \to  \tau$ in $((\Delta \cN \cC )_G ) ^{op} $
to the $G$-map 
$$G/N_G(\sigma) \maprt{g^{-1}} G/ N_G (g\sigma) \maprt{\varphi} G/N_G(\tau).$$ 
Hence $\cM$ defines the $G$-isotropy coefficient
system for $|\cC|$. Applying Lemma \ref{lem:EqualBredon}, we obtain 
the desired isomorphism.
\end{proof}

Let $M$ be an $R\bO_\cC (G)$-module $M$. For every integer $j \geq 0$, let 
$\cH _M ^j : \cO (G)^{op} \to R\text{-Mod}$ denote the functor defined in Definition \ref{def:HMj}.
 
\begin{definition}\label{def:HNG} Precomposing $\cH _M ^j$  with the functor 
$\zeta _{\cC} ^{op} : (Sd_{\cC } (G) /G) ^{op} \to \cO (G) ^{op}$, we obtain a functor 
$$\cH ^j _{M, N_G} : (Sd_{\cC } (G) /G) ^{op} \to R\text{-Mod}$$ such that for every $[\sigma] \in Sd_{\cC} (G)/G$, 
we have
$$\cH ^j _{M, N_G} ([\sigma])= H^* (\bO _{\cC} (N_G (\sigma )) ; \Res ^{\bO _{\cC} (G)} _{\bO _{\cC} (N_G (\sigma) )} M ).$$
\end{definition}
 
As a consequence of the results above, we obtain the following spectral sequence.
 
\begin{proposition}\label{pro:NormalizerSS} Let $G$ be a discrete group and $\cC$ be a collection subgroups 
of $G$ closed under taking products. Let $\bO_{\cC} (G)=\cO _{\cC}  (G)$ or $\overline \cF _{\cC} (G)$, and $M$ 
be an $R\bO_{\cC} (G) $-module. For each integer $j\geq 0$,  let $\cH ^j _{M, N_G} $ denote the
functor defined in Definition \ref{def:HNG}. Then there is a spectral sequence
$$E_2 ^{s, t} = \underset{Sd _{\cC} (G) /G }{\lim {}^s}  \cH ^t _{M, N_G}  \Rightarrow H^{s+t} (\bO _{\cC} (G)  ; M).$$
\end{proposition}

\begin{proof} This follows from Corollary \ref{cor:HyperCohSSDwyer} and Proposition \ref{pro:Subdiv2}.
\end{proof}

In the next section we consider centralizer and normalizer fusion systems. The results we prove  will allow 
us to express the spectral sequences we obtained in Propositions \ref{pro:CentDecOrb}  and \ref{pro:NormalizerSS} 
in terms of the cohomology of the centric orbit category of the centralizer and normalizer fusion systems.


\section{The centralizer and normalizer fusion systems}\label{sect:CentNormFusion}

If $G$ is an infinite group with a Sylow $p$-subgroup $S$, the subgroups of $G$ may not 
have Sylow $p$-subgroups. This makes it difficult to work with centralizers and normalizers 
of $p$-subgroups of $G$ when $G$ is an infinite group. We overcome this difficulty by using 
a result due to Libman  \cite[Prop 3.8]{Libman-WebbConj}  
(see also Parker \cite[Lemma 2.14]{Parker}) which states that if $G$ is a discrete group with 
a Sylow $p$-subgroup $S$ such that $\cF_S(G)$ is a saturated fusion 
system then the normalizer $N_G(P)$ of a $p$-subgroup $P$ always has a Sylow $p$-subgroup. 
In this section we prove a generalization of Libman's theorem to $K$-normalizer subgroups so that 
we can also apply it to centralizer fusion systems and to the normalizers of chains of subgroups in $S$.
  
Let $G$ be a discrete group and $S$ be a Sylow $p$-subgroup of $G$. For every subgroup $Q\leq S$ 
and every subgroup $K \leq \Aut(Q)$, the \emph{$K$-normalizer of $Q$} in $G$ is the subgroup
$$N_G^K(Q):=\{ g\in N_G(Q)\, |\, (c_g)|_Q\in K\}.$$ Let $N_S^K (Q):=S\cap N^K_G(Q)$. 
If $K=\Aut(Q)$, then $N^K_G(Q)=N_G(Q)$ and if $K=1$, then $N_G^K(Q)=C_G(Q)$. Another interesting 
case is where $K$ is the subgroup of automorphisms of $Q$ which stabilizes a chain 
$\sigma:=(Q_0 < Q_1 < \cdots < Q_n)$ of subgroups of $Q$ such that $Q_n=Q$. In this case we write $N_G(\sigma)$ 
for $N_G^K (Q)$.
 
Let $\cF$ be a saturated fusion system over $S$, and let $K \leq \Aut(Q)$. We say $Q$ is \emph{fully $K$-normalized} 
if for any morphism $\varphi : Q \to S$ in $\cF$, $$|N _S ^K (Q)|\geq |N_S ^{ \leftexp{\varphi} K} (\varphi Q)|$$ where  
$\leftexp{\varphi}K:=\{\varphi  k \varphi^{-1} | \, k\in K\} \leq \Aut (\varphi Q)$. 

\begin{definition}\label{def:GenNormalizer}
The \emph{$K$-normalizer fusion system} $N^K_{\cF} (Q)$ is the fusion system over $N^K_S(Q)$ 
whose morphisms $P \to P'$ are the morphisms $\varphi \in \Mor _{\cF} (P, P')$ which extend to a morphism 
$\widetilde \varphi : QP\to QP'$ in $\cF$ in such a way that $\widetilde \varphi |_Q \in K$.
\end{definition} 
 
If $K=\Aut(Q)$, then $N_{\cF} ^K (Q)$ is denoted by $N_{\cF}(Q)$ and called the \emph{normalizer fusion system}.
If $K=1$, then $N_{\cF} ^K (Q)$  is denoted by $C_{\cF} (Q)$ and called the \emph{centralizer fusion system}.
 
\begin{theorem}[Puig \cite{Puig06}]\label{thm:Puig} Let $\cF$ be a saturated fusion system over a finite $p$-group 
$S$, and let $Q\leq S$ and $K \leq \Aut(Q)$. If $Q$ is fully $K$-normalized then $N^K_{\cF} (Q)$ is a saturated 
fusion system.
\end{theorem}

As special cases of Puig's theorem, we obtain that for every fully $\cF$-normalized subgroup $Q\leq S$, the normalizer 
fusion system $N_{\cF} (Q)$ is saturated, and for every fully $\cF$-centralized subgroup $Q$, the centralizer fusion system 
$C_{\cF} (Q)$ is saturated. We will need the following lemma in our proofs below.

\begin{lemma}[{\cite[Lemma 4.36]{Craven-Book}}]\label{lem:Technical} Let $\cF$ be a saturated fusion system 
over $S$. Let $Q\leq S$ and $K \leq \Aut  (Q)$. If $\phi : Q \to S$ is a morphism in $\cF$ such that  $\phi Q$ is 
fully $\leftexp{\phi} K$-normalized, then there is a morphism $\psi : Q N_S ^K (Q) \to S$ in $\cF$ and $\alpha \in K$ 
such that $\psi |_Q=\phi \circ \alpha$, and $\beta  \in \leftexp{\phi} K$ such that $\beta \circ  \psi |_Q=\phi$.
\end{lemma}

The following proposition is  a generalization of Libman's result \cite[Prop 3.8]{Libman-WebbConj} 
(see also \cite[Lemma 2.14]{Parker}).

\begin{proposition}\label{pro:SylowExists} Let $G$ be a discrete group and $S$ be a  Sylow $p$-subgroup 
of $G$. Suppose that $\cF=\cF_S(G)$ is a saturated fusion system. Let $Q \leq S$ and $K \leq \Aut(Q)$. 
Then, $Q$ is fully $K$-normalized if and only if $N^K_S(Q)$ is a Sylow $p$-subgroup of $N^K_G(Q)$.
\end{proposition}

\begin{proof}  Assume that $Q$ is fully $K$-normalized. We will show that for every finite $p$-subgroup 
$P$ of $N_G ^K (Q)$, there is an $x\in N_G^K (Q)$ such that $\leftexp{x} P \leq N_S^K (Q)$. Let $P$ 
be a finite $p$-subgroup of $N^K_G(Q)$, and let $R=QP$.  Since $P\leq N ^K _G (Q)\leq N_G (Q),$ 
we have $R \leq N_G(Q)$.  Let $g\in G$ such that $\leftexp{g} R \leq S$. Then we have 
$\leftexp{g} Q \leq \leftexp{g} R \leq S$. Let $\phi : \leftexp{g} Q \to S$ be the conjugation map 
$c_{g^{-1}} : \leftexp{g} Q \to S$ defined by $c_{g^{-1}} (x)=g ^{-1} xg$. Note that $\phi$ is a 
morphism in $\cF$. Let $L\leq \Aut (\leftexp{g} Q) $ be the subgroup defined by 
$L:= \leftexp{\phi ^{-1} } K$. Since $\phi (\leftexp{g} Q)=Q$, $\leftexp{\phi} L=K$, and $Q$ is fully 
$K$-normalized, by Lemma \ref{lem:Technical}, there is a morphism 
$$\psi :  \leftexp{g} Q N_S ^L ( \leftexp{g}Q) \to S$$ in $\cF$ and an automorphism $\beta \in K$ 
such that $\beta \circ \psi |_{ \leftexp{g} Q} = \phi$.
 
Since $\leftexp{g} P \leq \leftexp{g} R \leq S$ and $\leftexp{g} P \leq \leftexp{g} (N ^K_G (Q) )=
N_G ^L ( \leftexp{g} Q)$, we have $\leftexp{g} P \leq N_S ^L ( \leftexp{g} Q)$. This gives that 
$\leftexp{g} R=\leftexp{g} Q \leftexp{g} P \leq \leftexp{g} Q N_S ^L ( \leftexp{g} Q)$, 
hence $\psi $ is defined on $\leftexp{g} R$. Let $u \in G$ be such that $\psi=c_u$. Let $x=ug$. 
Then $$\leftexp{x} R=\leftexp{u}  ( \leftexp{g} R ) =\psi (\leftexp{g} R)\leq S.$$
We also have 
$$c_{x} |_Q =c_u \circ c_g |_Q = \psi \circ \phi^{-1} |_Q =\psi \circ (\beta \circ \psi |_{\leftexp{g} Q} )^{-1} 
=\beta^{-1}  \in K.$$ This gives that $x\in N_G^K(Q)$. Hence  $\leftexp{x} P \leq N_G ^K (Q)$. Since 
$\leftexp{x} P \leq \leftexp{x} R \leq S$, we have  $$\leftexp{x} P \leq S \cap N_G ^K (Q) =N_S ^K (Q).$$
This completes the proof that $N_S^K(Q)$ is a Sylow $p$-subgroup of $N_G^K(Q)$.
 
For the converse, suppose that $N^K_S(Q)$ is a Sylow $p$-subgroup of $N^K_G(Q)$. Let $\varphi: Q \to S$ 
be a morphism in $\cF$.  Let $g \in G$ such that $\varphi=c_g$. Then we have 
$$N_S ^{ \leftexp{\varphi} K } (\varphi Q) =N ^{ \leftexp{c_g} K} _S ( \leftexp{g} Q)=\leftexp{g} (N^K _{S^g} (Q) ).$$
Note that $N^K _{S^g} (Q)$ is a $p$-subgroup of $N_G ^K (Q)$. Since $N_S^K (G)$ is a Sylow $p$-subgroup 
of $N^K_G(Q)$, we have $|N^K_S(Q)| \geq |N_{S^g}  ^K (Q)|$. This gives 
$|N_S ^K (Q)| \geq |N_S ^{\leftexp{\varphi} K} (\varphi Q)|$, hence $Q$ is fully $K$-normalized.  
\end{proof}
 
The next proposition easily follows from Proposition \ref{pro:SylowExists}.

\begin{proposition}  Let $G$ be a discrete group and $S$ be a  Sylow $p$-subgroup of $G$. 
Let $Q \leq S$ and $K \leq \Aut(Q)$. If $\cF=\cF_S(G)$ is a saturated fusion system, then 
the $K$-normalizer subgroup $N_G^K (Q)$ has a Sylow $p$-subgroup.
\end{proposition}

\begin{proof}  Assume that $\cF$ is saturated. 
Let $\phi : Q \to S$ be a morphism in $\cF$ such that $N_S ^{\leftexp{\phi} K} (\phi Q)$  
has the maximum order among all such normalizers. Let $L=\leftexp{\phi}K$ and $R=\phi Q$.  
Then for every morphism $\varphi : R \to S$ in $\cF$, we have 
$$|N_S ^L (R) |= |N_S ^{\leftexp{\phi } K} ( \phi Q) |  \geq 
|N_S ^{\leftexp{\varphi \phi} K } ( \varphi \phi Q) |=| N_S ^{\leftexp{\varphi} L } (\varphi L )|,$$
hence $R$ is fully $L$-normalized. By Proposition \ref{pro:SylowExists}, we conclude that 
$N_S ^L (R)$ is a Sylow $p$-subgroup of $N_G ^L (R)$. If $g\in G$ such that $\phi =c_g$, then
$$N^L _G (R)=N^{\leftexp{c_g} K} _G ({}^g Q)= {}^g ( N^K _G (Q) ).$$ Since $N^L _G (R)$ has 
a Sylow $p$-subgroup, its conjugate $N^K_G(Q)$ also has a Sylow $p$-subgroup.  
\end{proof}

Another consequence of Proposition \ref{pro:SylowExists} is the following proposition which 
is a generalization of \cite[Prop 3.8]{Libman-WebbConj}. 
 
\begin{proposition}\label{pro:FusionSame} Let $G$ be a discrete group with a Sylow $p$-subgroup 
$S$ such that $\cF=\cF_S(G)$ is a saturated fusion system. Let $Q\leq S$ and $K\leq \Aut(Q)$. 
If $Q$ is fully $K$-normalized, then $N^K_{\cF} (Q)=\cF_{N^K_S(Q)} (N^K_G(Q))$. 
\end{proposition}

\begin{proof} The argument given in \cite[Thm 4.27]{Craven-Book} for finite groups also holds here. 
By Proposition \ref{pro:SylowExists}, $N^K_S(Q)$ is a Sylow $p$-subgroup of $N^K_G(Q)$. 
Let $\varphi: P \to P'$ be  a morphism in $N^K_{\cF} (Q)$.
Then it extends to $\widetilde \varphi : QP \to QP'$ in $\cF$ such that $\widetilde \varphi |_Q \in K$. 
Let $x\in G$ such that $\widetilde \varphi =c_x$. Since $\widetilde \varphi |_Q \in K$, we have 
$x\in N^K_G(Q)$. This implies that $\varphi =c_x |_Q $ is a morphism in 
$\cF _{N^K_S(Q) } (N^K_G(Q))$. Conversely, if $c_x: P \to P'$ is a morphism in  
$\cF _{N^K_S(Q) } (N^K_G(Q))$, then $c_x$ extends to a homomorphism $QP \to QP'$ 
defined also by conjugation with $x$, hence $c_x$ lies in $N^K_{\cF} (Q)$.
\end{proof}

A subgroup $P$ of $S$ is called  $\cF$-centric if $C_S(P') \leq P'$ for every 
$P' \sim _{\cF} P$. Note that if $P$ is $\cF$-centric, then $C_S(P')=Z(P')$ has the same order 
for every $P' \sim _{\cF} P$. Thus if $P$ is an $\cF$-centric subgroup then $P$ and all its 
$\cF$-conjugates are fully $\cF$-centralized.

\begin{lemma}\label{lem:CentricS} Let $G$ be a discrete group with a Sylow $p$-subgroup 
$S$ such that  $\cF=\cF_S(G)$ is a saturated fusion system. Let $P \leq S$. Then 
$P$ is $\cF$-centric if and only if $Z(P)$ is a Sylow $p$-subgroup of $C_G(P)$.
\end{lemma}

\begin{proof}  If $P$ is $\cF$-centric, then by the argument above, $P$ is fully $\cF$-centralized. Then
by Proposition \ref{pro:SylowExists}, $C_S(P)$ is a Sylow $p$-subgroup of $C_G(P)$.  
Since $P$ is $\cF$-centric, we have $C_S(P)=Z(P)$. Hence $Z(P)$ is a Sylow $p$-subgroup 
of $C_G(P)$. For the converse, assume that $Z(P)$ is a Sylow $p$-subgroup of $C_G(P)$.
For every $g\in G$, we have $C_G (\leftexp{g} P) = \leftexp{g} C_G (P)$ and $Z( \leftexp{g} P)
=\leftexp{g} Z(P)$, hence $Z( \leftexp{g} P)$ is a Sylow $p$-subgroup  of $C_G ( \leftexp{g} P)$. 
If $g\in G$ such that  $\leftexp{g} P \leq S$, we have 
$Z(\leftexp{g} P ) \leq C_S( \leftexp{g} P) \leq C_G( \leftexp{g} P)$.  Since
$Z(\leftexp{g} P )$ is a Sylow $p$-subgroup of $C_G( \leftexp{g} P)$, we obtain that
$C_S ( \leftexp{g} P) =Z( \leftexp{g} P)$. This implies that $P$ is $\cF$-centric. 
\end{proof}

Note that the argument above also proves that if $P$ is an $\cF$-centric subgroup, 
then for every $g \in G$, the center $Z(\leftexp{g} P)$ is  a Sylow $p$-subgroup of $C_G( \leftexp{g} P)$. 
This suggests that we extend the definition of $p$-centric subgroups for discrete groups in the following way.

\begin{definition} Let $G$ be a discrete group with a Sylow $p$-subgroup $G$.  A $p$-subgroup 
in $G$ is called \emph{$p$-centric} if $Z(P)$ is a Sylow $p$-subgroup of $C_G(P)$.
\end{definition}

We have the following observation.

\begin{lemma}\label{lem:CentricG} Let $\cF$ be a saturated fusion system over $S$, and 
$G$ be a discrete group with a Sylow $p$-subgroup $S$ such that $\cF\cong \cF_S(G)$. 
A $p$-subgroup $P \leq G$ is conjugate to an  $\cF$-centric subgroup in $S$ if and only if 
it is a $p$-centric subgroup of $G$. Hence the collection of all $p$-subgroups of $G$ conjugate 
to an $\cF$-centric subgroup in $S$ is equal to the collection of $p$-centric subgroups in $G$.
\end{lemma}

\begin{proof} This follows from Lemma \ref{lem:CentricS} and from the fact that for every $g\in G$, 
a subgroup $P \leq G$ is $p$-centric if and only if $\leftexp{g} P$ is $p$-centric. 
\end{proof} 

In Sections \ref{sect:CentralizerDec} and  \ref{sect:NormalizerDec}, we consider $\cF$-centric 
subgroups in centralizer and normalizer subgroups. In our proofs we will need following lemma.

\begin{lemma}\label{lem:NormCentric} Let $\cF$ be a saturated fusion system over a finite 
$p$-group $S$. Let $Q\leq S$ and $K \leq \Aut(Q)$. Assume that $Q$ is a fully $K$-normalized 
subgroup of $S$. Then for every $P \leq N_S^K (Q)$, the following hold: \\
(i) If $P$ is $\cF$-centric, then $P$ is $N_{\cF} ^K (Q)$-centric. \\
(ii) If $P$ is $N_{\cF}^K (Q)$-centric, then $QP$ is $\cF$-centric. 
\end{lemma}

\begin{proof}  Assume that $P$ is $\cF$-centric. Let $P'\leq N^K_S(Q)$ be such that $P'$ is isomorphic 
to $P$ in $N^K_{\cF} (Q)$. Then $P ' \sim _{\cF} P$, hence $C_S(P') \leq P'$ since $P$ is $\cF$-centric. 
From this we obtain that  $$C_{N^K_S(Q)} (P') \leq C_S(P') \leq P',$$ hence $P$ is $N^K_{\cF} (Q)$-centric. 
This proves (i).

To prove (ii), we use an argument similar to the argument given in the proof of  \cite[Lemma 6.2]{BLO2}. 
Assume that $P \leq N_S^K (Q)$ is $N^K_{\cF}(Q)$-centric.  Let $\varphi : QP \to S$ be a morphism in $\cF$. 
We need to show that $C_S(\varphi(QP))\leq \varphi(QP)$. 
Since $Q$ is fully $K$-normalized, we can apply Lemma \ref{lem:Technical} to $\varphi^{-1} : \varphi Q \to Q$ 
to obtain that there is a morphism $$\psi : \varphi Q \cdot N_S ^{\leftexp{\varphi} K }(\varphi Q) \to  S$$
in $\cF$ and an automorphism $\beta \in K$ such that $\beta  \circ  (\psi |_{\varphi Q} )= \varphi ^{-1}$. We have
$$\varphi (QP)=\varphi Q \cdot \varphi P \leq \varphi Q \cdot N_S ^{\leftexp{\varphi} K}(\varphi Q) $$ 
hence the composition $\psi \circ  \varphi : QP \to S$ is defined. Moreover $(\psi \circ \varphi) |_Q =\beta^{-1} \in K$.

We claim that $(\psi \circ \varphi) |_P : P \to S$ is a morphism in $N_{\cF} ^K (Q)$. We will first show that 
$(\psi \circ \varphi) (P) \leq N_S^K (Q)$. We write $\psi \varphi$ for $\psi \circ \varphi$ to simplify the notation. 
Let $x\in P$ and $y = \psi \varphi  (x) $.  For every $q\in Q$, we have
$$c_y (q) = y qy^{-1} =\psi \varphi (x) \cdot q \cdot \psi \varphi (x ^{-1} ) =  
\psi \varphi (x \cdot \beta (q) \cdot x^{-1} )=\beta^{-1} (c_x ( \beta (q) ))=(\beta ^{-1} c_x \beta ) (q).$$
Hence $c_y=\beta^{-1} c_x \beta$. Since $x\in P \leq N_S ^K (Q)$, we have $c_x \in K$, hence $c_y \in K$. 
Note that $y= \psi \varphi (x) \in S$, hence we can conclude that  $y \in N_S ^K (Q)$. This gives
$\psi  \varphi (P) \leq N_S ^K (Q)$.  Since $\psi  \varphi |_Q : Q \to Q$ 
is equal to $\beta^{-1} \in K$, we conclude that $\psi \varphi |_P : P \to S$ is a morphism in $N_{\cF} ^K (Q)$. 

Note that we have $$C_S(\varphi (QP) )\leq C_S(\varphi Q) \leq N^{\leftexp{\varphi} K } _S(\varphi Q),$$  
so we can apply $\psi$ to obtain $\psi (C_S( \varphi (QP) ) \leq N_S ^K (Q)$.
We also have 
$$\psi (C_S (\varphi (QP) )) \leq C_S (\psi \varphi (QP))= C_S (Q \cdot \psi \varphi(P)) \leq C_S( \psi \varphi (P) ).$$ 
These two inclusions give
$$\psi (C_S (\varphi (QP) )) \leq C_S (\psi \varphi (P) ) \cap N_S^K (Q) =C_{N_S^K (Q)} (\psi \varphi (P) ) \leq \psi \varphi (P)$$
where the last inclusion follows from the fact that $P$ is $N_{\cF} ^K (Q)$-centric and 
$\psi \varphi |_P$ is a morphism in $N_{\cF } ^K (Q)$. We conclude that 
$$C_S(\varphi (QP))\leq \varphi (P) \leq \varphi(QP),$$ hence $QP$ is $\cF$-centric.
\end{proof}

As an easy corollary of Lemma \ref{lem:NormCentric}, we obtain the following.

\begin{lemma}[{\cite[Prop 2.4]{BLO2}}]\label{lem:CentricCentric} 
Let $\cF$ be a saturated fusion system over $S$, and $Q$ be a fully $\cF$-centralized subgroup 
of $S$. Suppose that $Q$ is abelian. Then a subgroup $P \leq C_S(Q)$ is $C_{\cF} (Q)$-centric 
if and only if it is $\cF$-centric. If one of these equivalent conditions hold, then $Q\leq P$.
\end{lemma}

\begin{proof} Take $K=1$ in Lemma \ref{lem:NormCentric}. Then $N_S ^K (Q)=C_S(Q)$ and 
$N_{\cF}^K (Q)=C_{\cF} (Q)$. Let $P \leq C_S(Q)$. By Lemma \ref{lem:NormCentric}, if $P$ is 
$\cF$-centric then it is $C_{\cF} (Q)$-centric. For the converse, assume that $P$ is $C_{\cF } (Q)$-centric. 
Then  by Lemma \ref{lem:NormCentric}, $QP$ is $\cF$-centric. Since $Q$ is abelian, we have 
$Q \leq C_S(Q)$. We also have $Q \leq C_S(P)$, hence $Q \leq C_S(P) \cap C_S(Q)= C_{C_S(Q) } (P)$. 
Since $P$ is $C_{\cF} (Q)$-centric, we have $C_{C_S(Q)} (P) \leq P$, hence $Q \leq P$. Thus $P=PQ$ 
is $\cF$-centric. 

If one of these equivalent conditions holds, then $P$ is $\cF$-centric. Since $P \leq C_S(Q)$, 
we obtain $Q\leq C_S(P) \leq P$. 
\end{proof}

A subgroup $Q\leq S$ is called \emph{normal} in $\cF$ if $N_{\cF} (Q)=\cF$. In this case 
we write $Q \lhd \cF$. We say  $Q\leq S$ is \emph{central} in $\cF$ if $C_{\cF} (Q)=\cF$.  
If a subgroup $Q$ is central in $\cF$, then it is also normal in $\cF$. The product of all central 
subgroups in $\cF$ is called the \emph{center} of $\cF$ and denoted 
by $Z(\cF)$. The product of all normal subgroups of $\cF$ is denoted by $O_p (\cF)$.

A subgroup $P\leq S$ is \emph{$\cF$-radical} if $\Out _{\cF} (P):= \Aut _{\cF } (P)/ \Inn (P)$ has 
no normal $p$-subgroups. A subgroup $P\leq S$ is \emph{$\cF$-centric-radical}  if it is
both $\cF$-radical and $\cF$-centric. The full subcategories of $\cF$ generated by subgroups 
which are $\cF$-centric, $\cF$-radical, and $\cF$-centric-radical are denoted by $\cF^c$, $\cF^r$, 
and $\cF^{cr}$, respectively. We recall the following well-known fact on normal subgroups of a fusion system.

\begin{proposition}\label{pro:CravenCentRad} 
Let $\cF$ be a saturated fusion system over $S$, and let $Q$ be a subgroup of $S$.
If $Q$ is normal in $\cF$, then it is included in every $\cF$-centric-radical subgroup of $S$. 
\end{proposition}

\begin{proof} See Proposition 4.46 in \cite{Craven-Book}.
\end{proof}

As a consequence we obtain the following.

\begin{lemma}\label{lem:CentricRadical} Let $\cF$ be a saturated fusion system over $S$. Let
$Q \leq S$ and $K \leq \Aut(Q)$. Suppose that $Q$ is fully $K$-normalized, and assume 
that $N_S ^K (Q)$ contains some $Q_0$ such that $Q_0$ is normal in $N_\cF ^K (Q)$ 
and $Q_0 \in \cF^c$. Then, if $P \leq N_S^K (Q)$ is an $N_{\cF}^K (Q)$-centric-radical 
subgroup, then $P$ is  $\cF$-centric. 
\end{lemma}

\begin{proof} Suppose that $P \leq N_S ^K (Q)$ is $N_{\cF} ^K (Q)$-centric-radical. Since
$Q_0$ is normal in $N_{\cF} ^K (Q)$, by Proposition \ref{pro:CravenCentRad} applied to the 
fusion system $N_{\cF} ^K (Q)$, we obtain $Q_0 \leq P$. Since $Q_0$ is $\cF$-centric,  
$P$ is $\cF$-centric. 
\end{proof}


\section{The centralizer decomposition for  $\cO^c (\cF)$}\label{sect:CentralizerDec}

Let $\cF$ be a saturated fusion system over $S$, and let $\cF^e$ denote the full 
subcategory of $\cF$ whose objects are the nontrivial elementary abelian $p$-subgroups 
of $S$ which are fully $\cF$-centralized. 
By Theorem \ref{thm:realization}, there is a discrete group $G$ with a Sylow subgroup 
isomorphic to $S$ such that $\cF_S (G) \cong \cF$. By Lemmas \ref{lem:CatEquivalence} 
and \ref{lem:CentricG}, if we take $\cC$ to be the collection of all $p$-centric subgroups 
in $G$, then there is an equivalence of categories $\overline \cF _{\cC} (G) \simeq \cO ^c (\cF ) $. 
Let $\cE$ denote the collection of all finite nontrivial elementary abelian $p$-subgroups of $G$.
 
\begin{lemma}\label{lem:CentSame} Let $G$ be a discrete group that realizes the fusion system $\cF$, and
$\cC$ be the collection of $p$-centric subgroups in $G$.
For every $E \in \cF^e$, the fusion orbit category 
$\overline \cF _{\cC} (C_G (E))$ over the collection $\cC |_{C_G(E)}$ is equivalent 
the centric orbit category $\cO ^c ( C_{\cF} (E))$. 
\end{lemma}

\begin{proof} 
By Proposition \ref{pro:FusionSame}, we have $\cF _{C_S( E) } (C_G (E) )=C_{\cF } (E)$. 
This gives that $\cF _{\cC} (C_G (E))$ is equivalent to the full subcategory of the fusion system 
$C_{\cF} (E)$ whose objects are subgroups of $C_S(E)$ which are $\cF$-centric. 
By Lemma \ref{lem:CentricCentric}, a subgroup $P \leq C_S(E)$ is $\cF$-centric if and only if 
it is $C_{\cF } (E)$-centric. Hence $\cF _{\cC} (C_G(E))$ is equivalent 
to the fusion system $C_{\cF} (E)^c$ as categories. We conclude that 
$\overline \cF _{\cC} (C_G(E))$ is equivalent 
to the centric orbit category $\cO ^c ( C_{\cF} (E))$. 
\end{proof}

We introduce the following notation.

\begin{definition}\label{def:MBar} Let $\theta : \overline \cF _{\cC} (G) \to\cO^c (\cF)$ denote the 
functor that gives the equivalence of categories proved in Lemma \ref{lem:CatEquivalence}. 
For every $R\cO ^c (\cF)$-module $M$, we denote the $R\overline \cF _{\cC} (G)$-module 
$M\circ \theta$ by $\overline M$.
\end{definition}

We have the following.

\begin{lemma}\label{lem:HCF}
Let $M$ be an $R\cO ^c (\cF)$-module. For every integer $j\geq 0$, there is 
a functor $$\cH ^j _{M, C_{\cF}} : \cF ^e  \to R\text{-Mod}$$ such  that for every $E \in \cF^e$, 
$$\cH ^j _{M, C_\cF} (E)= H^j (\cO ^c (C_\cF (E)) ; \Res ^{\cO ^c (\cF)} _{\cO ^c  (C_\cF (E) )} M ).$$
\end{lemma}

\begin{proof} Let $G$ be a discrete group with Sylow $p$-subgroup $S$ such that $\cF \cong \cF_S (G)$, 
and let $\cC$ be the collection of all $p$-centric subgroups in $G$. Consider the functor 
$\cH ^j _{M, C_G} : \cF_{\cD} (G) \to R\text{-Mod}$ defined in Definition \ref{def:HCG} 
for an arbitrary collection $\cD$. Take $\cD=\cE$ and $\bO _{\cC} (C_G(E))=\overline \cF _{\cC} (C_G(E))$. 
Then for every $j \geq 0$, we obtain a functor $$\cH^j _{M, C_G} : \cF_{\cE} (G) \to R\text{-Mod}$$ 
such  that for every $E \in \cE$, 
$$
\cH ^j _{M, C_G} (E)= H^* (\overline \cF_{\cC} (C_G (E)) ; \Res ^{ \overline \cF_{\cC} (G)} _{\overline \cF _{\cC} (C_G (E) )}  \overline M )
$$
where $\overline M$ is the $R\overline \cF _{\cC} (G)$-module associated to $M$ as defined in Definition \ref{def:MBar}.
By Lemma \ref{lem:CentSame}, for every $E \in \cF^e$, 
the fusion orbit category 
$\overline \cF _{\cC} (C_G (E))$ over the collection $\cC |_{C_G(E)}$ is equivalent 
the centric orbit category $\cO ^c ( C_{\cF} (E))$.  Hence by
Proposition \ref{pro:EquivCat}, there is an isomorphism
$$H^* (\overline \cF _{\cC} (C_G (E)) ; \Res ^{\overline \cF  _{\cC} (G)} _{\overline \cF _{\cC} (C_G (E) )} 
\overline M )\cong H^* (\cO ^c (C_\cF (E)) ; \Res ^{\cO ^c (\cF)} _{\cO ^c  (C_\cF (E) )} M )$$
induced by the equivalence of categories. 
Using these isomorphisms, as it is done in the proof of Lemma \ref{lem:HMj},
we obtain a functor $\cH ^j _{M, C_{\cF} }: \cF ^e \to R\text{-Mod}$ such that
$$\cH ^j _{M, C_{\cF}} (E) = H^* (\cO ^c (C_\cF (E)) ; \Res ^{\cO ^c (\cF) } _{\cO ^c  (C_{\cF} (E) )} M ).$$  
\end{proof}
 
The following is stated as Theorem \ref{thm:IntroCentDecSS} in the introduction.

\begin{theorem}\label{thm:CentDecSS}  Let $R$ be a commutative ring with unity, $M$ 
be an $R\cO ^c (\cF)$-module, and $\cH ^j _{M, C_{\cF}} $ denote the 
functor defined in Lemma \ref{lem:HCF}. Then there is a spectral sequence 
$$E_2 ^{s,t} = \underset{\cF^e \ }{\lim {}^s} \ \cH ^t _{M , C_{\cF}} \Rightarrow 
H^* (\cO ^c (\cF ) ; M).$$
\end{theorem}

\begin{proof}
Let $G$ be a discrete group that realizes $\cF$, and
$\cC$ be the collection of all $p$-centric subgroups in $G$.
Let $\overline M$ denote the $\overline \cF _{\cC} (G)$-module that corresponds 
to $M$ as defined in Definition \ref{def:MBar}.
By Proposition \ref{pro:CentDecOrb},  there is a spectral sequence 
$$E_2 ^{s,t} = \underset{\cF _{\cE} (G)}{\lim {}^s} \ \cH^t _{\overline M, C_G} 
\Rightarrow H^{s+t} (\overline \cF _{\cC} (G) ; \overline M).$$ 
By Lemmas \ref{lem:CatEquivalence} and 
\ref{lem:CentSame}, we have
the equivalence of categories $\overline \cF _{\cC} (G) \cong \cO ^c (\cF)$ 
and $\overline \cF _{\cC} (C_G(E) ) \cong \cO ^c ( C_{\cF } (E))$ for every $E \in \cF^e$. 
Since every elementary abelian $p$-subgroup in $\cE$ is 
$G$-conjugate to a subgroup $E \in \cF^e$, we also have $\cF _{\cE} (G) \cong \cF ^e$. Hence 
by Proposition \ref{pro:EquivCat}, we can replace 
$\cF _{\cE} (G)$ with $\cF^e$ and $\cH ^t _{\overline M, C_G}$ with $\cH ^t _{M, C_\cF }$.
\end{proof}
 
In \cite{BLO2}, Broto, Levi, and Oliver introduced the centralizer decomposition for $p$-local 
finite groups. Let $(S, \cF, \cL)$ be a $p$-local finite group.  This means that $\cF$ is a saturated 
fusion system over $S$, and $\cL$ is a centric linking system associated to $\cF$ 
(see Definition \ref{def:LinkingSystem}).
 
\begin{definition}[{\cite[Def 2.4]{BLO2}}]\label{def:CentLinking} 
For each $E\in \cF^e$, the centralizer linking system $C_{\cL} (E)$ is the category 
whose objects are the $C_{\cF} (E)$-centric subgroups $P \leq  C_S (E)$, and whose morphisms
$\Mor _{C_{\cL} (E)} (P,P') $ are the set of morphisms $\varphi  \in \Mor_{\cL} (PE ,P' E)$ 
whose underlying homomorphisms are the identity on $E$ and send $P$ into $P'$.
\end{definition}

By \cite[Prop 2.5]{BLO2}, the category $C_{\cL} (E)$ is the centric linking system associated 
to $C_{\cF} (E)$. For every $E\in \cF ^e$, let $\widetilde C_{\cL} (E)$ denote the category whose 
objects are the pairs $(P, \alpha)$ with $P\in \cF^c$ and $\alpha \in \Mor_{\cF} (E, Z(P) )$. 
The morphisms in $\widetilde C_{\cL} (E)$ are defined by
$$\Mor _{\widetilde C_{\cL} (E) } ((P, \alpha), (Q, \beta) ) :=\{ \varphi \in \Mor_{\cL } (P, Q) \ | 
\ \pi (\varphi ) \circ \alpha =\beta\}$$ where $\pi : \cL \to \cF^c$ denotes the canonical projection 
functor.  There is a natural map 
$$ f: \hocolim _{E \in (\cF ^e )^{op}} |\widetilde C_{\cL } (E) | \maprt{} |\cL |$$
induced by the forgetful functors $(P, \alpha ) \to P$. It is proved by Broto, Levi, and Oliver 
\cite[Thm 2.6]{BLO2} that $f$ is a homotopy equivalence, and that there is a homotopy 
equivalence $|C_{\cL } (E)| \to | \widetilde C_{\cL } (E) |$ induced by the functor $P \to (P, incl) $.
This homotopy equivalence given by $f$ is called \emph{the centralizer decomposition} 
for $(S, \cF, \cL )$ over the collection of all nontrivial elementary abelian $p$-subgroups of $S$.
 
Associated to the centralizer decomposition, there is a Bousfield-Kan spectral sequence 
$$E_2 ^{s,t } = \underset{\cF^e\ }{\lim {}^s} \ H^t ( |C_{\cL } (-)|; \bbF _p ) \Rightarrow 
H^{s+t} (|\cL |; \bbF_p ).$$ We say the centralizer decomposition is \emph{sharp} if $E_2 ^{s, t}=0$ 
for $s\geq 1$. When the centralizer decomposition is sharp, then there is an isomorphism 
$$ H^* (|\cL | ; \bbF_p) \cong \underset{E \in \cF ^e}{\lim} H^* (|C_{\cL } (E) | ; \bbF_p ).$$ 
Broto, Levi, and Oliver \cite{BLO2} proved that  the centralizer decomposition for every 
$p$-local finite group is sharp.  We can also state their result  in terms of the cohomology 
of the centralizer fusion systems.
 
\begin{theorem}[Broto-Levi-Oliver \cite{BLO2}, p. 821]\label{thm:SharpCent}
Let $\cF$ be a saturated fusion system over $S$. Then, for every $n \geq 0$, 
and every $i \geq 1$, $$\underset{\cF ^e\ }{\lim {}^i} \ H^n ( C_{\cF} (-); \bbF_p ) =0.$$  
\end{theorem}

\begin{remark}\label{rem:Stable}
In the above theorem $H^n (C_{\cF} (-); \bbF_p)$ denotes the functor which sends
$E \in \cF^e$ to 
$$ H^n (C_{\cF} (E); \bbF_p) := \underset{P \in \cO ^c (C_{\cF} ( E) ) }{\lim} H^n ( P; \bbF_p).$$
Note that $H^n (C_{\cF} (E); \bbF_p)$ is the subring of $C_{\cF} (E)$-stable elements 
in $H^n (C_S (E) ; \bbF_p)$.
\end{remark}

We now recall the definition of the sharpness of the subgroup decomposition.
For every $p$-local finite group 
$(S, \cF, \cL)$, there is a homotopy equivalence 
$$|\cL| \simeq  \hocolim _{\cO ^c (\cF) } \widetilde {B} $$
where $\widetilde B : \cO ^c (\cF) \to \mathrm{Top}$ is a functor such that ${\widetilde B} (P)$ is 
homotopy equivalent to the classifying space $BP$ for every $P \in \cF^c$ (see \cite[Prop 2.2]{BLO2}). 
This homotopy equivalence is called the subgroup decomposition for the 
$p$-local finite group $(S, \cF, \cL)$ over $\cF$-centric subgroups. The subgroup 
decomposition for $(S, \cF, \cL)$ is sharp if for every $n\geq 0$ and every 
$i \geq 1$, $$ \underset{ \cO ^c (\cF )} {\lim {}^i} \ H^n (-; \bbF _p )=0.$$

Now we are ready to prove the main theorem of this section (stated as Theorem  \ref{thm:IntroCentReduction} 
in the introduction). Recall that a subgroup  $Q\leq S$ is \emph{central} in $\cF$ if $C_{\cF} (Q)=\cF$. 
The product of all central subgroups in $\cF$ is called the center of $\cF$ and denoted by $Z(\cF)$.

\begin{theorem}\label{thm:CentReduction}
If the subgroup decomposition is sharp for every $p$-local finite group $(S, \cF, \cL)$ 
with $Z(\cF)\neq 1$, then it is sharp for every $p$-local finite group.
\end{theorem}

\begin{proof}
Let $n\geq 0$ be a fixed integer, and let $M:=H^n (-; \bbF_p)$ denote the group cohomology 
functor considered as a module over $\cO^c (\cF)$. By Theorem \ref{thm:CentDecSS}, 
there is a spectral sequence
$$E_2 ^{s,t} = \underset{\cF^e\ }{\lim {}^s} \ \cH ^t _{M, C_\cF}
\Rightarrow H^{s+t} ( \cO ^c (\cF ) ;  M)$$
where for each $E\in \cF^e$, we have
$$\cH ^t _{M, C_\cF} (E)= H^t  (\cO ^c (C_\cF (E)) ; \Res ^{\cO ^c (\cF)} _{\cO ^c  (C_\cF (E) )} M ). $$
For every $E\in \cF^e$, the centralizer fusion system $C_{\cF} (E)$ has a nontrivial center, 
hence by assumption, the subgroup decomposition for $(C_S(E), C_{\cF} (E), C_{\cL } (E))$ is sharp. 
This gives that
 for every $E\in \cF^e$, 
$$\cH ^t _{M, C_\cF} (E)= H^t  (\cO ^c (C_\cF (E)) ; \Res ^{\cO ^c (\cF)} _{\cO ^c  (C_\cF (E) )} M )= 
H^t  (\cO ^c (C_\cF (E)) ; H^n (-; \bbF_p))=0$$
for $t \geq 1$. Therefore $E_2 ^{s,t} =0$ for all $t\geq 1$. By Remark \ref{rem:Stable}, we have
$$\cH ^0 _{M, C_\cF } ( E)= H^0 ( \cO ^c (C_{\cF } (E) ; \Res ^{\cO ^c (\cF)} _{\cO ^c  (C_\cF (E) )} M)
\cong \lim _{\cO ^c (C_\cF (E) ) } H^n (-; \bbF_p) = H^n (C_{\cF } (E); \bbF_p).$$
Hence we have $$E_2 ^{s, 0} =\underset{\cF^e\ } {\lim {}^s} \  \cH ^0 _{M, C_\cF}  
\cong \underset{\cF^e\ }{\lim {}^s }  \ H^n ( C_{\cF} (-); \bbF_p).$$ 
By Theorem \ref{thm:SharpCent}, these higher limits vanish for $s \geq 1$, hence 
$H^i ( \cO ^c (\cF ); M)=0$ for all $i \geq 1$.
\end{proof}
  
Another consequence of the spectral sequence argument given above is the following: Given 
a $p$-local group $(S, \cF, \cL)$, if we know that 
\begin{enumerate} 
\item the subgroup decomposition for $(S, \cF, \cL)$ is sharp, and 
\item for every $E\in \cF^e$, the subgroup decomposition for $(C_S(E), C_{\cF} (E), C_{\cL } (L))$ 
is sharp,
\end{enumerate}
then we can conclude that the centralizer decomposition for $(S, \cF, \cL)$ over the nontrivial elementary 
abelian $p$-subgroups is sharp. 
We can consider this as a propagation result between two different types of homology decompositions 
in the sense given by Grodal and Smith in \cite{GrodalSmith}.


\section{The normalizer decomposition for  $\cO^c (\cF)$}\label{sect:NormalizerDec}

Let $\cF$ be a saturated fusion system over $S$, and let $sd(\cF^c)$ denote the poset 
category of all chains $\sigma =(P_0 < P_1 < \cdots < P_n)$  with $P_i \in \cF^c$ 
where there is a unique morphism $\sigma \to \tau$ in $sd(\cF^c)$ if $\tau$ is a subchain 
of $\sigma$. Two chains $\sigma =(P_0 < P_1 < \cdots < P_n)$ and
$\tau =(Q_0 < Q_1 < \cdots < Q_n)$ are $\cF$-conjugate if there is an isomorphism 
$\varphi : P_n \to Q_n$ in $\cF$ such that $\varphi (P_i)=Q_i$ for all $i$. 
In this case we write  $\sigma \sim _{\cF} \tau$. The $\cF$-conjugacy class of a chain 
$\sigma$ in $sd(\cF^c)$ is denoted by $[\sigma]$. 

\begin{definition}\label{def:CatConjClasses} The category of \emph{$\cF$-conjugacy 
classes of chains in $\cF ^c$}, denoted by ${\overline s}d (\cF ^c)$, is the poset category 
whose objects are $\cF$-conjugacy classes $[\sigma]$  of chains of subgroups 
in $\cF^c$ where there is a unique morphism 
$[\sigma] \to [\tau]$ in ${\overline s}d (\cF ^c)$  if $\tau$ is $\cF$-conjugate to a subchain 
of $\sigma$. 
\end{definition}

Let $G$ be a discrete group with a Sylow $p$-subgroup $S$ such that
$\cF \cong \cF _S (G)$. If $\cC$ is the collection of all $p$-centric subgroups in $G$, then 
by Lemma \ref{lem:CentricG} we have 
$\cF^c \cong \cF_{\cC} (G)$. We defined the category $Sd_{\cC} (G)/G$ 
to be the category whose objects are the $G$-orbits 
$[\sigma]$ of the chains of subgroups in $\cC$ and such that there is a unique morphism 
$[\sigma] \to [\tau]$ in $Sd_{\cC} (G)/G$ if there is an element $g\in G$ such that $\tau $ 
is a face of $g \sigma$ (see Definition \ref{def:OrbitSimplices}). It is easy to see that since 
$\cF ^c \cong \cF_{\cC} (G)$, the poset categories $Sd_{\cC} (G)/G$ and 
$\overline{s}d(\cF ^c )$ are isomorphic.

If  $\sigma=(P_0 < \cdots < P_n)$ is a chain of subgroups of a discrete group $G$, 
then the stabilizer of $\sigma$ under the $G$-action defined by conjugation  is the 
subgroup $N_G(\sigma)=\bigcap _i N_G(P_i)$. Note that $C_G (P_n) \leq N_G (\sigma)$ 
is a normal subgroup and the quotient group $N_G(\sigma) /C_G(P_n)$ is a subgroup 
of $N_G(P_n) /C_G(P_n) =\Aut _G (P_n) \leq \Aut (P_n)$. We define the $G$-automorphism 
group of $\sigma$ to be the subgroup  
$$\Aut _G (\sigma) :=N_G(\sigma )/C_G (P_n)\leq \Aut(P_n).$$
For a chain of $p$-groups $\sigma =(P_0< \cdots < P_n)$ in $S$, the automorphism group 
of $\sigma$ is defined by
$$\Aut(\sigma):=\{\alpha \in \Aut (P_n ) \, |\, \alpha (P_i)=P_i \text{ for all } i\} \leq \Aut(P_n).$$

In Section \ref{sect:CentNormFusion}, for every subgroup $K \leq \Aut (Q)$, we defined 
the $K$-normalizer group $N_G^K (Q)$ and the $K$-normalizer fusion system $N_{\cF} ^K (Q)$ 
(see Definition \ref{def:GenNormalizer}). It is easy to see that if we take 
$K=\Aut (\sigma) \leq \Aut (P_n)$, then $N_G^K (P_n)=N_G(\sigma)$.  
For the $K$-normalizer fusion system $N_{\cF} ^K (P_n)$, we have the following alternative description.

\begin{definition}\label{def:NormalizerFusSys}
Let $\cF$ be a saturated fusion system over $S$, and let $\sigma=(P_0 < \cdots < P_n)$
be a chain of subgroups in $S$. The \emph{normalizer fusion system} $N_{\cF} (\sigma)$ is 
the fusion system over $N_S(\sigma)$ whose morphisms $\Mor _{N_{\cF} (\sigma) } (P, R)$ 
are all the morphisms $\varphi \in \Mor _{\cF} (P, R)$ which extend  to a morphism 
$\widetilde \varphi : P_nP\to P_nR$ in $\cF$ in such a way that 
$\widetilde \varphi (P_i) =P_i$ for all $i\in \{ 1, \cdots, n\}$.
\end{definition}

We say a chain $\sigma=(P_0 < \cdots < P_n)$ in $S$ is \emph{fully $\cF$-normalized} 
if $|N_S(\sigma) | \geq |N_S (\tau) |$ for every $\tau\sim _{\cF} \sigma$. Note that this 
is equivalent to saying that $P_n$ is fully $K$-normalized
with $K=\Aut(\sigma) \leq \Aut(P_n)$. Using the results of 
Section \ref{sect:CentNormFusion} we conclude the following proposition which 
was also proved in \cite[Prop 3.9]{Libman-WebbConj} using a different approach.

\begin{proposition}\label{pro:NormTheSame} Let $\cF$ be a saturated fusion system 
over $S$, and let $\sigma=(P_0 < \cdots < P_n)$ be a chain of subgroups in $S$ 
which is fully $\cF$-normalized. Then
\begin{enumerate}
\item the normalizer fusion system $N_{\cF} (\sigma)$ is saturated, and 
\item if $\cF=\cF_S(G)$ for a discrete group $G$ with a Sylow $p$-subgroup $S$, then 
$N_S(\sigma)$ is a Sylow $p$-subgroup of 
$N_G(\sigma)$ and  $\cF _{N_S(\sigma) } (N_G(\sigma) )= N_{\cF} (\sigma)$.
\end{enumerate}
\end{proposition}

\begin{proof} The first statement follows from Theorem \ref{thm:Puig} and the second 
part follows from Propositions \ref{pro:SylowExists} and \ref{pro:FusionSame}.
\end{proof}

Let $\sigma$ be a chain of subgroups in $\cF^c$ which is fully $\cF$-normalized.
Let $\cN:= N_{\cF} (\sigma)$, and let $\cN^c$ denote the full subcategory of $\cN$ 
generated by $\cN$-centric subgroups in $N_S (\sigma)$.
We denote by $\cN^{cn}$ the full subcategory of $\cN$ generated by subgroups 
of $N_S(\sigma)$ which are $\cF$-centric. Let $\cN ^{cr}$ denote the full subcategory 
of $\cN$ generated by $\cN$-centric-radical subgroups of $N_S(\sigma)$.   

\begin{lemma}\label{lem:Inclusions} Let $\sigma=(P_0< \cdots < P_n)$ be a chain of subgroups in $\cF^c$.
Suppose that $\sigma$  is fully $\cF$-normalized.
Let $\cN:= N_{\cF} (\sigma)$, and $\cN ^{cr}$, $\cN^{cn}$, $\cN ^c$ be the subcategories defined above. Then 
$$ \cN ^{cr} \subseteq \cN ^{cn} \subseteq \cN ^c.$$
\end{lemma} 

\begin{proof} Note that $N_S (\sigma)$ contains $P_0 \in \cF^c$ as a subgroup and 
$P_0$ is normal in $N_{\cF } (\sigma)$. Applying Lemma \ref{lem:CentricRadical}, 
we obtain that if $P \leq N_S (\sigma) $ is an $\cN$-centric-radical subgroup, 
 then $P$ is $\cF$-centric. Hence $ \cN ^{cr} \subseteq \cN ^{cn}.$ 
For the second inequality, note that by Lemma \ref{lem:NormCentric}(i),  if $P \leq N_S(\sigma)$ is $\cF$-centric 
then it is $\cN$-centric, hence $\cN ^{cn} $ is a full subcategory of $\cN^c$. 
\end{proof}

For higher limits over different collections of subgroups, we will now prove a proposition.
Although it is not stated in the following way, we think this statement is essentially what was proved in
\cite[Cor 3.6]{BLO2}. 

\begin{proposition}\label{pro:HighLimitReduction} 
Let $\cF$ be a saturated fusion system over $S$, and $\cC$ be the collection of subgroups 
of $S$ closed under taking overgroups. Assume that
$\Ob (\cF^{cr})  \subseteq \cC \subseteq \Ob( \cF ^c)$, and let $\nu : \cO (\cF _{\cC} ) \to \cO ^c (\cF)$ 
denote the inclusion functor for the corresponding orbit categories. Then for every 
$ \bbZ_{(p)} \cO^c (\cF)$-module $M$,  there is an isomorphism 
$$\underset{\cO (\cF _{\cC})}{\lim {}^*}  (M \circ \nu ) \cong \underset{ \cO ^c (\cF) }{\lim {}^*} M.$$
\end{proposition}

\begin{proof} We can filter the module $M$ with atomic functors, i.e. with functors 
that vanish except on a single $\cF$-conjugacy class, and prove the result by induction 
using long exact sequences. Hence it is enough to prove the above isomorphism 
when $M$ is an atomic functor.

Let $Q \in  \cF^c$ be such that $M$ vanishes on subgroups outside of the 
$\cF$-conjugacy class of $Q$. If $Q \in \cC$, then by \cite[Lemma 1.5(a)]{Oliver-Linking}, 
the isomorphism in the proposition holds. So assume that $Q \not \in \cC$. 
In this case the restriction of $M$ to the subcategory $\cO (\cF _{\cC} )$ is zero. 
Therefore in this case we need to show that $\underset{ \cO ^c (\cF) }{\lim {}^*} M=0$.

Let $\Gamma$ be a finite group and $N$ be a $\bbZ_{(p)} G$-module. Let $F_N$ denote 
the $\cO _p (\Gamma)$-module such that $F_N (1)=N $ and $F(P) =0$ for all $p$-subgroups 
$1 \neq P \leq \Gamma$. We define $$ \Lambda ^* (\Gamma ; N) :=\underset{\cO _p
(\Gamma) }{\lim {}^*} F_N.$$  By \cite[Prop 3.2]{BLO2}, there is an isomorphism 
$$\underset{\cO^c (\cF)}{\lim {}^*}  M \cong \Lambda ^* (\mathrm{Out} _{\cF}  (Q); M(Q) ).$$
Since $Q$ is not $\cF$-radical, the group $\mathrm{Out} _{\cF} (Q)$ has a normal $p$-subgroup. 
Then by \cite[Prop 6.1(ii)]{JMO} we have $\Lambda ^* (\mathrm{Out} _{\cF}  (Q); M(Q) )=0$. 
Hence the proof is complete.
\end{proof}

Now, we are ready to prove  the following.

\begin{proposition}\label{pro:CoefEqu}
 Let $\cF$ be a saturated fusion system over $S$, $G$ be a discrete group realizing $\cF$, and $\cC$ be the 
collection of all $p$-centric subgroups in $G$. Let $M$ denote an $\bbZ_{(p)} \cO (\cF)$-module 
and $\overline M$ denote the corresponding $\bbZ_{(p)} \overline \cF (G)$-module. 
For every fully $\cF$-normalized chain of subgroups $\sigma=(P_0 < \cdots < P_n)$ 
with $P_i \in \cF^c$, there is an isomorphism
$$H^* ( \overline \cF _{\cC} (N_G(\sigma)); \Res ^{\overline \cF  (G)} _{\overline \cF _{\cC} (N_G (\sigma ) )} \overline M) 
\cong H^* (\cO ^c ( N_{\cF } (\sigma ) ); \Res ^{\cO (\cF) } _{\cO ^c ( N_{\cF} (\sigma) ) } M).$$ 
\end{proposition}

\begin{proof} Let $\cN =\cN _{\cF} (\sigma)$. By Proposition \ref{pro:NormTheSame}, 
$\cF _{N_S(\sigma) } (N_G(\sigma) )= \cN$, hence the fusion orbit category 
$\overline \cF _{\cC} ( N_G (\sigma ) )$ is equivalent to the orbit category $\cO ^{cn} ( \cN)$. 
This gives an isomorphism
$$H^* ( \overline \cF _{\cC} (N_G(\sigma)); \Res ^{\overline \cF  (G)} _{\overline \cF _{\cC} (N_G (\sigma ) )} \overline M) 
\cong H^* (\cO (\cN ^{cn} ) ; \Res ^{\cO (\cF) } _{\cO ( \cN ^{cn}  ) } M).$$ 
By Lemma \ref{lem:Inclusions}, we have  $\cN ^{cr} \subseteq \cN ^{cn} \subseteq \cN ^c$.
Applying Proposition \ref{pro:HighLimitReduction} to the fusion system $\cN$, we obtain 
the desired isomorphism.
\end{proof}

Using the isomorphism in Proposition \ref{pro:CoefEqu}, we define the following  functor.

\begin{lemma}\label{lem:HNF}
Let $M$ be an $\bbZ _{(p)} \cO (\cF)$-module. For every integer $j\geq 0$, there is 
a functor $$\cH ^j _{M, N_{\cF}} : \overline{s}d (\cF ^c)   \to \bbZ_{(p)}\text{-Mod}$$ such that 
for every $[ \sigma] \in \overline{s}d (\cF ^c )$, 
$$\cH ^j _{M, N_\cF} ( [\sigma])= H^j (\cO ^c (N_\cF (\sigma )) ; \Res ^{\cO (\cF)} _{\cO ^c  (N_\cF (\sigma ) )} M )$$
where $\sigma$ is a representative for $[\sigma]$ such that $\sigma$ is a fully $\cF$-normalized chain of subgroups of $S$.
\end{lemma}

\begin{proof} The proof follows from the  same argument given in the proof of Lemma \ref{lem:HCF}. 
In this case we use the isomorphisms proved in Proposition \ref{pro:CoefEqu}.
\end{proof}

The following is stated as Theorem \ref{thm:IntroNormDecSS} in the introduction.
 
\begin{theorem}\label{thm:NormDecSS}  Let $M$ be an $\bbZ _{(p)} \cO (\cF)$-module, 
and $\cH ^j _{M, N_{\cF}} $ denote the functor defined in Lemma \ref{lem:HNF}. 
Then there is a spectral sequence 
$$E_2 ^{s,t} = \underset{\overline{s}d (\cF^c )}{\lim {}^s} \ \cH ^t _{M , N_{\cF} }   \Rightarrow 
H^* (\cO ^c (\cF ) ;  \Res ^{\cO (\cF)} _{\cO ^c (\cF ) } M).$$
\end{theorem}

\begin{proof}
Let $G$ be a discrete group with a Sylow $p$-subgroup $S$ such that 
$\cF \cong \cF_S (G)$. Let $\cC$ be the collection of all $p$-centric subgroups in $G$.
The collection $\cC$ is closed under taking overgroups, hence it is closed under taking 
products. Let $\overline M$ denote the $\overline \cF (G)$-module corresponding to $M$.
By Proposition \ref{pro:NormalizerSS}, there is a spectral sequence
$$E_2 ^{s, t} = \underset{Sd _{\cC} (G) /G }{\lim {}^s}  \cH ^t _{\overline M, N_G}  
\Rightarrow H^{s+t} (\overline \cF _{\cC} (G)  ; \Res ^{\overline \cF (G) } _{\overline \cF _{\cC} (G)} 
\overline M)$$ where $\cH ^j _{\overline M, N_G} $ is the functor defined in Definition \ref{def:HNG}. 
Since $\cF ^c \cong \cF_{\cC} (G)$, the poset 
categories $Sd_{\cC} (G)/G$ and $\overline{s}d(\cF ^c )$ are isomorphic. In each $\cF$-conjugacy class 
$[\sigma] \in \overline{s}d (\cF ^c )$, we choose the chain $\sigma$ which is fully $\cF$-normalized.
By Proposition \ref{pro:CoefEqu}, there is an isomorphism of functors $$\cH ^t _{\overline M, N_G} 
\cong \cH ^t _{M, N_{\cF} }$$ once we identify $Sd _{\cC} (G)/G$ with $\overline{s}d(\cF)$. We also have 
$\overline \cF _{\cC} (G) \cong \cO ^c (\cF ) $. Hence applying Proposition \ref{pro:EquivCat}, we obtain 
the desired spectral sequence. 
\end{proof}
 
In the rest of the section we consider the normalizer decomposition for $p$-local finite groups introduced 
by Libman \cite{Libman-NormDec}.  A $p$-local finite group $(S, \cF, \cL)$ is a triple where $S$ is a finite 
$p$-group, $\cF$ is a saturated fusion system over $S$, and $\cL$ 
is a centric linking system associated to $\cF$. We recall the definition of a centric linking system.

\begin{definition}[{\cite[Def 3.2]{AshOliver}}]\label{def:LinkingSystem}
Let $\cF$ be a fusion system over the $p$-group $S$. A centric linking system
associated to $\cF$ is the category $\cL$ whose objects are the $\cF$-centric subgroups 
of $S$, together with a pair of functors $$\cT _S ^{\cF^c} (S) \xrightarrow{\epsilon}  \cL \xrightarrow{\pi} \cF^c$$
satisfying the following conditions:
\begin{enumerate}
\item The functor $\epsilon$ is the identity on objects and injective on morphisms, while $\pi$ is the inclusion 
on objects and is surjective on each morphism set.
\item For each $P, Q \in \Ob (\cL)$, $\epsilon _P (P) \leq \Aut _{\cL } (P)$ acts freely on the set 
$\Mor _{\cL } (P, Q)$ by precomposition and $\pi _{P, Q}$ induces a bijection from 
$\Mor _{\cL} (P, Q)/ \epsilon_P (Z(P))$ onto $\Mor _{\cF } (P, Q)$.
\item For each $P, Q \in \Ob (\cL)$ and $g \in \cT _S (P, Q)$, the composite functor $\pi \circ \epsilon$ 
sends $g \in \Mor _{\cT} (P, Q)$ to $c_g \in \Mor _{\cF} (P, Q)$.
\item For $P, Q \in \Ob (\cL)$, $\psi \in \Mor _{\cL} (P, Q)$, and $g \in P$, the following square commutes in $\cL$.
$$\xymatrix{P \ar[r]^{\psi} \ar[d]_{\epsilon _P (g)} & Q \ar[d]^{\epsilon _Q (\pi (\psi ) (g)) } \\ P \ar[r]^{\psi} & Q}$$
\end{enumerate}
\end{definition}

For every $P, Q \in \cF^c$ such that $P\leq Q$, let $i _P ^Q : P \to Q$ in $\cL$ denote 
the image of the inclusion map $P \to Q$ under $\epsilon _P$. For each chain 
$\sigma=(P_0<\dots < P_n)$ of $\cF$-centric subgroups in $S$, let $\Aut _{\cL } (\sigma)$ denote the subgroup of $
\prod _{i=1} ^n  \Aut _{\cL} (P_i)$ formed by tuples $(\alpha _0, \dots, \alpha _n )$ of automorphisms 
$\alpha _i \in \Aut _{\cL } ( P_i)$ such that 
the following diagram commutes
$$\xymatrix{P_0   \ar[r]^{i^{P_1}_{P_0}}  \ar[d]^{\alpha_0}  &  P_1 \ar[r]^{i^{P_2} _{P_1}} \ar[d]^{\alpha_1}  
& \cdots  \ar[r] &  P_{n-1} \ar[d]^{\alpha_{n-1}}     \ar[r]^{i^{P_n} _{P_{n-1}}}  & P_n \ar[d]^{\alpha_n} \\  
P_0   \ar[r]_{i^{P_1}_{P_0}}  &  P_1 \ar[r]_{i^{P_2} _{P_1}} & \cdots  \ar[r] & P_{n-1} \ar[r]_{i ^{P_n} _{P_{n-1}} }  & P_n}$$
(see \cite[Def 1.4]{Libman-NormDec} for details).
 
In $\cL$, every morphism is a monomorphism and an epimorphism in the categorical sense 
(see  \cite[Remark 2.10 and Prop 2.11]{Libman-NormDec}). As a consequence, for each $j=0,\dots,n$, the map 
$$\pi _j: \Aut _{\cL } (\sigma )\to \Aut _{\cL} (P_j )$$ defined by $(\alpha_0, \dots, \alpha_n )\to \alpha_j$ 
is a monomorphism, so the automorphism group $\Aut _{\cL}(\sigma) $ can be considered as a subgroup 
of $\Aut _{\cL} (P_j)$ via the projection map $\pi_j$. This allows us to identify $\cB ( \Aut_{\cL} (\sigma))$ 
with a subcategory of $\cL$. Note that if $\tau$ is a face of $\sigma$, then there is group homomorphism 
$\Aut_{\cL } (\sigma)\to \Aut _{\cL } (\tau) $ defined by restriction.

\begin{theorem}[{\cite[Thm A]{Libman-NormDec}, \cite[Thm 5.36]{AKO}}]\label{thm:LibmanNorm} 
Let $(S, \cF, \cL)$ be a $p$-local finite group. Then there is a functor $\delta : \overline{s}d (\cF ^c ) \to Top$ 
such that the following hold: 
\begin{enumerate}
\item There is a homotopy equivalence 
$$\hocolim _{[\sigma]\in \overline{s}d (\cF^c)} \delta ([\sigma]) \maprt{\simeq} |\cL|.$$ 
\item For each $\sigma \in sd (\cF^c ) $, there is a natural homotopy equivalence 
$$ B\Aut _{\cL } (\sigma ) \maprt{\simeq} \delta ([\sigma]).$$
\item For each $\sigma \in sd (\cF^c)$, the map from $B\Aut _{\cL} (\sigma)$ to $|\cL |$ induced 
by the equivalences in (1) and (2) is equal to the map induced by the inclusion of 
$\cB(\Aut _{\cL } (\sigma))$ into $\cL$.
\end{enumerate}
\end{theorem}

Associated to the normalizer decomposition, there is a Bousfield-Kan spectral sequence 
$$E_2 ^{s,t} =\underset{\overline{s}d (\cF^c)}{\lim {}^s} H^t (\Aut _{\cL } (-) ; \bbF_p ) 
\Rightarrow H^{s+t} ( |\cL | ; \bbF_p).$$

\begin{definition}\label{def:SharpNorm} We say the normalizer decomposition 
for $(S, \cF, \cL)$ is \emph{sharp} if $E_2^{s,t}=0$ for all $s>0$ and $t\geq 0$.  
\end{definition}

We prove below (Theorem \ref{thm:NormSharp}) that 
the sharpness of the normalizer decomposition is equivalent to the sharpness 
of the subgroup decomposition. For this, we need to introduce more definitions 
and prove some lemmas. 

Let $(S, \cF, \cL)$ be a $p$-local finite group and 
$\sigma =(P_0<\cdots< P_n)$ be a fully $\cF$-normalized chain of $\cF$-centric 
subgroups in $S$. The normalizer fusion system $N_{\cF} (\sigma)$ is defined 
in Definition  \ref{def:NormalizerFusSys} as the $K$-normalizer fusion system 
$N_{\cF} ^K (P_n)$ where $K=\Aut (\sigma) \leq \Aut (P_n)$.
By Proposition \ref{pro:NormTheSame}, $N_{\cF} (\sigma)$ is saturated. By 
Lemma \ref{lem:NormCentric}, for every $N_{\cF} (\sigma)$-centric subgroup 
$P$, the subgroup $P_nP$ is $\cF$-centric. 

\begin{definition} 
Let $\sigma$ be a fully $\cF$-normalized chain of subgroups in $\cF^c$ and $\cN=N_{\cF} (\sigma)$. 
The \emph{normalizer linking system} $N_{\cL} (\sigma)$ is the category whose objects 
are the $\cN$-centric subgroups in $N_S(\sigma)$, and 
whose morphisms for $P, P' \in \cN ^c$ are defined by
$$\Mor _{N_{\cL} (\sigma) } (P, P')= \{ \varphi \in \Mor _{\cL } ( P_n P , P_n P')\ |\ \pi (\varphi ) 
(P)\leq P', \pi (\varphi ) (P_i )\leq P_i \text{ for all } i \}.$$  
\end{definition}

We have the following lemma.

\begin{lemma}\label{lem:AssociatedLinkingSystem} Let $(S, \cF, \cL)$ be a $p$-local finite 
group and $\sigma$ be a fully $\cF$-normalized chain of subgroups in $\cF^c$. Then the normalizer 
linking system $N_{\cL} (\sigma)$ is the linking system associated to the normalizer fusion system 
$N_{\cF} (\sigma)$. 
\end{lemma}

\begin{proof} When $\sigma=(P_0)$ is a chain with one subgroup this statement is proved in \cite[Lemma 6.2]{BLO2}.
We modify the argument given there to a chain of subgroups. Let $\cN:=N_{\cF} (\sigma)$ and $\cN^c$ denote the full
subcategory of $\cN$-centric subgroups of $N_S (\sigma)$. The projection functor 
$\pi ^{\sigma} : N_{\cL } (\sigma) \to \cN$ is defined to be the inclusion map 
on objects and it sends a morphism $\varphi \in \Mor _{N_{\cL } (\sigma) } (P, P')$ to $\pi (\varphi ) |_P$. 
Let $$\epsilon ^{\sigma} : \cT _{N_S (\sigma ) } ^{\cN ^c} \to N_{\cL } (\sigma)$$ 
be the functor which is the identity map on objects and such that for $P, P' \in \cN^c$,  
the map $\epsilon ^{\sigma} _{P, P'} : N_{N_S (\sigma ) } (P, P') \to \Mor _{N_{\cL } (\delta) } (P, P')$ is defined
by restricting $\epsilon _{PP_n, P'P_n} : N_{S} (PP_n, P' P_n) \to \Aut _{\cL} (PP_n, P' P_n)$ to the subset
$N_{N_S (\sigma ) } (P, P')$.  It is clear that with these definitions the
condition (1) of Definition \ref{def:LinkingSystem} holds.

The subgroup $PP_n$ acts freely on $\Mor _{\cL} (PP_n, P'P_n) $ via $\epsilon _{PP_n}$. Hence 
$Z(P)$ acts freely on $\Mor _{N_{\cL}(\sigma)} (P, P')$. 
If $\varphi _1, \varphi _2$ are  two morphisms in $\Mor _{N_{\cL} (\sigma) } (P, P') $
such that $\pi ^{\sigma} (\varphi_1)=\pi ^{\sigma} (\varphi _2)$, then $\pi(\varphi _1)|_P=\pi (\varphi _2)|_P$. 
Since $P$ is $\cN$-centric, by \cite[Prop A.8]{BLO2} applied to the fusion system $\cN$, there is some $g \in Z(P)$ such that
$\pi (\varphi _2)=\pi (\varphi ) \circ c_g$. By condition $(2)$, there is an $h \in Z(PP_n)\leq Z(P) $ such that
$$\varphi _2= \varphi _1 \circ \epsilon _{PP_n} (g) \circ \epsilon _{PP_n} (h) =\varphi _1 \circ \epsilon _P ^{\sigma} (gh)$$
holds in $\Mor _{\cL} (PP_n, P'P_n)$. Hence the condition (2) holds for the category $N_{\cL } (\sigma)$.
Conditions (3) and (4) hold for $N_{\cL} (\sigma)$ since they hold for $\cL$. We conclude that
$N_{\cL} (\sigma)$ is the linking system associated to the normalizer fusion system 
$N_{\cF} (\sigma)$.  
\end{proof}

Recall that the cohomology of the normalizer fusion system $\cN=N_{\cF} (\sigma)$ is defined by 
$$ H^n (\cN ; \bbF_p) := \underset{P \in \cO ^c (\cN ) }{\lim} H^n ( P; \bbF_p)$$
which is equal to the subring of $\cN$-stable elements in $H^n (N_S (\sigma) ; \bbF_p)$.

\begin{lemma}\label{lem:trivial} Let $(S, \cF, \cL)$ be a $p$-local finite group 
and  let $\sigma $ be a chain of $\cF$-centric subgroups in $S$ which is fully 
$\cF$-normalized. Let $\cN:=N _{\cF} (\sigma)$ denote the normalizer 
fusion system of $\sigma$. Then for every $n\geq 0$, we have 
$$H^t ( \cO ^{c} ( \cN ) ; H^n(- ; \bbF_p))=0$$ for $t >0$, and for $t=0$, we have 
$$H^0 ( \cO ^{c} ( \cN ) ; H^n(-; \bbF_p)) = H^n ( \cN ; \bbF_p)\cong H^n (\Aut _{\cL } (\sigma) ; \bbF_p ).$$
\end{lemma}

\begin{proof}  Let $\sigma=(P_0 < \cdots < P_n)$. The subgroup $P_0$ is $\cF$-centric  
and normal in $\cN=N_{\cF} (\sigma)$. Hence by \cite[Prop 4.3]{5A1}, there is a finite group 
$H$ with a Sylow $p$-subgroup $T$ such that $\cN \cong \cF_T(H)$. By Theorem \cite[Thm B]{DiazPark}, 
the subgroup decomposition for a $p$-local finite group $(S, \cF, \cL)$ is sharp if the fusion system $\cF$ 
is realized by a finite group. Hence we have $$H^t ( \cO ^{c} ( \cN  ) ; H^n(- ; \bbF_p ))=0$$ for every $t>0$, 
and for $t=0$, we have $$H^0 ( \cO ^{c} ( \cN ) ; H^n(-; \bbF_p)) \cong  \underset{\cO ^c (\cN) }{\lim} 
H^n (-; \bbF_p ) = H^n ( \cN ; \bbF_p).$$ 

For the last isomorphism in the statement of the lemma, we claim that $\cN \cong \cF _{N_S (\sigma)} (\Aut _{\cL} (\sigma))$.
By Lemma \ref{lem:AssociatedLinkingSystem}, $\cN _{\cL } (\sigma)$ is the linking system 
associated to  $\cN=\cN_{\cF} (\sigma)$. Then by \cite[Prop 4.3 (a), (b)]{5A1}, $H \cong \Aut _{N_{\cL}(\sigma)} (P_0)$ and 
$N_{\cL} (\sigma)  \cong \cL _T  ^c (H)$. 
This gives $$H \cong \Aut _{N_{\cL} (\sigma) }(P_0)= \{ \varphi \in 
\Aut _{\cL } (P_n) \, |\,  \pi (\varphi ) (P_i )\leq P_i \text{ for all } i \} \cong \Aut _{\cL } (\sigma).$$
Hence $\cN \cong \cF_{N_S(\sigma) } (\Aut _{L} (\sigma) )$. This gives
$$H^n ( \cN  ; \bbF_p) \cong H^* ( \Aut _{\cL } (\sigma) ; \bbF_p)$$
by the Cartan-Eilenberg theorem for finite groups. Hence the proof is complete.
\end{proof}

Now we prove our main theorem in this section.

\begin{theorem}\label{thm:NormSharpIso} Let $\cF$ be a saturated fusion system over $S$. 
For every $n \geq 0$, and for every $i \geq 0$, there is an isomorphism
$$ \underset{\cO ^c (\cF )}{\lim {}^i}  H^n (-; \bbF _p ) \cong 
\underset{\overline{s}d(\cF^c)}{\lim {}^i} H^n ( \Aut _{\cL } (-)  ; \bbF_p).$$
\end{theorem} 

\begin{proof}
By Theorem  \ref{thm:NormDecSS}, for every $\bbZ _{(p)} \cO (\cF)$-module $M$, 
there is a spectral sequence 
$$E_2 ^{s,t} = \underset{\overline{s}d (\cF^c )}{\lim {}^s} \ \cH ^t _{M , N_{\cF} }   \Rightarrow 
H^* (\cO ^c (\cF ) ;  \Res ^{\cO (\cF)} _{\cO ^c (\cF ) } M)$$
where $\cH ^j _{M, N_{\cF}} $ denotes the 
functor defined in Lemma \ref{lem:HNF}. Fix $n \geq 0$ and take $M=H^n (-;\bbF_p)$.
By Lemma \ref{lem:trivial}, for every $[\sigma] \in \overline{s}d(\cF^c )$ we have
 $$\cH ^t _{M, N_{\cF}} ( [\sigma] )= H^t (\cO ^c (N_{\cF} (\sigma) ) ; H^n (-; \bbF_p ))=0$$ 
 for every $t>0$, and  for $t=0$, we have
$$\cH^0 _{M, N_{\cF}} ([\sigma]) =H^0 ( \cO ^c (N_{\cF} (\sigma) ) ; H^n (-; \bbF_p ))\cong H^n (\Aut _{\cL } (\sigma); \bbF_p).$$
This gives $E_2^{s,t}=0$ for $t>0$, and for $t=0$ we obtain 
$$E_2 ^{s, 0} \cong \underset{\overline{s}d(\cF^c)}{\lim {}^s} H^n( \Aut _{\cL } (-) ; \bbF_p).$$
Hence we conclude that
$$\underset{\cO ^c (\cF)}{\lim {}^i} H^n(-; \bbF_p )  \cong E_2 ^{i,0} \cong 
\underset{\overline{s}d(\cF^c)}{\lim {}^i}  H^n( \Aut _{\cL } (-)  ;\bbF_p  )$$
for every $i \geq 0$.
\end{proof}

As an immediate corollary of Theorem \ref{thm:NormSharpIso}, we obtain the following theorem 
which was stated as Theorem \ref{thm:IntroNormSharp} in the introduction.
 
\begin{theorem}\label{thm:NormSharp} For every $p$-local finite group $(S, \cF, \cL)$, 
the subgroup decomposition is sharp if and only if the normalizer decomposition is sharp.
\end{theorem}

\begin{proof} This follows from Theorem \ref{thm:NormSharpIso} and Definitions
\ref{def:SharpSub} and \ref{def:SharpNorm}.
\end{proof}

This reduces the problem of showing the sharpness of subgroup decomposition to a problem 
of showing the sharpness of the normalizer decomposition. Of course this could be an equivalently 
difficult problem to solve.


\section{Fusion systems realized by finite groups}\label{sect:FiniteGroups}

Let $G$ be a discrete group and $\cC$ be a collection of subgroups of $G$. 
Let  $X_{\cC} ^{\beta}=E_{\cC} G$ denote the Dwyer space for the subgroup decomposition 
as defined in Section \ref{sect:SubDecomp}. 
For every $H \in \cC$, the fixed point subspace $(E_{\cC} G )^H$ is homotopy equivalent  
to the realization of the subposet $\cC_{\geq H}$  
hence it is contractible (see \cite{GrodalSmith}). 
We show below that if $G$ is a finite group and $\cC$ is a collection 
of $p$-subgroups in $G$ closed under taking $p$-overgroups, then for every $P\in \cC$, the orbit 
space $C_G(P)\backslash (E_{\cC} G)^P$ is $\bbZ _{(p)}$-acyclic.
We start with a lemma.
 
\begin{lemma}\label{lem:OrbitSpace} Let $G$ be a finite group and $X$ be a 
$G$-CW-complex such that for every $p$-subgroup $P\leq G$, the fixed point subspace 
$X^P$ is contractible. Then the orbit space $X/G$ is $\bbZ _{(p)}$-acyclic.
\end{lemma}
 
\begin{proof} Let $S\leq G$ be a Sylow $p$-subgroup of $G$. By a transfer argument 
it is enough to show that $X/S$ is $\bbZ_{(p)}$-acyclic (see \cite[\S III.2]{Bredon-Compact}). 
Consider the fixed point map $f: X\to pt$. Since for every $P\leq S$, the map $X^P \to pt$ 
is a homotopy equivalence, we can conclude that $f$ is an $S$-homotopy equivalence 
(see \cite[Prop 2.5.8]{BensonSmith-Book}). This gives that the orbit space $X/S$ is  contractible, 
hence it is $\bbZ_{(p)}$-acyclic.
\end{proof}
 
As a consequence of the above lemma we obtain the following:
 
\begin{proposition}\label{pro:DwyerSubgroup} Let $G$ be a finite group and $\cC$ be a collection 
of $p$-subgroups in $G$ closed under taking $p$-overgroups. Then for every $P\in \cC$, the orbit 
space $C_G(P)\backslash (E_{\cC} G)^P$ is $\bbZ _{(p)}$-acyclic.
\end{proposition} 
 
\begin{proof} Let $X=E_{\cC} G$ and $P\in \cC$. The fixed point set $X^P$ 
has a natural $N_G(P)$-action on it, hence it has an action of $C_G(P)$ via restriction. For every 
$p$-subgroup $Q\leq C_G(P)$, we have $(X^P)^Q = X^{PQ}$. Since $\cC$ is closed under taking 
$p$-overgroups, we have $PQ\in \cC$, hence $X^{PQ}$ is contractible. This allows us to apply 
Lemma \ref{lem:OrbitSpace} and conclude that  $C_G(P)\backslash X^P$ is $\bbZ_{(p)}$-acyclic.
\end{proof}

An interesting consequence of this proposition is the following:

\begin{theorem}\label{thm:TheSame}
Let $G$ be a finite group and $\cC$ be a collection of $p$-subgroups of $G$ closed under taking 
$p$-overgroups. Then for every $\bbZ_{(p)} \overline \cF _{\cC}(G)$-module $M$, 
there is an isomorphism 
$$H^* ( \overline \cF _{\cC} (G) ; M ) \cong H ^* ( \cO _{\cC} (G) ;  \Res_{pr} M) $$
induced by the projection functor  $pr: \cO _{\cC} (G) \to \overline \cF _{\cC} (G) $.
\end{theorem}

\begin{proof} Let $X=E_{\cC} G$ and $R=\bbZ_{(p)}$. The isotropy subgroups of $X$ 
are in $\cC$ and for every $P\in \cC$, the fixed point set $X^P$ is contractible. This implies 
that the chain complex $C_* (X^?; R)$ is a projective resolution of the constant functor 
$\underline R$ as an $R\cO _{\cC} (G)$-module. By Proposition \ref{pro:DwyerSubgroup}, 
the chain complex of $R\overline \cF _{\cC} (G)$-modules $C_* (C_G (?) \backslash X^?; R)$ 
is also $R$-acyclic, hence it gives a projective resolution of 
$\underline R$ as an $R\overline \cF_{\cC} (G)$-module. Combining these we obtain
\begin{align*} 
H^* (\overline \cF _{\cC} (G); M) & \cong H^* (\hom _{R\overline \cF _{\cC} (G) } (C_* (C_G (?) 
\backslash X^?; R) , M ))  \\
& \cong H^* (\hom _{R\cO _{\cC} (G) } (C_* (X^?; R) ,  \Res _{pr} M )) \\
& \cong H^* ( \cO _{\cC} (G); \Res _{pr} M).
\end{align*}
The isomorphism in the middle follows from Lemma \ref{lem:TwoWays}. 
\end{proof}

For a finite group $G$, let $O_p(G)$ denote the (unique) maximal normal subgroup 
in $G$. A $p$-subgroup $P \leq G$ is called a \emph{principal $p$-radical subgroup} 
if it is $p$-centric and $O_p (N_G(P) / P C_G (P) )$ is trivial. The collection of principal 
$p$-radical subgroups is denoted by $\cD_p (G)$. 
Now as a corollary of Theorem \ref{thm:TheSame}, a theorem of Grodal  \cite[Thm 1.2]{Grodal-Higher} 
gives the following.

\begin{corollary}\label{cor:Grodalthm}
Let $G$ be a finite group, and $\cC$ and $\cC'$ be two collections of $p$-subgroups 
of $G$ closed under taking $p$-overgroups satisfying 
$\cC ' \cap D_p (G) \subseteq \cC \subseteq \cC'$. Then for any 
$\bbZ_{(p)}\overline \cF_{\cC'} (G)$-module $M$, there is an isomorphism
$$H^* (\overline \cF _{\cC'} (G); M) \cong H^* (\overline \cF _{\cC} (G); \Res _{\cC} M).$$  
\end{corollary}

\begin{proof} By Theorem \ref{thm:TheSame}, it is enough to prove that 
$$H^* (\cO _{\cC'} (G); \Res _{pr} M)\cong H^* (\cO _{\cC} (G) ;  \Res _{pr} \Res _{\cC} M).$$ 
Since $C_G(P)$ acts trivially on $(\Res _{pr} M) (P)$ for every $P \in \cC$, 
the result follows from \cite[Thm 1.2]{Grodal-Higher}.
\end{proof}

As a consequence we obtain a generalization of the vanishing result due to Diaz and Park 
\cite{DiazPark}.

\begin{theorem}\label{thm:DiazPark}
Let $\cF=\cF_S(G) $ for a finite group $G$ with a Sylow $p$-subgroup $S$. Let $\cC$ 
be a collection of subgroups of $S$ closed under taking $p$-overgroups such that $\cC$ 
includes all $\cF$-centric-radical subgroups in $S$. Then for every $n\geq 0$ and for every $i\geq 1$, 
$$\underset{\cO (\cF _{\cC}) }{\lim {}^i}   H^n (-; \bbF_p) =0.$$ 
\end{theorem}

\begin{proof}
By Lemma \ref{lem:CentricG}, a subgroup $P \leq S$ is $\cF$-centric 
if and only if it is $p$-centric in $G$.
For $P \leq S$, we have $$\Out_{\cF} (P):= \Aut _{\cF } (P)/ \Inn (P) \cong N_G(P)/ P C_G(P).$$
This means $P\leq S$ is $\cF$-radical if and only if $O_p (N_G (P)/PC_G (P))=1$. Hence $P\leq S$ is  
$\cF$-centric-radical if and only if $P$ is a principal $p$-radical subgroup. 

Let $\cC'$ denote the collection of all $p$-subgroups in $G$. We denote the orbit category 
over the collection $\cC'$ by $\cO _p (G)$. By the assumption on $\cC$ we have
$$\cC ' \cap \cD_p(G)=\cD_p(G) \subseteq \cC \subseteq \cC'.$$
Applying Theorem \ref{thm:TheSame} and Corollary \ref{cor:Grodalthm} to these collections, 
we obtain $$\underset{\cO (\cF _{\cC})}{\lim {}^i}   
H^n (-; \bbF_p ) \cong     \underset{\cO (\cF _{\cC'})}{\lim {}^i}   
H^n (-; \bbF_p ) \cong       \underset{\cO _p (G)}{\lim {}^i } H^n (- ; \bbF_p ).$$  
Since $H^n (-; \bbF_p)$ is a cohomological Mackey functor,
by  \cite[Prop 5.14]{Jackowski-McClure-HomDec}, 
$$\underset{\cO _p (G)}{\lim {}^i } H^n (- ; \bbF_p )=0.$$ Hence the proof is complete.
\end{proof}

It is interesting to ask whether Proposition \ref{pro:DwyerSubgroup} and 
Theorem \ref{thm:TheSame} still hold when $G$ is an infinite group. 
The following example illustrates that these results do not hold for infinite groups 
in general.

\begin{example}\label{ex:Sharpness} Let $\cF$ be the fusion system of the symmetric 
group $S_3$ at prime $3$. If we apply the Leary-Stancu construction to $\cF$, 
we find the infinite group
$$G =\langle b, a \, | \, b^3=1, aba^{-1}=b^2 \rangle \cong C_3 \rtimes \bbZ$$
realizing $\cF$ (see \cite[Ex 4.3]{GundoganY}). The subgroup $\langle a^2 \rangle \cong \bbZ$ 
is central in $G$ and $G/ \langle a^2 \rangle$ is isomorphic to  the symmetric group $S_3$.
Hence we have an extension of groups $1 \to \bbZ \to G \to S_3 \to 1$ whose extension 
class in $\bbF_3$-coefficients is zero. This gives that for every $n \geq 0$,
$$H^n(G; \bbF_3) \cong H^n (S_3; \bbF_3) \oplus H^{n-1} (S_3;\bbF_3).$$
Note that $$H^*( S_3 ; \bbF_3) \cong H^* (C_3; \bbF_3 ) ^{C_2} \cong (\Lambda (y) \otimes \bbF_3 [x] ) ^{C_2}$$ 
where $|y|=1$, $|x|=2$, and $C_2$ acts on $x$ and $y$ via $x\to -x$ and $y\to -y$. Hence 
$H^i (S_3; \bbF_3)\cong \bbF_3$ for $i=0, 3$ mod $4$,  and zero otherwise. We have 
$$H^* (\cF ; \bbF_3) =\lim _{P\in \cO(\cF)} H^* (P; \bbF_3)\cong H^* (S_3 ; \bbF_3) \neq H^* (G; \bbF_3),$$
since $H^{n-1} (S_3; \bbF_3) \neq 0$ for $n=0,1$ mod $4$. Hence this is an example 
of the situation where the cohomology of the infinite group $G$ realizing the fusion system 
$\cF$ is not isomorphic to the cohomology of $\cF$ (see \cite{GundoganY}).

Let $\cC$ be the collection of all finite $3$-subgroups in $G$. We can take $X=E_{\cC} G$ 
to be the real line with $G$-action via the homomorphism $G \to \bbZ$.  For any commutative 
ring $R$, the  chain complex $C_* (X^?; R)$ gives a sequence 
$$ 0 \to R[G/C_3^?]\maprt{1-a} R[G/C_3 ^?] \to \underline R \to 0$$ which is a projective 
resolution of $\underline R$ as an $R\cO _{\cC} (G)$-module. For every $R\cO_{\cC} (G)$-module 
$M$, the higher limits $\lim ^* _{\cO_{\cC} (G)} M$ are the cohomology modules of the cochain 
complex $$ 0 \to M(C_3) \maprt{1-a} M(C_3) \to 0.$$ Hence 
$\lim^0 _{\cO_{\cC} (G) }  M \cong \ker (1-a)$ and $\lim ^1 _{\cO_{\cC} (G)} M \cong \mathrm{coker} (1-a)$. 
If we take $M= H^n (-; \bbF_p)$ with $n=0,3$ mod $4$, then the action of the element $a \in G$ 
on $M$ is trivial, hence we have $$\underset{\cO_{\cC} (G)}{\lim {}^i} H^n (-; \bbF_p) \cong \bbF_3 \neq 0$$
for $i=0,1$ and  $n=0,3$ mod $4$. 

The orbit space $C_G (C_3) \backslash X^{C_3} \cong \bbZ \backslash \bbR \cong S^1$ 
is not $\bbZ_{(3) }$-acyclic. This shows that Proposition \ref{pro:DwyerSubgroup} does not 
hold for this group. If we take the orbit spaces of $C_G(P)$-actions on $X^P$ for all $P \in \cC$, 
then we obtain a chain complex
$$ 0 \to H_1 (?) \to R \Mor _{\cF_{\cC} (G) } (?, C_3)  \to R\Mor _{\cF_{\cC} (G) } (?, C_3)  \to \underline R \to 0$$ 
of $R\overline \cF_{\cC} (G)$-modules. Note that $H_1 (P)\cong R$ for $P=1$ and $P=C_3$, 
and the restriction map $H_1 (C_3) \to H_1 (1)$ is given by multiplication by 2. Thus if we take 
$R=\bbZ _{(3)}$, then $H_1 (?)\cong \underline R$, and we can splice the sequences to get a projective 
resolution for  $\underline R$ as an $R\overline \cF_{\cC} (G)$-module. Using this projective resolution, 
we obtain that for every $\bbZ _{(3)} \overline \cF_{\cC} (G)$-module $M$,  
$$\underset{\overline \cF _{\cC} (G) }{\lim {}^i} M \cong H^i (C_2 ; M (C_3))=0$$
for every $i\geq 1$. In particular, $$\lim {}^i _{\overline \cF_{\cC} (G)} H^n (-; \bbF_3)=0$$ 
for all $i\geq 1$. Hence Conjecture \ref{conj:Sharpness} holds for the fusion system 
$\cF=\cF_{C_3} (S_3)$. 

To summarize, for the infinite group $G=C_3 \rtimes \bbZ$ realizing the fusion system 
$\cF=\cF_{C_3} (S_3)$, we conclude that
\begin{enumerate}
\item the statements of Proposition \ref{pro:DwyerSubgroup}  and Theorem \ref{thm:TheSame} 
do not hold for $G$  when we take 
$\cC$ as the collection of all $3$-subgroups in $G$,
\item the mod-$3$ subgroup decomposition for $BG$ over the collection of all finite 
$3$-subgroups is not sharp for  $G$, and
\item Conjecture \ref{conj:Sharpness} holds for the fusion system $\cF$.
\end{enumerate}
\end{example}

\end{document}